\documentclass[a4paper,12pt]{amsart}

\usepackage[margin=2.5cm]{geometry} 
\usepackage{amsthm,amsmath,amssymb}
\usepackage{enumitem}
\usepackage{hyperref}
\hypersetup{
    colorlinks = true, 
    urlcolor   = black,
    linkcolor  = black, 
    citecolor  = black 
}
\usepackage{tikz,tikz-cd}

\theoremstyle{plain}
\newtheorem{theorem}{Theorem}[section]
\newtheorem{prop}[theorem]{Proposition}
\newtheorem{cor}[theorem]{Corollary}
\newtheorem{lemma}[theorem]{Lemma}
\newtheorem{conj}[theorem]{Conjecture}

\theoremstyle{definition}

\newtheorem{ex}[theorem]{Example}
\newtheorem{remark}[theorem]{Remark}
\newtheorem{notation}[theorem]{Notation}

\renewcommand{\a}[0]{\mathbf a}
\newcommand{\C}[0]{\mathbb C}
\newcommand{\Cu}[0]{\mathcal C}
\renewcommand{\d}[0]{\mathbf d}
\newcommand{\E}[0]{\mathcal E}
\newcommand{\I}[0]{\mathcal I}
\renewcommand{\L}[0]{\mathbb L}
\newcommand{\lbeta}[0]{\bar{\beta}}
\newcommand{\N}[0]{\mathbb N}
\renewcommand{\P}[0]{\mathbb P}
\newcommand{\Q}[0]{\mathbb Q}
\newcommand{\x}[0]{\mathbf x}
\newcommand{\bxi}[0]{\mathbf \xi}
\DeclareMathOperator{\lcm}{lcm}

\title[The monodromy conjecture for a space monomial curve]
{The monodromy conjecture for a space monomial curve with a plane semigroup}

\author[J.~Mart\'{\i}n-Morales]{Jorge Mart\'{\i}n-Morales}
\address[J.~Mart\'{\i}n-Morales]{Centro Universitario de la Defensa, IUMA \\
Academia General Militar \\ 
Ctra.~de Huesca s/n. \\ 
50090 Zaragoza, Spain}
\email{jorge@unizar.es}

\author[W.~Veys]{Willem Veys}
\author[L.~Vos]{Lena Vos}
\address[W.~Veys \and L. Vos]{
KU Leuven \\
Departement Wiskunde \\
Celestijnenlaan 200B, bus 2400 \\
3001 Leuven, Belgium}
\email{wim.veys@kuleuven.be, lena.vos@kuleuven.be}

\thanks{The first author is partially supported by MTM2016-76868-C2-2-P from the Departamento de Industria
e Innovaci\'on del Gobierno de Arag\'on and Fondo Social Europeo E22 17R Grupo Consolidado \'Algebra y Geometr\'ia, and by FQM-333 from Junta de Andaluc\'ia. The second author is partially supported by the Research Foundation - Flanders (FWO) project G.0792.18N. The third author is supported by a PhD Fellowship of the Research Foundation - Flanders (no. 71587).}

\begin{document}

\begin{abstract}
This article investigates the monodromy conjecture for a space monomial curve that appears as the special fiber of an equisingular family of curves with a plane branch as generic fiber. Roughly speaking, the monodromy conjecture states that every pole of the motivic, or related, Igusa zeta function induces an eigenvalue of monodromy. As the poles of the motivic zeta function associated with such a space monomial curve have been determined in earlier work, it remains to study the eigenvalues of monodromy. After reducing the problem to the curve seen as a Cartier divisor on a generic embedding surface, we construct an embedded $\Q$-resolution of this pair and use an A'Campo formula in terms of this resolution to compute the zeta function of monodromy. Combining all results, we prove the monodromy conjecture for this class of monomial curves.    
\end{abstract}


\maketitle

\setcounter{tocdepth}{1}
\tableofcontents

\section*{Introduction} \label{Intro}

The classical \emph{monodromy conjecture} predicts a relation between two invariants of a polynomial, one originating from number theory and the other from differential topology. More precisely, it states that the poles of the motivic, or related, Igusa zeta function of a polynomial $f \in \C[x_0, \ldots, x_n]$ induce eigenvalues of the local monodromy action of $f$, seen as a function $f:\C^{n+1}\rightarrow \C$, on the cohomology of its Milnor fiber at some point $x \in f^{-1}(0) \subset \C^{n+1}$. Generalizing the motivic Igusa zeta function to an ideal and using the notion of Verdier monodromy, one can similarly formulate the monodromy conjecture for ideals. To date, both conjectures have only been proven in full generality for polynomials and ideals in two variables, see~\cite{L} and~\cite{VV}, respectively. In higher dimension, various partial results were shown for one polynomial (see for instance the introduction of~\cite{BV} for a list of references), but for multiple polynomials, the most general result so far is a proof for monomial ideals~\cite{HMY}. Very recently, Musta\c{t}\u{a}~\cite{Mu} showed that the monodromy conjecture for polynomials implies the one for general ideals. However, since the monodromy conjecture for one polynomial is still open in more than two variables, this does not provide an immediate solution of the monodromy conjecture for ideals. In the present article, the monodromy conjecture is investigated for a class of binomial ideals in arbitrary dimension that define space curves deforming to plane branches. As the poles of the motivic Igusa zeta function associated with these binomial ideals have already been studied in~\cite{MVV}, we concentrate on the eigenvalues of monodromy. A short summary of the main results of the present article and of~\cite{MVV} can be found in~\cite{MMVV}.

\vspace{14pt}

To construct the ideals of our interest, we start with a germ $\Cu:=\{f=0\}\subset (\C^2,0)$ of a complex plane curve defined by an irreducible series $f\in\C[[x_0,x_1]]$ with $f(0) = 0$. The semigroup $\Gamma(\Cu)$ of $\Cu$ is the image of the associated valuation \[\nu_{\C}: \frac{\C[[x_0,x_1]]}{(f)}\setminus \{0\} \longrightarrow \N: h \mapsto \dim_{\C}\frac{\C[[x_0,x_1]]}{(f,h)}.\] This semigroup is finitely generated and has a unique minimal generating set $(\lbeta_0,\ldots,\lbeta_g)$. Define $Y$ as the image of the monomial map $M:(\C,0) \rightarrow (\C^{g+1},0)$ given by $t\mapsto (t^{\bar{\beta}_0},\ldots,t^{\bar{\beta}_g}).$ This is an irreducible curve which is smooth outside the origin and whose semigroup is the \lq plane\rq\ semigroup $\Gamma(\Cu)$. Furthermore, it is the special fiber of an equisingular family $\eta:(\chi,0) \subset (\C^{g+1}\times \C,0)\rightarrow (\C,0)$ with generic fiber isomorphic to $\Cu$. The ideal $\I \subset \C[x_0,\ldots,x_g]$ defining $Y$ in $\C^{g+1}$ is generated by binomial equations of the form 
	\[\left\{ 
	\begin{array}{r c l l}
	f_1 := x_1^{n_1} & - & x_0^{n_0}  &  = 0 \\
	f_2 := x_2^{n_2}  &- & x_0^{b_{20}}x_1^{b_{21}} &= 0  \\
	& \vdots & &\\
	f_g := x_g^{n_g} &- & x_0^{b_{g0}}x_1^{b_{g1}}\cdots x_{g-1}^{b_{g(g-1)}} & = 0.\\
 	 \end{array}
	\right.\]
Here, $n_i > 1$ and $b_{ij} \geq 0$ are integers that can be expressed in terms of $(\bar{\beta}_0,\ldots, \bar{\beta}_g)$, see~\eqref{eq:ni-betai}. The curve $Y$ is called the \emph{monomial curve associated with $\Cu$}, but, to simplify the notation, we will refer to it as a \emph{(space) monomial curve} $Y \subset \C^{g+1}$. In this article, the case of interest is $g \geq 2$.

\vspace{14pt}

In~\cite{MVV}, it was shown that a complete list of poles of the motivic zeta function associated with a space monomial curve $Y \subset \C^{g+1}$ is given by \[\L^g, \qquad \L^{\frac{\nu_k}{N_k}}, \qquad k = 1,\ldots, g,\] where \[\frac{\nu_k}{N_k} = \frac{1}{n_k\lbeta_k}\bigg(\sum_{l=0}^k \lbeta_l - \sum_{l=1}^{k-1}n_l\lbeta_l\bigg) + (k-1) + \sum_{l=k+1}^g\frac{1}{n_l}.\] Here, $\L$ denotes the class of the affine line in the \emph{Grothendieck ring} of complex varieties.

\vspace{14pt}

It thus remains to investigate the monodromy eigenvalues of a space monomial curve $Y \subset \C^{g+1}$ and to show that every pole in the above list yields such an eigenvalue. To this end, we will make use of the following \emph{A'Campo formula} for the monodromy eigenvalues in terms of a principalization $\varphi: \tilde X \rightarrow \C^{g+1}$ of the ideal $\I$ defining $Y$. Let $E_j$ for $j \in J$ be the irreducible components of $\varphi^{-1}(Y)$, and denote by $N_j$ and $\nu_j - 1$ the multiplicity of $E_j$ in the divisor of $\varphi^{\ast}\I$ and $\varphi^{\ast}(dx_0 \wedge \cdots \wedge dx_g)$, respectively. Let $\sigma: X' \rightarrow \C^{g+1}$ be the blow-up of $\C^{g+1}$ along $Y$ with exceptional divisor $E' := \sigma^{-1}(Y)$. By the universal property of the blow-up, there exists a unique morphism $\psi: \tilde X \rightarrow X'$ such that $\sigma \circ \psi = \varphi$. Then, from~\cite{VV}, a complex number is a monodromy eigenvalue associated with $Y$ if and only if it is a zero or pole of the \emph{zeta function of monodromy} at a point $e \in E'$ given by 
\begin{equation}\label{eq:ACampo-princ}
Z_{Y,e}^{mon} (t)= \prod_{j \in J}(1-t^{N_j})^{\chi(E^{\circ}_j \cap \psi^{-1}(e))},
\end{equation}
where $\chi$ denotes the topological Euler characteristic and $E_j^\circ := E_j \backslash \cup_{i \neq j}(E_i\cap E_j)$ for every $j \in J$. This is a generalization of the original formula of A'Campo~\cite{AC} expressing the monodromy eigenvalues of one polynomial $f \in \C[x_0,\ldots, x_g]$ in terms of an embedded resolution $\varphi: \tilde X \rightarrow \C^{g+1}$ of $\{f = 0\}$, see~\eqref{eq:ACampo-poly}. Both A'Campo formulas can be generalized in a straightforward way to ideals and polynomials, respectively, defining a subscheme $Y$ of a general variety $X$ with $\text{Sing}(X) \subset Y$.

\vspace{14pt}

We will apply formula~\eqref{eq:ACampo-princ} to a specific point in the exceptional divisor $E'$ that we define by means of a \emph{generic embedding surface} of $Y$. For every set $(\lambda_2, \ldots, \lambda_g) $ of $g-1$ non-zero complex numbers, we introduce an affine scheme $S(\lambda_2,\ldots, \lambda_g)$ in $\C^{g+1}$ given by the equations
\[\left\{ 
	\begin{array}{c c l l l}
	f_1& + & \lambda_2f_2  &  = 0 \\
	f_2& + & \lambda_3f_3 &= 0  \\
	& \vdots & &\\
	f_{g-1} & +  & \lambda_gf_g & = 0.\\
    \end{array}
 	\right.\]
Every such scheme contains $Y$ as a Cartier divisor defined by one of the equations $f_i = 0$. For \emph{generic} coefficients $(\lambda_2,\ldots, \lambda_g)$, the scheme $S(\lambda_2,\ldots, \lambda_g)$ is a normal surface which is smooth outside the origin. If we denote by $S'$ the strict transform of such a generic embedding surface $S := S(\lambda_2,\ldots, \lambda_g)$ under the blow-up $\sigma$, then our interest goes to the monodromy zeta function $Z_{Y,p}^{mon}(t)$ at the point $p := S' \cap \sigma^{-1}(0)$. Using the above A'Campo formulas, it turns out that, for generic coefficients, $Z_{Y,p}^{mon}(t)$ is equal to the monodromy zeta function $Z_{Y,0}^{mon}(t)$ of $Y$ considered on $S$ at the origin; this will be shown in Theorem~\ref{thm:red-to-on-curve-surface}. In fact, this result will be stated and proven in a more general context, which makes it possibly useful for other instances of the monodromy conjecture.

\vspace{14pt}

To compute the monodromy zeta function $Z_{Y,0}^{mon}(t)$ of $Y \subset S$ at the origin, we will consider another generalization of A'Campo's formula in terms of an \emph{embedded $\Q$-resolution} of $Y \subset S$ that was proven in~\cite{Ma1}. Roughly speaking, a $\Q$-resolution is a resolution in which the final ambient space is allowed to have abelian quotient singularities, and the zeta function of monodromy at the origin can be written as \[Z^{mon}_{Y,0}(t) = \prod_{\underset{1 \leq l \leq s}{1 \leq j \leq r}} \left(1-t^{m_{j,l}}\right)^{\chi(E_{j,l}^{\circ})},\] where $\{E_{j,l}\}_{j=1,\ldots,r,l = 1,\ldots, s}$ is a finite stratification of the exceptional varieties $E_1,\ldots,E_r$ of the $\Q$-resolution such that the \emph{multiplicity} $m_{j,l}$ of $E_j$ along each $E_{j,l}$ is constant. To construct an embedded $\Q$-resolution of $Y \subset S$, we will compute $g$ \emph{weighted blow-ups}. After each blow-up, we will be able to eliminate one variable so that we obtain a situation very similar to the one we have started with, but with one equation in $Y$ and $S$ less. Therefore, in the last step, the problem will have been reduced to the resolution of a cusp in a Hirzebruch-Jung singularity of type $\frac{1}{d}(1,q)$, which can be solved with a single weighted blow-up. One can compare this process to the resolution of an irreducible plane curve with $g$ Puiseux pairs using toric modifications; after each weighted blow-up, the number of Puiseux pairs is lowered by one, and the last step coincides with the resolution of an irreducible plane curve with one Puiseux pair. Our case, however, will be more challenging as the strict transform of $Y$ after the first blow-up will pass in general through the singular locus of the ambient space. The resulting $\Q$-resolution is described in Theorem~\ref{thm:resolutionY}, and its resolution graph is a tree as in Figure~\ref{fig:dual-graph}. Stratifying the exceptional divisor of the resolution such that the multiplicity is constant along each stratum and computing the Euler characteristics of the strata yields \[Z^{mon}_{Y,0}(t) = \frac{\prod\limits_{k = 0}^g(1-t^{M_k})^{\frac{\lbeta_k}{M_k}}}{\prod\limits_{k = 1}^g(1-t^{N_k})^{\frac{n_k\lbeta_k}{N_k}}},\] where \[M_k := \lcm\Big(\frac{\lbeta_k}{\gcd(\lbeta_0,\ldots, \lbeta_k)},n_{k+1},\ldots, n_g\Big), \quad k = 0,\ldots, g,\] and \[N_k := \lcm\Big(\frac{\lbeta_k}{\gcd(\lbeta_0,\ldots, \lbeta_k)},n_k,\ldots, n_g\Big), \quad k = 1,\ldots, g.\] It follows that the monodromy zeta function $Z_{Y,p}^{mon}(t)$ of $Y \subset \C^{g+1}$ at $p = S' \cap \sigma^{-1}(0)$ is given by the same expression, see Theorem~\ref{thm:zeta-function-mon-Y}. 

\vspace{14pt}

With this expression for $Z_{Y,p}^{mon}(t)$, we will be able to prove (both the local and global version of) the monodromy conjecture for a space monomial curve $Y \subset \C^{g+1}$. More precisely, in Theorem~\ref{thm:mon-conj}, we will show for every pole $\L^{\frac{\nu_k}{N_k}}$ with $\frac{\nu_k}{N_k} \notin \N$ that $e^{-2\pi i\frac{\nu_k}{N_k}}$ is a pole of $Z^{mon}_{Y,p}(t)$. It follows that every pole $\L^{-s_0}$ of the motivic Igusa zeta function associated with $Y$ indeed yields a monodromy eigenvalue $e^{2\pi is_0}$ of $Y$. 

\vspace{14pt}

We end the introduction with fixing some notation used throughout this article. We let $\N$ be the set of non-negative integers. The greatest common divisor and lowest common multiple of a set of integers $m_1,\ldots,m_r \in \mathbb Z$ is denoted by $\gcd(m_1,\ldots, m_r)$ and $\lcm(m_1,\ldots,m_r)$, respectively. To shorten the notation, we will sometimes use $(m_1,\ldots, m_r)$ for the greatest common divisor. A useful relation between these two numbers for $m_1,\ldots, m_r$ a set of non-zero integers and $m$ a common multiple is
\begin{equation} \label{eq:rel-gcd-lcm}
	\gcd\Big(\frac{m}{m_1}, \ldots, \frac{m}{m_r} \Big) = \frac{m}{\lcm(m_1,\ldots, m_r)}.
\end{equation}
Finally, by a \emph{complex variety}, we mean a reduced separated scheme of finite type over $\C$, which is not necessarily irreducible. A \emph{curve} is a variety of dimension one, and a \emph{surface} a variety of dimension two.

\vspace{14pt}

\textbf{Acknowledgement.} 
We would like to thank Hussein Mourtada for the suggestion to investigate the monodromy conjecture for the class of space monomial curves with a plane semigroup. We would also like to warmly thank the referees for their useful and constructive comments. The first author would also like to thank the Fulbright Program (within the Jos\'e Castillejo grant by Ministerio de Educaci\'on, Cultura y Deporte) for its financial support while writing this paper and the University of Illinois at Chicago, especially Lawrence Ein, for the warm welcome and support in hosting him.


\section{Space monomial curves with a plane semigroup} \label{SpaceMonomial}

We start this article by introducing the class of monomial curves we are interested in. They arise in a natural way as the special fibers of equisingular families of curves whose generic fibers are isomorphic to a plane branch. More precisely, let ${\Cu := {\{f = 0\}} \subset (\C^2,0)}$ be the germ at the origin of an irreducible plane curve defined by a complex irreducible series $f \in \C[[x_0,x_1]]$ with $f(0)=0$. Carrying out a linear change of variables if necessary, we can assume that the curve $\{x_0=0\}$ is transversal to $\Cu$ and that the curve $\{x_1 = 0\}$ has maximal contact (among all smooth curves) with $\Cu$. For $h \in \C[[x_0,x_1]]$, the \emph{local intersection multiplicity} of $\Cu$ and the curve $\{h = 0\}$ is defined as \[(f,h)_0 := \dim_\C \frac{\C[[x_0,x_1]]}{(f,h)}.\] This induces a valuation \[\nu_\Cu: \frac{\C[[x_0,x_1]]}{(f)} \setminus \{0\} \longrightarrow \N: h \mapsto (f,h)_0.\] The image of this valuation is called the \emph{semigroup} of $\Cu$ and denoted by $\Gamma(\Cu)$. Because $\N \setminus \Gamma(\Cu)$ is finite, there exists a unique minimal system of generators $(\bar{\beta}_0,\ldots,\bar{\beta}_g)$ of $\Gamma(\Cu)$ satisfying $\lbeta_0 < \cdots < \lbeta_g$ and $\gcd(\lbeta_0, \ldots, \lbeta_g) = 1$, see for instance~\cite{Z}. Additionally, we introduce the integers $e_i:=\gcd(\bar{\beta}_0,\ldots,\bar{\beta}_i)$ for $i=0,\ldots,g$ and $n_i:=\frac{e_{i-1}}{e_i}$ for $i=1,\ldots,g.$ From the minimality of the generators $(\bar{\beta}_0,\ldots,\bar{\beta}_g)$, one can easily see that ${\lbeta_0 = e_0 > e_1 > \cdots > e_g = 1}$ and that $n_i \geq 2$ for all $i = 1,\ldots, g$. One can also show that every $n_i\bar{\beta}_i$ for $i=1,\ldots,g$ is contained in the semigroup generated by $\bar{\beta}_0,\ldots,\bar{\beta}_{i-1}$; this follows for example from~\cite{Az}. In other words, for each $i = 1, \ldots, g$, we can find non-negative integers $b_{ij}$ for $0 \leq j < i$ such that
	\begin{equation}\label{eq:ni-betai}
		n_i\bar{\beta}_i=b_{i0}\bar{\beta}_0+\cdots +b_{i(i-1)}\bar{\beta}_{i-1}.
	\end{equation}
If we require in addition that $b_{ij}<n_j$ for $j \neq 0$, then these integers are unique. For later purposes, we denote $n_0 := b_{10}$ and list some other properties used in this article: 
\begin{enumerate}
	\item[(i)] for $i = 0,\ldots, g-1$, we have that $e_i = n_{i+1}\cdots n_g$; 
	\item[(ii)] for $i = 0,\ldots, g-1$, we have that $n_j \mid \lbeta_i$ for all $j > i$;
	\item[(iii)] for $i = 1,\ldots, g$, we have that $\gcd(\frac{\lbeta_i}{e_i},n_i) = \gcd(\frac{\lbeta_i}{e_i},\frac{e_{i-1}}{e_i}) = 1$, and, in particular, that $\gcd(n_0,n_1) = \gcd(\frac{\lbeta_1}{e_1},n_1) = 1$; and
	\item[(iv)] for $i = 1,\ldots, g$, we have that $n_i\lbeta_i < \lbeta_{i+1}$.
\end{enumerate} 
In terms of the generators $(\lbeta_0,\ldots, \lbeta_g)$, the curve we will consider is defined as the image of the monomial map $M:(\C,0) \rightarrow (\C^{g+1},0)$ given by $t \mapsto (t^{\bar{\beta}_0},\ldots,t^{\bar{\beta}_g})$. We denote this curve by $Y$ and call it the \emph{monomial curve associated with} $\Cu$. It is an irreducible (germ of a) curve with $\Gamma(\Cu)$ as semigroup and which is smooth outside the origin, see~\cite{T1} for these and other properties of $Y$.

\vspace{14pt}

We can construct $Y$ as a deformation of $\Cu$ as follows. First of all, we can consider a \emph{system of approximate roots} or a \emph{minimal generating sequence} $(x_0, \ldots, x_g)$ of the valuation $\nu_{\C}$, which consists of elements $x_i \in \C[[x_0,x_1]]$ for $i=0,\ldots,g$ such that $\nu_{\C}(x_i)=\bar{\beta}_i$, see for instance~\cite{AM},~\cite{Sp} and~\cite{T1}.  For $i=0,1$, this condition is equivalent to the above assumptions on $x_0$ and $x_1$, respectively. These elements satisfy equations of the form \[x_{i+1}=x_i^{n_i}-c_ix_0^{b_{i0}}\cdots x_{i-1}^{b_{i(i-1)}}- \sum_{\gamma=(\gamma_0,\ldots,\gamma_i)} c_{i,\gamma }x_0^{\gamma_0}\cdots x_i^{\gamma_i}, \qquad  i = 0, \ldots, g,\] where $x_{g+1}=0, c_i \in \C\setminus \{0\},  c_{i,\gamma }\in \C,$ $0\leq \gamma_j <n_j$ for $1\leq j \leq i$, and $\sum_{j=0}^i \gamma_j\bar{\beta}_{j}>n_i\bar{\beta}_{i}.$ These equations realize $\Cu$ as a complete intersection in $(\C^{g+1},0)$. Even more, this complete intersection is Newton non-degenerate in the sense of~\cite{AGS} and~\cite{Te1}. It was proven (resp. conjectured) that such an embedding always exists in characteristic $0$~\cite{Te2} (resp. in positive characteristic~\cite{T2}). We now consider the following slight modification of the above equations in the variables $x_0,\ldots, x_g$ including an extra variable $v$: \[vx_{i+1}=x_i^{n_i}-c_ix_0^{b_{i0}}\cdots x_{i-1}^{b_{i(i-1)}}- \sum_{\gamma=(\gamma_0,\ldots,\gamma_i)} c_{i,\gamma }vx_0^{\gamma_0}\cdots x_i^{\gamma_i}, \qquad  i = 0, \ldots, g.\] For varying $v$ in $(\C,0)$, these equations define a family of germs of curves in $(\C^{g+1} \times \C,0)$, which is \emph{equisingular} for instance in the sense that $\Gamma(\Cu)$ is the semigroup of all curves in the family. We denote this family by $(\chi,0)$ and let $\eta:(\chi,0)\rightarrow (\C,0)$ be the restriction of the projection onto the second factor $(\C^{g+1}\times \C,0)\rightarrow (\C,0)$. The generic fiber $\eta^{-1}(v)$ for $v \neq 0$ is isomorphic to $\Cu$, and the special fiber $Y = \eta^{-1}(0)$ is defined in $(\C^{g+1},0)$ by the equations $x_i^{n_i} - c_ix_0^{b_{i0}} \cdots x_{i-1}^{b_{i(i-1)}} = 0$ for $i = 1, \ldots, g$. The coefficients $c_i$ are needed to see that any irreducible plane branch is a (equisingular) deformation of a such a curve. However, for simplicity, we will assume that every $c_i = 1$, which is always possible after a suitable change of coordinates. 

\vspace{14pt}

Clearly, we can also consider the global curve in $\C^{g+1}$ defined by the above binomial equations; from now on, we define a \emph{(space) monomial curve} $Y \subset \C^{g+1}$ as the complete intersection curve given by
	\begin{equation} \label{eq:equations-Y}
		\left\{\begin{array}{r c l l}
		f_1:= x_1^{n_1} & - & x_0^{n_0}  &  = 0 \\
		f_2:= x_2^{n_2}  &- & x_0^{b_{20}}x_1^{b_{21}} &= 0  \\
		& \vdots & &\\
		f_g := x_g^{n_g} &- & x_0^{b_{g0}}x_1^{b_{g1}}\cdots x_{g-1}^{b_{g(g-1)}} & = 0.\\
    	\end{array}\right.
	\end{equation}
This is still an irreducible curve which is smooth outside the origin. As such a monomial curve for $g = 1$ is just a cusp in the complex plane, of which the monodromy conjecture is well known, we will assume that $g \geq 2$. 


\section{The monodromy conjecture for ideals} \label{MonodromyGeneral}

This section provides a short introduction to the monodromy conjecture for ideals. Let $\I = (f_1, \ldots, f_r)$ be a non-trivial ideal in $\C[x_0, \ldots, x_n]$ and let $Y := V(\I)$ be its associated subscheme in the affine space $\C^{n+1}$. Assume that $Y$ contains the origin.

\vspace{14pt}

An important notion needed to introduce the monodromy conjecture for $\I$ is a \emph{principalization} (or \emph{log-principalization, log-resolution, monomialization}) of an ideal, which is a generalization of an \textit{embedded resolution} of a hypersurface. By Hironaka's Theorem~\cite{Hi}, a sequence of blow-ups can be used to transform a general ideal $\I = (f_1, \ldots, f_r)$ into a locally principal and monomial ideal. More formally, a \emph{principalization} of $\I$ is a proper birational morphism $\varphi: \tilde{X} \rightarrow \C^{n+1}$ from a smooth variety $\tilde{X}$ to $\C^{n+1}$ such that the total transform $\varphi^{\ast}\I$ is a locally principal and monomial ideal with support a simple normal crossings divisor, and such that the exceptional locus (or exceptional divisor) of $\varphi$ is contained in the support of $\varphi^{\ast}\I$. 

\vspace{14pt}

The motivic Igusa zeta function associated with $\I$ can be expressed in terms of a principalization $\varphi: \tilde{X} \rightarrow \C^{n+1}$ of $\I$ as follows. Let $E_j$ for $j \in J$ be the irreducible components (with their reduced scheme structure) of the total transform $\varphi^{-1}(Y)$. Among these, the components of the exceptional divisor are called the \emph{exceptional varieties}; the other components are components of the \emph{strict transform} of $Y$. Denote by $N_j$ the multiplicity of $E_j$ in the divisor on $\tilde{X}$ of $\varphi^{\ast}\I$, that is, the divisor of $\varphi^{\ast}\I$ is given by $\sum_{i \in J}N_jE_j$. Similarly, let $\nu_j-1$ be the multiplicity of $E_j$ in the divisor on $\tilde{X}$ of $\varphi^{\ast}(dx_0 \wedge \cdots \wedge dx_n)$. The numbers $(N_j,\nu_j)$ for $j \in J$ are called the \emph{numerical data} of the principalization. For every subset $I \subset J$, we also define $E_I^\circ := (\cap_{i \in I} E_i) \backslash (\cup_{l \not \in I}E_l)$. In terms of this notation, the \emph{local motivic Igusa zeta function} associated with the ideal $\I$ (or with the scheme $Y$) is given by \[Z^{mot}_{\I}(T) = \L^{-(n+1)}\sum_{I\subset J} [E^{\circ}_I \cap\varphi^{-1}(0)]\prod_{i \in I} \frac{(\L-1)\L^{-\nu_i}T^{N_i}}{1-\L^{-\nu_i}T^{ N_i}} \in \mathcal{M}_{\C}[[T]].\] Here, $[E^{\circ}_I\cap\varphi^{-1}(0)]$ and $\L := [\C]$ are the class of $E^{\circ}_I\cap\varphi^{-1}(0)$ and of the affine line, respectively, in the \emph{Grothendieck ring} of complex varieties $K_0(\text{Var}_\C)$, and $\mathcal{M}_\C$ is the localization of $K_0(\text{Var}_\C)$ with respect to $\L$. The precise definition of the Grothendieck ring of complex varieties can be found for instance in~\cite{MVV}. In the \emph{global version} of the motivic zeta function, we replace $[E^{\circ}_I\cap\varphi^{-1}(0)]$ by $[E^{\circ}_I]$. From this expression, it is immediate that both the local and the global motivic zeta function are rational functions in $T$, and that all candidate poles are of the form $\L^{\frac{\nu_j}{N_j}}$ for some $j \in J$. In concrete examples `most' of these candidate poles cancel; a phenomenon that the monodromy conjecture tries to explain.

\begin{remark}
In~~\cite{DL2}, Denef and Loeser introduced the motivic Igusa zeta function for a polynomial $f$ using the \emph{jet schemes} of $\{f = 0\}$, instead of an embedded resolution. However, in the same article, they showed the equivalence between both expressions. Similarly, one can write the motivic zeta function associated with a general ideal $\I$ in terms of the jet schemes of its corresponding scheme $V(\I)$. In fact, this is the definition used to compute the motivic zeta function of a space monomial curve $Y \subset \C^{g+1}$ in~\cite{MVV}. 
\end{remark}

The monodromy eigenvalues associated with the ideal $\I$ can also be expressed in terms of a principalization of $\I$. Before elaborating on this, we first briefly discuss the original definition by Verdier. For more details, we refer to~\cite{Di} and~\cite{VV}. For one polynomial $f \in \C[x_0,\ldots, x_n]$, there are two equivalent definitions for its eigenvalues of monodromy: the original definition in terms of the \emph{Milnor fibration}~\cite{Mi}, and a more abstract description by Deligne~\cite{De} using the notion of the \emph{complex of nearby cycles} on $Y = \{f=0\}$. While the original definition does not have a (straightforward) generalization to ideals, Deligne's description was the inspiration for Verdier~\cite{Ver} to define monodromy eigenvalues for an ideal by introducing the notion of the \emph{specialization complex} as follows. For a scheme $Z$, we denote by $D^b_c(Z)$ the full subcategory of the derived category $D(Z)$ consisting of complexes of sheaves of $\C$-vector spaces with bounded and constructible cohomology, and by $\C^{\boldsymbol\cdot} \in D^b_c(Z)$ the complex concentrated in degree zero induced by the constant sheaf $\C_Z$ on~$Z$. For one polynomial $f$, we can associate with $\C^{\boldsymbol\cdot} \in D^b_c(\C^{n+1})$ the \emph{complex of nearby cycles} $\psi_f\C^{\boldsymbol{\cdot}} \in D^b_c(Y)$ equipped with a \emph{monodromy transformation} $M^k_y: \mathcal H^k(\psi_f\C^{\boldsymbol{\cdot}})_y \rightarrow \mathcal H^k(\psi_f\C^{\boldsymbol{\cdot}})_y$ for each $y \in Y$ and $k\geq 0$, where $\mathcal H^k(\psi_f\C^{\boldsymbol{\cdot}})_y$ denotes the stalk at $y$ of the $k$th cohomology sheaf of $\psi_f\C^{\boldsymbol{\cdot}}$. An \emph{eigenvalue of monodromy} or \emph{monodromy eigenvalue} of $f$ (or of $Y$) is an eigenvalue of such a transformation $M^k_y$ for some $y \in Y$ and $k \geq 0$. For a general ideal $\I$, Verdier considered the \emph{normal cone} $C_Y\C^{n+1}$ of $Y = V(\I)$ in $\C^{n+1}$ defined as \[C_Y\C^{n+1} := \text{Spec}(\oplus_{k\geq 0}\I^k/\I^{k+1}),\] and related to $\C^{\boldsymbol{\cdot}} \in D^b_c(\C^{n+1})$ the \emph{specialization complex} $\text{Sp}_Y\C^{\boldsymbol{\cdot}} \in D^b_c(C_Y\C^{n+1})$ with a monodromy transformation $M^k_y: \mathcal H^k(\text{Sp}_Y\C^{\boldsymbol{\cdot}})_y \rightarrow \mathcal H^k(\text{Sp}_Y\C^{\boldsymbol{\cdot}})_y$ for each $y \in C_Y\C^{n+1}\setminus Y$ and $k\geq 0$. The \emph{(Verdier) monodromy eigenvalues} of $\I$ (or of $Y$) are the eigenvalues of these automorphisms. Despite the fact that the specialization complex lives on the normal cone of $Y$ instead of on $Y$ itself, where the complex of nearby cycles lives, it turns out that these two definitions for the monodromy eigenvalues in the hypersurface case are equivalent. 

\vspace{14pt}

In~\cite{AC}, A'Campo proved a formula for the monodromy eigenvalues of a polynomial $f$ in terms of an embedded resolution of $\{f=0\}$. This formula was generalized to ideals in~\cite{VV}. Later in this article, we will make use of an A'Campo formula in the more general context of a Cartier divisor on a normal surface. In fact, the notion of monodromy eigenvalues can be generalized in a straightforward way to any ideal sheaf $\I$ on a general variety $X$. Therefore, we state the formula in the following general context. Let $\I$ be a sheaf of ideals on a variety $X$, let $Y := V(\I)$ be the associated subscheme in $X$, and suppose that $\text{Sing}(X) \subset Y$. Consider the blow-up $\sigma: X' \rightarrow X$ of $X$ with center $Y$, and let $E'$ be its exceptional divisor, that is, the inverse image $\sigma^{-1}(Y)$ (with its non-reduced scheme structure). One can show that $E'$ is the projectivization $P(C_YX)$ of the normal cone $C_YX$ of $Y$ in $X$. Denote the corresponding projectivization map by $p:C_YX \setminus Y \rightarrow E' = P(C_YX)$. For a point $e\in E'$, we define the monodromy eigenvalues of $\I$ at $e$ as the eigenvalues of the monodromy transformation $M^k_y$ for some $y \in C_YX\setminus Y$ mapped to $e$ under $p$; this is independent of the choice of $y$. Hence, we can define the \emph{zeta function of monodromy} or \emph{monodromy zeta function} of $\I$ at $e \in E'$ as \[Z_{\I,e}^{mon}(t) := \prod_{k \geq 0}\det(\text{Id} - tM^k_y)^{(-1)^k},\] where $y \in C_YX\setminus Y$ is an arbitrary point in $p^{-1}(e)$. For one polynomial $f$, Denef~\cite[Lemma 4.6]{Den2} showed that every monodromy eigenvalue associated with $f$ is a zero or pole of the monodromy zeta function of $f$ at some point $e \in E'$. This result can easily be generalized to ideals. 

\begin{theorem}\cite{VV} \label{thm:ACampo-Princ}
Let $\I$ be a sheaf of ideals on a variety $X$. Let $Y = V(\I)$ be the associated subscheme in $X$, and suppose that $\text{Sing}(X) \subset Y$. Consider a principalization $\varphi: \tilde X \rightarrow X$ of $\I$. Denote by $E_j$ for $j \in J$ the irreducible components of $\varphi^{-1}(Y)$ with numerical data $(N_j,\nu_j)$, and define $E^{\circ}_j = E_j \setminus \cup_{i \neq j}(E_i \cap E_j)$ for every $j \in J$. Let ${\sigma: X' \rightarrow X}$ be the blow-up of $X$ with center $Y$ and let $E' = \sigma^{-1}(Y)$ be its exceptional divisor. By the universal property of the blow-up, there exists a unique morphism $\psi: \tilde X \rightarrow X'$ such that $\sigma \circ \psi = \varphi$. For a point $e \in E'$, the zeta function of monodromy of $\I$ at $e$ is given by \[Z_{\I,e}^{mon} (t)= \prod_{j \in J}(1-t^{N_j})^{\chi(E^{\circ}_j \cap \psi^{-1}(e))},\] where $\chi$ denotes the topological Euler characteristic. 
\end{theorem}

When $\I = (f)$ is a principal ideal, we can consider the blow-up $\sigma$ as the identity so that $\varphi = \psi$ and 
	\begin{equation}\label{eq:ACampo-poly}
		Z_{f,y}^{mon} (t)= \prod_{j \in J}(1-t^{N_j})^{\chi(E^{\circ}_j\cap \varphi^{-1}(y))},
	\end{equation}
which is the classical A'Campo formula for $y \in Y = \{f=0\}$. In the next section, we will introduce another generalization of this formula in which the final ambient space $\tilde X$ of the embedded resolution $\varphi: \tilde X \rightarrow X$ of $\{f = 0\}$ is allowed to have abelian quotient singularities. Such a resolution is called an \emph{embedded $\Q$-resolution}, and it is this formula that we will use to compute the monodromy eigenvalues associated with a space monomial curve $Y \subset \C^{g+1}$ by considering it as a Cartier divisor on a generic embedding surface.

\vspace{14pt}

After having introduced the two invariants of an ideal that are investigated in the monodromy conjecture, we can now state this conjecture in more detail. 

\begin{conj}
Let $\I = (f_1,\ldots, f_r)$ be an ideal in $\C[x_0,\ldots, x_n]$ whose associated subscheme $Y = V(\I)$ in $\C^{n+1}$ contains the origin. Let $\sigma: X' \rightarrow \C^{n+1}$ be the blow-up of $\C^{n+1}$ with center $Y$. If $\L^{-s_0}$ is a pole of the local motivic Igusa zeta function associated with $\I$, then $e^{2\pi i s_0}$ is a zero or pole of the monodromy zeta function of $\I$ at a point in $\sigma^{-1}(B \cap Y)$ for $B \subset \C^{n+1}$ a small ball around the origin.
\end{conj}

So far, this conjecture has only been proven for ideals in two variables~\cite{VV}. In this article, we will show the conjecture for the space monomial curves introduced in Section~\ref{SpaceMonomial}; this solves it for an interesting class of binomial ideals in arbitrary dimension. Along the same lines, we will also prove the global version of the monodromy conjecture.


\section{Monodromy zeta function formula for embedded \texorpdfstring{$\Q$}{Q}-resolutions} \label{MonodromyWeighted}

As mentioned earlier in this article, we will make use of an A'Campo formula for the monodromy zeta function of a polynomial $f \in \C[x_0,\ldots, x_n]$ in terms of an \emph{embedded $\Q$-resolution} of $\{f=0\}$. Roughly speaking, this is a resolution $\varphi: \tilde X \rightarrow \C^{n+1}$ in which we allow $\tilde X$ to have abelian quotient singularities and the divisor $\varphi^{-1}(\{f=0\})$ to have normal crossings on such a variety. In this section, we briefly introduce all concepts needed to understand this formula. We refer to~\cite{AMO1} for more details.

\vspace{14pt}

We start with the notion of a \emph{$V$-manifold} of dimension $n$ which was introduced by Satake~\cite{Sa} as a complex analytic space admitting an open covering $\{U_i\}$ in which each $U_i$ is analytically isomorphic to some quotient $B_i/G_i$ for $B_i \subset \C^n$ an open ball and $G_i$ a finite subgroup of $GL(n,\C)$. We are interested in $V$-manifolds in which every $G_i$ is a finite abelian subgroup of $GL(n,\C)$. In fact, every quotient $\C^n/G$ for $G \subset GL(n,\C)$ a finite abelian group is isomorphic to a specific kind of quotient space, called a \emph{quotient space of type $(\d;A)$} in which $\d$ is an $r$-tuple of positive integers and $A$ is an $(r\times n)$-matrix over the integers. More precisely, we can write $G = \mu_{d_1} \times \cdots \times \mu_{d_r}$ as a product of finite cyclic groups, where $\mu_{d_i}$ is the cyclic group of the $d_i$th roots of unity. We will denote $G$ by $\mu_{\d}$, where $\d$ is the $r$-tuple $(d_1,\ldots, d_r)$, and an element in $\mu_{\d}$ by $\bxi_{\d} := (\xi_{d_1},\ldots, \xi_{d_r})$. For a matrix $A = (a_{ij})_{i,j} \in \mathbb Z^{r\times n}$, we can define an action of $\mu_{\d}$ on $\C^n$ by 
	\begin{equation}\label{eq:action}
		\mu_{\d} \times \C^n  \longrightarrow  \C^n:(\bxi_{\d} , \x)  \mapsto (\xi_{\d}^{\a_1}x_1,\ldots, \xi_{\d}^{\a_n}x_n) =  (\xi_{d_1}^{a_{11}} \cdots \xi_{d_r}^{a_{r1}}\, x_1,\, \ldots\, , \xi_{d_1}^{a_{1n}} \cdots \xi_{d_r}^{a_{rn}}\, x_n ),
	\end{equation}
where $\a_j := (a_{1j},\ldots, a_{rj})^t$ is the $j$th column of $A$. Note that we can always consider the $i$th row $(a_{i1},\ldots, a_{in})$ of $A$ modulo $d_i$. The resulting quotient space $\C^n/\mu_{\d}$ is called the \emph{quotient space of type $(\d;A)$} and denoted by \[ X(\d; A) := X \left( \begin{array}{c|ccc} d_1 & a_{11} & \cdots & a_{1n}\\ \vdots & \vdots & \ddots & \vdots \\ d_r & a_{r1} & \cdots & a_{rn} \end{array} \right).\] If $r = 1$, the quotient space $X(d;a_1,\ldots, a_n)$ is said to be \emph{cyclic}. The class of an element $\x = (x_1,\ldots, x_n) \in \C^n$ under an action $(\d;A)$ will be denoted by $[\x]_{(\d;A)} = [(x_1,\ldots, x_n)]_{(\d;A)}$, where the subindex is omitted if there is no possible confusion. The image of each coordinate hyperplane $\{x_i = 0\}$ in $\C^n$ for $i = 1,\ldots, n$ under the natural projection $\C^n \rightarrow X(\d;A)$ will still be denoted by $\{x_i = 0\}$ and called a \emph{coordinate hyperplane} in $X(\d;A)$. One can show that the original quotient space $\C^n/G$ is isomorphic to $X(\d;A)$ for some matrix $A$, and that every space $X(\d;A)$ is a normal irreducible algebraic variety of dimension $n$ with its singular locus, which is of codimension at least two, situated on the coordinate hyperplanes. Hence, a $V$-manifold with abelian quotient singularities is a normal variety which can locally be written like $X(\d;A)$. 

\begin{ex}\label{ex:quotient-space}
If $n = 1$, then each quotient space $X((d_1,\ldots, d_r);(a_{11},\ldots, a_{r1})^t)$ is isomorphic to $\C$: let  \[l = \lcm\bigg(\frac{d_1}{\gcd(d_1,a_{11})},\ldots, \frac{d_r}{\gcd(d_r,a_{r1})}\bigg),\] then   $X((d_1,\ldots, d_r);(a_{11},\ldots, a_{r1})^t) \rightarrow \C:[x] \rightarrow x^l$ is an isomorphism.
\end{ex} 

Different types $(\d;A)$ can induce isomorphic quotient spaces: for example, if $k$ divides $d,a_2,\ldots, a_n$, then $X(d;a_1,\ldots, a_n)$ is isomorphic to $X(\frac{d}{k};a_1,\frac{a_2}{k},\ldots, \frac{a_n}{k})$ under the isomorphism defined by 
	\begin{equation}\label{eq:normalizing}
 		[(x_1,x_2,\ldots, x_n)]\longmapsto [(x_1^k,x_2,\ldots,x_n)].
	\end{equation}
A particularly interesting kind of types are the \emph{normalized} types. These are types $(\d;A)$ in which the group $\mu_{\d}$ is small as subgroup of $GL(n,\C)$ (i.e., it does not contain rotations around hyperplanes other than the identity) and acts freely on $(\C^{\ast})^n$. In this case, we will also say that the quotient space $X(\d;A)$ is \emph{written in a normalized form}. Equivalently, a space $X(\d;A)$ is written in a normalized form if and only if for all $\x \in \C^n$ with exactly $n-1$ coordinates different from $0$, the stabilizer subgroup is trivial. Note that in the cyclic case, the stabilizer subgroup of a point $(x_1, \ldots, x_n) \in \C^n$ with only $x_i = 0$ has order $\gcd(d,a_1,\ldots, \hat{a}_i,\ldots, a_n)$. 

\begin{ex}
The space $X(d;a_1,a_2)$ is written in a normalized form if and only if both $\gcd(d,a_1)$ and $\gcd(d,a_2)$ are equal to $1$. We can normalize it with the isomorphism (assuming that $\gcd(d,a_1,a_2) = 1$) \[X(d;a_1,a_2)  \longrightarrow X \bigg( \frac{d}{(d,a_1)(d,a_2)}; \frac{a_1}{(d,a_1)}, \frac{a_2}{(d,a_2)}\bigg): [(x_1,x_2)] \mapsto \big[ (x_1^{(d,a_2)},x_2^{(d,a_1)}) \big],\] which is the composition of two isomorphisms of the form~\eqref{eq:normalizing}.
\end{ex}

In general, it is possible to convert any type into a normalized form. Especially in the cyclic case, this is not hard, using isomorphisms such as~\eqref{eq:normalizing}. See~\cite[Lemma 1.8]{AMO1} for a list of some other useful isomorphisms.

\vspace{14pt}

An \emph{analytic function} $f: X(\d;A) \rightarrow \C$ on a quotient space of some type $(\d;A)$ is a holomorphic function $f: \C^n \rightarrow \C$ compatible with the action, that is, $f({\bf\xi}_{\d}\cdot \x)  = f(\x)$ for all ${\bf\xi}_{\d}\in \mu_{\d}$ and $\x \in \C^n$. To compute the local equation of the divisor defined by $f:(X(\d;A), [{\bf p}]) \rightarrow (\C,0)$ as a germ of functions at ${\bf p} = (p_1, \ldots, p_n) \in \C^n \setminus \{0\}$, one would naturally use the change of coordinates $x_i \mapsto x_i + p_i$. However, this coordinate change induces an isomorphism on $X(\d;A)$ if and only if the $i$th row of $A$ is zero (modulo $d_i$) for all $i$ for which $p_i \neq 0$. Hence, we first need to find an isomorphism $(X(\d;A),[{\bf p}]) \simeq (X(\d';A'),[{\bf p}])$ with $(\d';A')$ having this property. One can show that this is satisfied by $(\d';A')$ with $X(\d';A') = \C^n/(\mu_{\d})_{\bf p}$, where $(\mu_{\d})_{\bf p}$ is the stabilizer subgroup of ${\bf p}$. In particular, if $X(d;a_1,\ldots, a_n)$ is cyclic, then the order of the stabilizer subgroup of ${\bf p}$ is $m = \gcd(d,\{a_i \mid p_i \neq 0\})$ so that $(\d';A') = (m;a_1,\ldots,a_n)$ in which $a_i$ modulo $m$ will be zero if $p_i \neq 0$. On $X(\d';A')$, we can apply the usual change of coordinates $x_i \mapsto x_i + p_i$ to find the local equation of $f$ at ${\bf p}$. This method will be very useful for the description of the $\Q$-resolution of a space monomial curve seen as a Cartier divisor on a generic embedding surface in Section~\ref{Resolution}.

\vspace{14pt}

An important class of $V$-manifolds are the \emph{weighted projective spaces}. Consider a weight vector $\omega = (p_0,\ldots, p_n)$ of positive integers. The \emph{weighted projective space of type $\omega$}, denoted by $\P^n_{\omega}$, is the set of orbits $(\C^{n+1}\setminus\{0\})/\C^{\ast}$ under the action \[\C^{\ast} \times (\C^{n+1}\setminus \{0\}) \longrightarrow \C^{n+1}\setminus \{0\}: (t,(x_0,\ldots,x_n)) \mapsto (t^{p_0}x_0,\ldots, t^{p_n}x_n).\] We denote the class of an element $\x = (x_0,\ldots, x_n)\in\C^{n+1}\setminus\{0\}$ by $[\x]_{\omega} = [x_0:\ldots:x_n]_{\omega}$, where we again omit $\omega$ if possible. Note that for the trivial weight vector $\omega =(1,\ldots,1)$, we obtain the classical projective space $\P^n$. Furthermore, one can show that $\P^1_{\omega}$ is always isomorphic to $\P^1$, cf. Example~\ref{ex:quotient-space}. As for the classical projective space, we can define an open covering $\P^n_{\omega} = V_0 \cup \cdots \cup V_n,$ where $V_i := \{x_i \neq 0\}$. It is easy to see that for every $i$, the map \[X(p_i;p_0, \ldots, \hat{p}_i, \ldots, p_n) \longrightarrow V_i: (x_0, \ldots, \hat{x}_i, \ldots,x_n) \mapsto [x_0: \ldots:x_{i-1}:1:x_{i+1}: \ldots:x_n]_{\omega}\] is an isomorphism. It follows that $\P^n_{\omega}$ contains cyclic quotient singularities. Even more, each weighted projective space $\P^n_{\omega}$ is a normal irreducible projective variety of dimension $n$ whose singular locus, which is of codimension at least two, consists of quotient singularities lying on the intersection of at least two coordinate hyperplanes. For more information on weighted projective spaces, see for instance~\cite{Dol}.

\vspace{14pt}

Another notion we need is a \emph{$\Q$-normal crossings divisor}, which was first introduced by Steenbrink~\cite{Ste}. Let $X$ be a $V$-manifold with abelian quotient singularities and $D$ a hypersurface on $X$. We say that $D$ has \emph{$\Q$-normal crossings} if it is locally isomorphic to the quotient of a normal crossings divisor under an action $(\d;A)$. More precisely, for every point $p \in X$, there exists an isomorphism of germs $(X,p) \simeq (X(\d;A),[0])$ such that $(D,p) \subseteq (X,p)$ is identified with a germ of the form \[(\{[\x] \in X(\d;A)\mid x_1^{m_1}\cdots x_k^{m_k} = 0\},[0]).\] The \emph{multiplicity} of a $\Q$-normal crossings divisor $D$ at a point $p \in D$ is defined as follows. Suppose that $p$ is contained in only one irreducible component of $D$; we will only consider this situation, see~\cite{Ma2} for a more general definition in case $p$ is possibly contained in multiple irreducible components. In this case, the local equation of $D$ at $p$ is of the form $x_i^m: X(\d;A) \rightarrow \C$ for $x_i$ a local coordinate of $X$ at $p$. The \emph{multiplicity} $m(D,p)$ of $D$ at $p$ is defined as 
	\begin{equation}\label{eq:def-mult}
		m(D,p) := \frac{m}{l_i}, \qquad l_i := \lcm\bigg(\frac{d_1}{\gcd(d_1,a_{1i})},\ldots, \frac{d_r}{\gcd(d_r,a_{ri})}\bigg).
	\end{equation}
One can show that this definition is independent of the type $(\d;A)$.

\vspace{14pt}

We can now define an \emph{embedded $\Q$-resolution}, see for instance~\cite{AMO2}. Let $X$ be an abelian quotient space and $Y \subseteq X$ an analytic subvariety of codimension one. An \emph{embedded $\Q$-resolution} of $(Y,0) \subseteq (X,0)$ is a proper analytic map $\varphi: \tilde{X} \rightarrow (X,0)$ such that the following properties hold:
\begin{enumerate}
	\item[(i)] $\tilde{X}$ is a $V$-manifold with abelian quotient singularities;
	\item[(ii)] $\varphi$ is an isomorphism over $\tilde{X} \setminus \varphi^{-1}(\text{Sing}(Y))$; and
	\item[(iii)] the total transform $\varphi^{-1}(Y)$ is a hypersurface with $\Q$-normal crossings on $\tilde{X}$.
\end{enumerate}
As for usual embedded resolutions, we can use the operation of blowing up to construct an embedded $\Q$-resolution, but in this case, we use \emph{weighted blow-ups}. Since we will only use weighted blow-ups at a point in this article, we restrict to explaining this kind of blow-ups.

\vspace{14pt}

We first briefly recall the classical blow-up of $\C^{n+1}$ at the origin. We use the notation $\x := (x_0,\ldots,x_n) \in \C^{n+1}$ and $[{\bf u}] := [u_0:\ldots:u_n] \in \P^n$. Define \[\hat{\C}^{n+1}:= \big\{(\x,[{\bf u}])\in \C^{n+1} \times \P^n \mid {\bf x } \in \overline{[{\bf u}]}\big\},\] where ${\bf x } \in \overline{[{\bf u}]}$ means that $u_ix_j=u_jx_i$ for all  $i,j = 0, \ldots, n$. The \emph{blow-up of $\C^{n+1}$ at $0$} is given by the projection $\pi: \hat{\C}^{n+1} \rightarrow \C^{n+1}$. This is a proper birational morphism inducing an isomorphism $\hat{\C}^{n+1} \setminus \pi^{-1}(0) \simeq \C^{n+1} \setminus \{0\}$. The exceptional divisor $\pi^{-1}(0)$ can be identified with $\P^n$, and $\hat{\C}^{n+1}$ can be covered by $n+1$ charts ${U_i := \{u_i \neq 0\}}$ which are isomorphic to $\C^{n+1}$ under maps of the form \[\C^{n+1} \longrightarrow U_i: {\bf x } \mapsto \big((x_0x_i, \ldots,x_i,\ldots, x_nx_i),[x_0: \ldots:x_{i-1}:1:x_{i+1}: \ldots:x_n] \big).\] 

\vspace{14pt}

The weighted blow-up of $\C^{n+1}$ at the origin with respect to a weight vector $\omega = (p_0,\ldots,p_n)$ of positive integers is defined similarly. Let \[\hat{\C}_{\omega}^{n+1}:= \big\{(\x, [{\bf u}]_{\omega})\in \C^{n+1} \times \P^{n}_{\omega} \mid \x \in \overline{[{\bf u}]}_{\omega}\big\},\] then \emph{the $\omega$-weighted blow-up of $\C^{n+1}$ at $0$} is the projection $\pi: \hat{\C}_{\omega}^{n+1} \rightarrow \C^{n+1}$. In this case, the condition $\x \in \overline{[{\bf u}]}_{\omega}$ can be rewritten as $x_i=t^{p_i}u_i$ for all $i = 0, \ldots, n$ and some fixed $t \in \C \setminus \{0\}$. This blow-up is again a proper birational morphism and it is an isomorphism on $\hat{\C}_{\omega}^{n+1} \setminus \pi^{-1}(0)$. The exceptional divisor can now be identified with the weighted projective space $\P_{\omega}^n$, and $\hat{\C}_{\omega}^{n+1}$ can be covered by $n+1$ charts $U_i := \{u_i \neq 0\}$ where each $U_i$ is isomorphic to $X(p_i;p_0, \ldots,p_{i-1},-1,p_{i+1}, \ldots, p_n)$ under the morphism $X(p_i;p_0, \ldots,-1, \ldots, p_n)\rightarrow U_i$ defined by
	\begin{equation}\label{eq:chart-blow-up}
		\x \longmapsto  \big((x_0x_i^{p_0}, \ldots, x_i^{p_i}, \ldots,x_nx_i^{p_n}),[x_0: \ldots:1:\ldots: x_n]_{\omega} \big).
	\end{equation}
These charts are compatible with the charts $V_i$ of $\P^n_{\omega}$ described above in the following sense: in $U_i$, the exceptional divisor is described by $x_i = 0$, and the $i$th chart of $\P^n_{\omega}$ is $X(p_i;p_0,\ldots, \hat{p}_i,\ldots, p_n)$. 

\vspace{14pt}

For a general abelian quotient space $X(\d;A) = \C^{n+1}/\mu_{\d}$, the weighted blow-up at $0$ with respect to $\omega =(p_0,\ldots, p_n)$ can be obtained from the $\omega$-weighted blow-up of $\C^{n+1}$ at $0$ as follows. The action of $\mu_{\d}$ on $\C^{n+1}$ extends in a natural way to an action on $\hat{\C}^{n+1}_{\omega}$ by \[\bxi_{\d} \cdot (\x,[{\bf u}]_{\omega}) = \big((\bxi_{\d}^{\a_0}\, x_0,\, \ldots\, ,\bxi_{\d}^{\a_n}\, x_n),\, [\bxi_{\d}^{\a_0}\, u_0:\, \ldots\, :\bxi_{\d}^{\a_n}\, u_n]_{\omega}\big).\] The \emph{$\omega$-weighted blow-up of $X(\d;A)$ at $0$} is defined as the projection \[\pi: \hat{X}(\d;A)_{\omega} := \hat{\C}^{n+1}_{\omega} / \mu_{\d} \longrightarrow X(\d;A):[(\x, [{\bf u}]_{\omega})]_{(\d;A)} \mapsto [\x]_{(\d;A)},\] which is once more a proper birational morphism. It induces an isomorphism on $\hat{X}(\d;A)_{\omega} \setminus \pi^{-1}(0)$, and the exceptional divisor is identified with $\P^n_{\omega}/{\mu_{\d}}$, which we will also write as $\P^n_w(\d;A)$. Because the action of $\mu_{\d}$ on $\hat{\C}^{n+1}_{\omega}$ respects the charts $U_i = \{u_i \neq 0\}$ of $\hat{\C}^{n+1}_{\omega}$, we can cover $\hat{X}(\d;A)_{\omega}$ with the $n+1$ charts $\hat{U}_i := U_i/\mu_{\d}$. Using the isomorphisms $U_i \simeq X(p_i;p_0, \ldots, -1,\ldots, p_n)$, one can show that each $\hat{U}_i$ is also isomorphic to an abelian quotient space. For example, under the isomorphism $U_0 \simeq X(p_0;-1,p_1, \ldots,p_n)$, the action of $\mu_{\d} = \mu_{d_1} \times \cdots \times \mu_{d_r}$ on $U_0$ can be identified with the action of $\mu_{\d p_0}/(\mu_{p_0}\times \cdots \times \mu_{p_0})$ on $X(p_0;-1,p_1, \ldots,p_n)$ given by \[[{\bf\xi}] \cdot [\x]_{(\d;A)} = [(\bxi^{\a_0}\,x_0,\bxi^{p_0\a_1-p_1\a_0}\,x_1, \ldots, \xi^{p_0\a_n-p_n\a_0}\,x_n)]_{(\d;A)}.\] Hence, the quotient space 
	\begin{equation}\label{eq:chart-blow-up-quotient-space}
		X\left(\begin{array}{c | cccc} 
		p_0 & -1 & p_1 & \cdots & p_n \\ 
		\d p_0 & \a_0 & p_0\a_1-p_1\a_0 & \cdots &p_0\a_n-p_n\a_0 
		\end{array} \right)
	\end{equation}
is isomorphic to $\hat{U}_0$ under the map \[[\x] \longmapsto \big[\big((x_0^{p_0}, x_0^{p_1}x_1, \ldots, x_0^{p_n}x_n),[1:x_1: \ldots:x_n]_{\omega}\big)\big]_{(\d;A)}.\] The other charts are similar. The charts of $\hat{X}(\d;A)_{\omega}$ are again compatible with those of the exceptional divisor: we can cover $\P^n_{\omega}/{\mu_{\d}} = \hat{V}_0\cup \cdots \cup \hat{V}_k$ with $\hat{V}_i := V_i/\mu_{\d}$ and $\hat{V}_i = \hat{U}_i\vert_{\{x_i=0\}}$. It follows, for example, that the space 
	\begin{equation}\label{eq:chart-weighted-proj-space-with-action}
		X\left(\begin{array}{c | ccc} p_0 &  p_1 & \cdots & p_n \\ 
		\d p_0 & p_0\a_1-p_1\a_0 & \cdots &p_0\a_n-p_n\a_0  \end{array} \right)
	\end{equation}
is isomorphic to $\hat{V}_0$.

\vspace{14pt}

We are finally ready to introduce the generalization of A'Campo's formula in terms of an embedded $\Q$-resolution. As in the previous section, we again work in a slightly more general situation; let $f: (X,0) \rightarrow (\C,0)$ be a non-constant regular function on a variety $X$ and let $(Y,0)$ be the hypersurface defined by $f$. Consider an embedded $\Q$-resolution $\varphi: \tilde X \rightarrow X$ of $(Y,0)$, and denote by $E_0$ and $E_j$ for $j = 1,\ldots, r$ the strict transform of $Y$ and the exceptional varieties, respectively. Define $E_I^{\circ} := ( \cap_{i \in I} E_i) \setminus ( \cup_{l\notin I} E_l)$ for every $I \subset \{0,\ldots,r\}$. Let $\tilde X = \sqcup_{l = 1 }^s Q_l$ be a finite stratification of $\tilde X$ given by its quotient singularities so that for every $I$ and $l$, there exist a fixed abelian group $G$ and positive integers $m_1,\ldots, m_k$ such that the local equation of $f \circ \varphi$ at a point $p \in E_I^{\circ} \cap Q_l$ is of the form $x_1^{m_1}\cdots x_k^{m_k}: B/G \rightarrow \C$ for $B$ an open ball around $p$ on which $G$ acts diagonally such as in~\eqref{eq:action}, and $x_1,\ldots, x_k$ local coordinates of $\tilde X$ at $p$. Lastly, for every $j = 1,\ldots, r$ and $l = 1,\ldots, s$, put $E_{j,l}^{\circ} := E_j^{\circ} \cap Q_l$ and $m_{j,l} := m(E_j,p)$ for a point $p \in E_{j,l}^{\circ}$, where the multiplicity defined as in~\eqref{eq:def-mult} is independent of the chosen point $p$. 

\begin{theorem}\cite{Ma1}\label{thm:ACampo-QEmb}
Let $f: (X,0) \rightarrow (\C,0)$ be a non-constant regular function on a variety $X$. Let $Y = \{f=0\}$ be its associated hypersurface in $X$, and suppose that $\text{Sing}(X) \subset Y$. Consider an embedded $\Q$-resolution $\varphi: \tilde X \rightarrow X$ of $(Y,0)$. Using the notation above, the zeta function of monodromy of $f$ at $0$ is given by \[Z^{mon}_{f,0}(t) = \prod_{\underset{1 \leq l \leq s}{1 \leq j \leq r}} \left(1-t^{m_{j,l}}\right)^{\chi(E_{j,l}^{\circ})},\] where $\chi$ denotes the topological Euler characteristic.
\end{theorem}

In~\cite{Ma1}, this formula was proven for $f: (M,0) \rightarrow (\C,0)$ a non-constant analytic function germ on a quotient space $M = \C^{n+1}/\mu_{\d}$; by exactly the same arguments, this result can be obtained in our setting. Furthermore, for plane curve singularities in $\C^2$, this theorem was proven earlier in~\cite{Vey}, and if $\varphi: \tilde X \rightarrow X$ is an embedded resolution of $(Y,0)$, then we recover the classical formula~\eqref{eq:ACampo-poly} of A'Campo.


\section{Monodromy via generic embedding surfaces} \label{RedCurveSurface}

In this section, we will elaborate on how we can simplify the problem of computing the Verdier monodromy eigenvalues associated with a space monomial curve $Y \subset \C^{g+1}$ with $g \geq 2$ by considering $Y$ as a Cartier divisor on a generic embedding surface. As the results in this section are true for curves defined by a larger class of ideals, we state them in the following generalized setting; this makes them possibly useful to investigate the monodromy eigenvalues associated with other ideals in this class.

\vspace{14pt}

Consider a complete intersection curve $Y = V(\I)$ in $\C^{g+1}$ whose ideal $\I = (f_1, \ldots, f_g)$ is generated by a regular sequence $f_1, \ldots, f_g \in \C[x_0, \ldots, x_g]$, and whose singular set is $\text{Sing}(Y) = \{0\}$. We start with the construction of a \emph{generic embedding surface} of $Y$. For every set $(\lambda_2, \ldots, \lambda_g) $ of $g-1$ non-zero complex numbers, we introduce an affine scheme $S(\lambda_2,\ldots, \lambda_g)$ in $\C^{g+1}$ defined by
	\begin{equation} \label{eq:equations-S}
		\left\{\begin{array}{c c l l l}
		f_1& + & \lambda_2f_2  &  = 0 \\
		f_2& + & \lambda_3f_3 &= 0  \\
		& \vdots & &\\
		f_{g-1} & +  & \lambda_gf_g & = 0.\\
    	\end{array}\right.
	\end{equation}
The curve $Y$ is contained in every such $S(\lambda_2,\ldots, \lambda_g)$ and, because all $\lambda_i$ are non-zero, can be defined by just one equation $f_i = 0$ for some $i \in \{1,\ldots, g\}$. In other words, $Y$ is a Cartier divisor in $S(\lambda_2,\ldots, \lambda_g)$. Since every $S(\lambda_2,\ldots, \lambda_g)$ is given by $g-1$ equations in $\C^{g+1}$, the dimension of each of its irreducible components, as well as its own dimension, is at least two. The next proposition shows that for \emph{generic} coefficients $(\lambda_2,\ldots, \lambda_g)$ (i.e., the point $(\lambda_2,\ldots, \lambda_g)$ is contained in the non-empty complement of a specific closed subset of $(\C \setminus \{0\})^{g-1}$), the dimension of the scheme $S(\lambda_2,\ldots, \lambda_g)$ is exactly two. Even more, it is a surface, and we can call it a \emph{generic (embedding) surface} of $Y$. We also prove some extra properties which are needed later on. 

\begin{prop} \label{prop:S-surface}
For generic $(\lambda_2, \ldots, \lambda_g) \in (\C \setminus \{0\})^{g-1}$, the scheme $S(\lambda_2,\ldots, \lambda_g)$ is a normal equidimensional surface which is smooth outside the origin. 
\end{prop}

\begin{proof}
We use the following affine version of Bertini's Theorem, which can be found in~\cite[Cor. 6.7]{J}.

\vspace{14pt}

\noindent \emph{Let $X$ be a smooth equidimensional variety of dimension $m$ and let $f: X \rightarrow \C^n$ be a dominant morphism of $\C$-schemes. Then, for a generic point $\xi \in \C^n$, the inverse image $f^{-1}(\xi)$ is a smooth equidimensional variety of dimension $m-n$.}

\vspace{14pt}

\noindent Consider $X := \C^{g+1} \setminus \bigcup_{i=2}^g\{f_i = 0\}$ and the morphism \[f: X \longrightarrow \C^{g-1}: x \mapsto \left(-\frac{f_1(x)}{f_2(x)},-\frac{f_2(x)}{f_3(x)}, \ldots, -\frac{f_{g-1}(x)}{f_g(x)}\right).\] Clearly, $X$ is a smooth irreducible variety of dimension $g+1$. To check that $f$ is dominant, it is enough to show that its image contains a dense subset of $\C^{g-1}$. Note that for every $\lambda = (\lambda_2, \ldots, \lambda_g) \in (\C \setminus \{0\})^{g-1}$, the inverse image $f^{-1}(\lambda)$ is exactly the scheme $S(\lambda_2,\ldots, \lambda_g)$ without the curve $Y$, which is never empty as $S(\lambda_2,\ldots, \lambda_g)$ is at least two-dimensional. Hence, the image $f(X)$ contains $(\C \setminus \{0\})^{g-1}$, and we can apply the above version of Bertini's Theorem; for generic $(\lambda_2, \ldots, \lambda_g) \in (\C \setminus \{0\})^{g-1}$, the scheme $S(\lambda_2,\ldots, \lambda_g) \setminus Y$ is a smooth equidimensional variety of dimension two. Because all irreducible components of $S(\lambda_2,\ldots, \lambda_g)$ have at least dimension two, it immediately follows that $S(\lambda_2,\ldots, \lambda_g)$ itself is also equidimensional of dimension two. Furthermore, using the Jacobian criterion, one can check that $S(\lambda_2,\ldots, \lambda_g)$ is smooth at every point in $Y \setminus \{0\}.$ These two facts together imply that $S$ is a complete intersection in $\C^{g+1}$ which is regular in codimension one (i.e., its singular locus has codimension at least two). As being regular in codimension one is equivalent to being normal for a complete intersection in $\C^{g+1}$ (see, e.g.,~\cite[Ch. II, Prop. 8.23]{Ha}), we can conclude that $S$ is indeed a normal equidimensional surface which is smooth outside the origin. 
\end{proof}

\begin{remark}\label{rmk:S-irreducible}
It is possible that a generic $S(\lambda_2,\ldots, \lambda_g)$ is irreducible; we did not find an easy argument or counterexample. This will, nevertheless, not have any influence on the results in this article: as $S(\lambda_2,\ldots, \lambda_g)$ is a normal equidimensional surface and smooth outside the origin, its irreducible components are pairwise disjoint surfaces, all smooth except for the single component containing the curve $Y$. Hence, because we are only interested in the behavior of $S(\lambda_2,\ldots, \lambda_g)$ around the curve $Y$, we can, in some sense, only consider the one component containing $Y$ and forget about the other components.
\end{remark}

We will now explain the relation between the monodromy eigenvalues of $Y$ considered in $\C^{g+1}$ and the monodromy eigenvalues of $Y$ considered on a generic surface $S(\lambda_2,\ldots, \lambda_g)$. At several places in this section, we will impose extra conditions on $(\lambda_2,\ldots, \lambda_g)$, but it will still represent a generic point of $(\C\setminus\{0\})^{g-1}$ in the end. To shorten the notation, from now on, we will denote a generic surface by $S$. 

\vspace{14pt}

Let $\varphi: \tilde{X} \rightarrow \C^{g+1}$ be a principalization of $\I$. We can assume that $\varphi$ consists of two parts: 
\begin{enumerate}
	\item[(i)] a composition of blow-ups $\varphi_1: \tilde{X}_1 \rightarrow \C^{g+1}$ above $0$ to desingularize the strict transform of $Y$, and to make it have normal crossings with one exceptional variety and no intersection with all other components of $\varphi_1^{-1}(0)$; and
	\item[(ii)] one last blow-up $\varphi_2: \tilde{X} \rightarrow \tilde{X}_1$ along the strict transform of $Y$ to change it into a locally principal divisor.
\end{enumerate}
The exceptional variety coming from the last blow-up is denoted by $\tilde{E}$ and has numerical data $(1,g)$. The other irreducible components of the total transform $\varphi^{-1}(Y)$ are denoted by $E_j$ for $j \in J$, and their corresponding data by $(N_j,\nu_j)$.  Note that $\tilde{E}$ is mapped surjectively onto $Y$ under $\varphi$ and that $\varphi^{-1}(0) = \cup_{j\in J} E_j$. Let $\sigma: X'  \rightarrow \C^{g+1}$ be the blow-up of $\C^{g+1}$ with center $Y$, let $E'$ be the corresponding exceptional variety, and let $\psi: \tilde{X} \rightarrow X'$ be the unique morphism such that $\sigma \circ \psi = \varphi$. It immediately follows that $\psi$ is a surjective proper birational morphism inducing an isomorphism $\tilde{X} \setminus \varphi^{-1}(Y) \simeq X' \setminus E'$. Because of the specific construction of the principalization, the morphism $\psi$ even induces an isomorphism $\tilde{X} \setminus \cup_{j \in J} E_j \simeq X' \setminus \sigma^{-1}(0)$; indeed, because $Y \setminus \{0\}$ remains unchanged during the first series of blow-ups, both $\sigma$ and $\varphi$ restricted to $\C^{g+1} \setminus \{0\}$ are just the blow-up along $Y \setminus \{0\}$, and they are thus equal up to an isomorphism. Furthermore, $\tilde{E}$ is sent surjectively onto $E'$ under $\psi$, while every other exceptional variety $E_j$ is mapped onto a closed subset of $\sigma^{-1}(0)$.

\vspace{14pt}

With this notation, the zeta function of monodromy associated with $Y \subset \C^{g+1}$ at a point $e \in \sigma^{-1}(0) \subset E'$ is given by 
	\begin{equation} \label{eq:zeta-function-mon-Y}
		Z_{Y,e}^{mon} (t)= \prod_{j \in J}(1-t^{N_j})^{\chi(E^{\circ}_j \cap \psi^{-1}(e))},
	\end{equation}
where $E_j^{\circ} = E_j \backslash \cup_{i \neq j}(E_i \cap E_j)$ for all $j \in J$, see Theorem~\ref{thm:ACampo-Princ}. We will show that this zeta function for a generic point $e \in \sigma^{-1}(0)$ is equal to the zeta function of monodromy at the origin associated with the Cartier divisor $Y$ on a generic surface $S$. 

\vspace{14pt}
 
We begin by considering the strict transform $S' :=  \overline{\sigma^{-1}(S \setminus Y)}$ of $S$ under $\sigma$. By the behavior of a subvariety under a blow-up, the restriction of $\sigma$ to this strict transform is the blow-up of $S$ along the Cartier divisor $Y \subset S$. Consequently, $S'$ is a surface isomorphic to $S$, and $Y' := E' \cap S'$ is a curve on $S'$ isomorphic to $Y$. This can also be deduced from the equations of the blow-up as follows. Because $\I$ is generated by a regular sequence, the blow-up of $\C^{g+1}$ with center $Y = V(\I)$ is given by the projection 
 	\begin{equation} \label{eq:equations-X'}
 	 	\sigma: X' = \text{Proj} \frac{\C[x_0, \ldots, x_g][X_1, \ldots, X_g]}{(f_iX_j - f_jX_i: ~i,j = 1, \ldots, g)} \longrightarrow \C^{g+1},
 	\end{equation}
 see for instance~\cite[Section IV.2]{EH}. In other words, $X'$ is the closed subscheme of $\text{Proj}~ \C[x_0, \ldots, x_g][X_1, \ldots, X_g] \simeq \C^{g+1} \times \P^{g-1}$ defined by the equations $f_iX_j - f_jX_i$ for $i,j = 1, \ldots, g$. The exceptional variety $E'$ is locally on $X_k \neq 0$ given by the principal ideal $(f_k)$ and glues globally to $Y \times \P^{g-1}$. Finally, the strict transform $S'$ is \[\text{Proj}~ \frac{\C[x_0, \ldots, x_g][X_1, \ldots, X_g]}{(f_iX_j - f_jX_i,~ X_k + \lambda_{k+1}X_{k+1} ; ~ i,j = 1, \ldots, g, ~ k = 1, \ldots, g-1)}.\] Since all $\lambda_i$ are non-zero, the system of equations $X_k + \lambda_{k+1}X_{k+1} = 0$ for $ k = 1, \ldots, g-1$ has a unique homogeneous solution, say $P = [p_1: \ldots : p_g] \in \P^{g-1}$. Note that all $p_i \neq 0$ and that $\frac{p_i}{p_{i+1}} = -\lambda_{i+1}$ for $i = 1,\ldots, g-1$. Hence, $S'$ can be rewritten as \[\text{Spec}~ \frac{\C[x_0,\ldots, x_g]}{(f_ip_j - f_jp_i; ~ i, j = 1,\ldots, g)} \times \{P\}  \subseteq \C^{g+1} \times \P^{g-1}.\] Using the relations between the numbers $p_i$, it is easy to see that this is the same as $S \times \{P\}$, so that $S'$ is indeed isomorphic to $S$ under $\sigma$. From this argument, it also follows that $Y' = Y \times \{P\}$ is isomorphic to $Y$.
 
 \vspace{14pt}
 
 The point $P \in \P^{g-1}$ is completely determined by the generic coefficients $(\lambda_2,\ldots, \lambda_g)$ and corresponds to a unique point $p := (0,P) = S' \cap \sigma^{-1}(0)$ on $S'$. We will call $p$ the \textit{generic point} associated with the generic surface $S$. As $\text{Sing}(S) = \text{Sing}(Y) = \{0\}$, we have $\text{Sing}(S') = \text{Sing}(Y') = \{p\}$, and we can use the classical formula~\eqref{eq:ACampo-poly} of A'Campo for the monodromy zeta function $Z^{mon}_{Y',p}(t)$ at $p$ of the Cartier divisor $Y'$ on the surface $S'$. We claim that this zeta function is equal to the monodromy zeta function $Z_{Y,p}^{mon} (t)$ given in~\eqref{eq:zeta-function-mon-Y} at the generic point $p \in \sigma^{-1}(0) \subset E'$. As a direct consequence, the latter zeta function of monodromy is equal to the zeta function of monodromy $Z^{mon}_{Y,0}(t)$ at the origin associated with $Y \subset S$.
 
 \vspace{14pt}

To compute the monodromy zeta function $Z^{mon}_{Y',p}(t)$ with A'Campo's formula, we need an embedded resolution of $Y'$ on $S'$. To construct such a resolution, we consider the strict transform $\tilde{S} := \overline{\varphi^{-1}(S \setminus Y)}$ of $S$ under the principalization $\varphi$, and we put $\tilde{Y} := \tilde E \cap \tilde S$. 

\begin{lemma}\label{lemma:tildeS-surface}
For generic $(\lambda_2,\ldots, \lambda_g) \in (\C\setminus \{0\})^{g-1}$, the strict transform $\tilde S$ of $S$ under $\varphi$ is a smooth equidimensional surface.
\end{lemma}

\begin{proof}
We first determine the local defining equations of $\tilde{S}$. After the principalization $\varphi$, the ideal $\I = (f_1, \ldots, f_g)$ is transformed into the locally principal ideal $\varphi^{\ast}\I = (f^*_1, \ldots, f^*_g)$ with $f_i^{\ast} = f_i \circ \varphi$ for $i =1,\ldots, g$. This means that in every point $x \in \tilde{X}$, we have local coordinates $y = (y_0,\ldots, y_g)$ such that $(f^*_1(y), \ldots, f^*_g(y)) = (h(y))$ for some generator $h(y)$. Then, on the one hand, there exist regular functions $\tilde f_1(y), \ldots, \tilde f_g(y)$ such that $f_i^*(y) = \tilde f_i(y)h(y)$ for all $i = 1, \ldots, g$, and, on the other hand, there exist regular functions $h_1(y), \ldots, h_g(y)$ such that $h(y) = \sum_{i=1}^g h_i(y)f^*_i(y)$. We can deduce that $1 = \sum_{i=1}^gh_i(y)\tilde f_i(y)$, and, in particular, that $\tilde f_1(y), \ldots, \tilde f_g(y)$ do not have common zeros. In addition, it follows that $\tilde S$ is locally given by equations of the form $\tilde f_1(y) + \lambda_2\tilde f_2(y) = \cdots = \tilde f_{g-1}(y) +\lambda_g \tilde f_g(y) = 0$, where the $\tilde{f}_i(y)$ have no common zeros. Now, locally around each point $x \in \tilde S$ in the smooth irreducible ${(g+1)}$-dimensional variety $\tilde X$, we can repeat the proof of Proposition~\ref{prop:S-surface} to conclude that $\tilde S \setminus \bigcup_{i=2}^g\{\tilde f_i(y) = 0\}$ for generic $(\lambda_2,\ldots, \lambda_g) \in (\C\setminus \{0\})^{g-1}$ is a smooth equidimensional variety of dimension two. Because the set $\bigcup_{i=2}^g\{\tilde f_i(y) = 0\}$ on $\tilde S$ is equal to the empty set of common zeros $\{\tilde f_1(y) = \tilde f_2(y) = \cdots = \tilde f_g(y) = 0\}$, we indeed found that $\tilde S$ is a smooth equidimensional surface for generic coefficients $(\lambda_2,\ldots, \lambda_g)$.
\end{proof}

\begin{remark}\label{rmk:tildeS-irreducible}
It is again not important whether $\tilde S$ for generic $(\lambda_2,\ldots, \lambda_g)$ is irreducible, cf. Remark~\ref{rmk:S-irreducible}. Even more, the surface $\tilde S$ is irreducible if and only if $S$ is. It is, however, important that there is only one component of $\tilde S$ which intersects $\psi^{-1}(Y')$.
\end{remark}

Throughout the rest of this section, we assume that the coefficients $(\lambda_2,\ldots, \lambda_g)$ are generic in $(\C \setminus \{0\})^{g-1}$ such that $S$ and $\tilde S$ satisfy the properties of Proposition~\ref{prop:S-surface} and Lemma~\ref{lemma:tildeS-surface}, respectively. To recapitulate, we visualize all morphisms and varieties in the following diagram: 

\begin{figure}[htb!] 
	\centering
	\includegraphics[width=\textwidth]{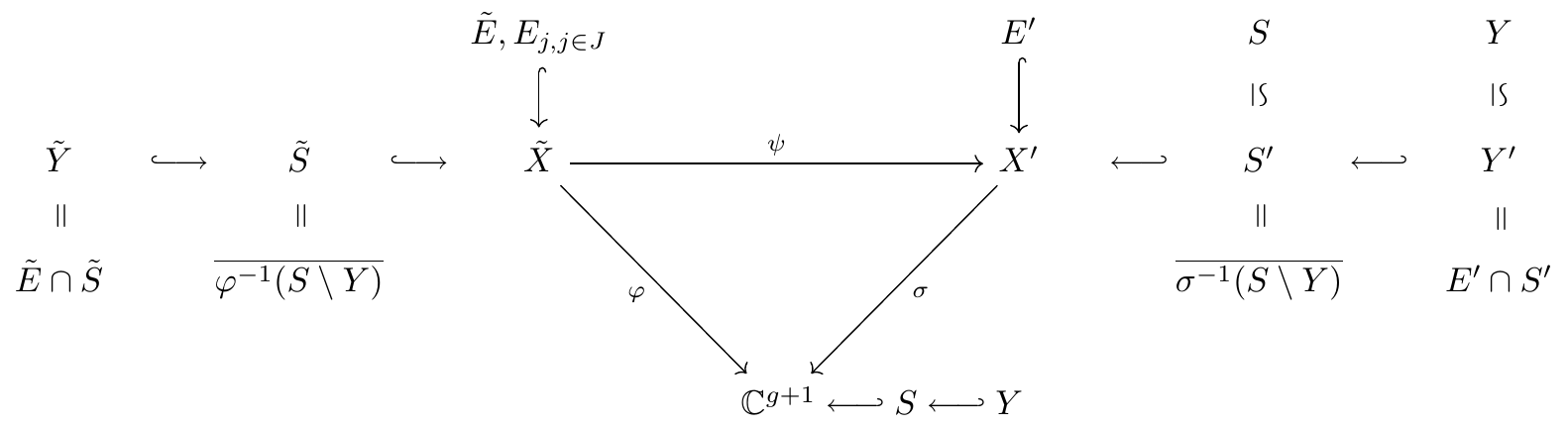}
\end{figure}

We will show that, under some extra conditions on $(\lambda_2,\ldots, \lambda_g)$, the restriction ${\rho: \tilde S \rightarrow S'}$ of $\psi$ to $\tilde S$ is an embedded resolution of $Y'$ on $S'$. Recall that every $E_j$ for $j \in J$ is mapped onto a closed subset of $\sigma^{-1}(0)$ under $\psi$. Let $J_1 \subset J$ be the set of indices $j \in J$ such that $E_j$ is mapped surjectively onto $\sigma^{-1}(0)$. Note that $J_1 \neq \emptyset$: the second last exceptional variety $E_k$ of $\varphi$, which is the only one intersecting $\tilde E$, will always be mapped surjectively onto $\sigma^{-1}(0)$ since $\tilde E$ is mapped surjectively onto $E'$ and $\tilde E \setminus (\tilde E \cap E_k) \simeq E' \setminus \sigma^{-1}(0)$. Then, every $E_j$ for $j \in J_2 := J \setminus J_1$ is mapped onto a proper closed subset $\psi(E_j)$ of $\sigma^{-1}(0) \simeq \P^{g-1}$, and the set $\sigma^{-1}(0)\setminus \cup_{j \in J_2} \psi(E_j)$ is non-empty. The next result tells us, among others, that for a generic surface $S$ corresponding to a generic point $p$ in the latter set, the surface $\tilde S$ is equal to $\psi^{-1}(S')$. This implies that the map $\rho: \tilde S \rightarrow S'$ is a well-defined proper surjective morphism from a smooth surface $\tilde S$ to $S'$, or thus, that $\rho$ is a good candidate for an embedded resolution of $Y'$ on $S'$. 

\begin{lemma} \label{lemma:generic-p} For a generic point $p \in \sigma^{-1}(0)\setminus \cup_{j \in J_2} \psi(E_j)$, we have that
\begin{enumerate} [wide, labelindent=0pt]
	\item[(i)] for all $j \in J_1$, the inverse image $\psi_j^{-1}(p)$ of $p$ under $\psi_j: E_j \rightarrow \sigma^{-1}(0)$ is smooth and equidimensional of dimension one; and
	\item[(ii)] the total inverse image $\psi^{-1}(p)$ of $p$ under $\psi: \tilde X \rightarrow X'$ is connected and equidimensional of dimension one. 
\end{enumerate}
Furthermore, for each surface $S$ corresponding to such a generic point $p$, the strict transform $\tilde S$ of $S$ under $\varphi$ is equal to $\psi^{-1}(S')$.
\end{lemma}

\begin{proof}
To prove items (i) and (ii), we will again apply a kind of Bertini's Theorem; this time, we use the following projective version obtained from~\cite[Cor. 6.11]{J}.

\vspace{14pt}

\noindent \emph{Let $X$ be a complex scheme of finite type which is equidimensional of dimension $m$, and let $f: X \rightarrow \P^n$ be a dominant morphism of $\C$-schemes. Then, for a generic point $\xi \in \P^n$, the inverse image $f^{-1}(\xi)$ is equidimensional of dimension $m-n$. If $X$ is in addition smooth, then the inverse image $f^{-1}(\xi)$ for a generic point $\xi$ is also smooth.}

\vspace{14pt}

The statement in (i) for each $j \in J_1$ follows immediately from this version of Bertini's Theorem applied to the surjective morphism $\psi_j: E_j \rightarrow \sigma^{-1}(0) \simeq \P^{g-1}$, where $E_j$ is a smooth irreducible hypersurface in $\tilde X$ of dimension $g$. For (ii), we consider the surjective morphism $\psi: \varphi^{-1}(0) \rightarrow \sigma^{-1}(0)$. As the irreducible components of $\varphi^{-1}(0)$ are the $g$-dimensional exceptional varieties $E_j$ for $j \in J$, this version of Bertini tells us that $\psi^{-1}(p)$ for a generic point $p$ is equidimensional of dimension one. To show the connectedness, we make use of Zariski's Main Theorem stating that a proper birational morphism ${f: X \rightarrow X'}$ between irreducible varieties with $X'$ normal has connected fibers. From the equations~\eqref{eq:equations-X'} of $X'$, it is easy to see that $X'$ is locally a complete intersection in $\C^{g+1}\times \P^{g-1}$. In fact, the blow-up of an affine space $\C^n$ along any subscheme defined by a regular sequence is a local complete intersection. Because $Y \setminus \{0\}$ is smooth, we know that $X' \setminus \sigma^{-1}(0)$ is smooth. Therefore, $X'$ is a local complete intersection in $\C^{g+1} \times \P^{g-1}$ which is regular in codimension one, and we can conclude that $X'$ is normal (see, e.g.,~\cite[Ch. II, Prop. 8.23]{Ha}). Hence, Zariski's Main Theorem for the proper birational morphism $\psi: \tilde X \rightarrow X'$ assures that every fiber is connected. In particular, the fiber of a generic point $p \in \sigma^{-1}(0)\setminus \cup_{j \in J_2} \psi(E_j)$ is connected, which ends the proof of (ii).

\vspace{14pt}

Let $S$ be a generic surface corresponding to such a generic point $p$. To show that $\tilde S = \psi^{-1}(S')$, we first rewrite $\tilde S = \overline{\varphi^{-1}(S \setminus Y)}$ as follows: \[\tilde S = \overline{\psi^{-1}(S' \setminus Y')} = \overline{\psi^{-1}(S' \setminus \{p\})}.\] The first equality immediately comes from the fact that $S'\setminus Y' = \sigma^{-1}(S \setminus Y)$ by the properties of the blow-up, together with the commutativity of the above diagram. The second equality can be seen from the next small argument. It is trivial that $\overline{\psi^{-1}(S'\setminus Y')}\subset \overline{\psi^{-1}(S' \setminus \{p\})}$. For the other inclusion, we remark that the closure of $S'\setminus Y'$ in  $X' \setminus \sigma^{-1}(0)$ is equal to $S'\setminus \{p\}$. Since $\psi$ induces an isomorphism $\tilde{X} \setminus \cup_{j\in J} E_j \simeq X' \setminus \sigma^{-1}(0)$, this implies that the closure of $\psi^{-1}(S' \setminus Y')$ in $\tilde X \setminus \cup_{j\in J} E_j$ must be equal to $\psi^{-1}(S'\setminus\{p\})$, which in turn implies the reverse inclusion $\overline{\psi^{-1}(S' \setminus \{p\})} \subset \overline{\psi^{-1}(S' \setminus Y')}$. The inclusion $\tilde S \subset \psi^{-1}(S')$ follows now easily from the continuity of $\psi$: \[\tilde S = \overline{\psi^{-1}(S' \setminus Y')} \subseteq \psi^{-1}(\overline{S' \setminus Y'}) =  \psi^{-1}(S').\] Using the third description of $\tilde S$ and the fact that $\psi$ is an isomorphism above $X' \setminus \sigma^{-1}(0)$, one can see that $\tilde S \setminus \psi^{-1}(p) = \psi^{-1}(S' \setminus \{p\})$. Hence, it remains to show that $\psi^{-1}(p) \subset \tilde S$. We do this in three steps.

\vspace{14pt}

First, we show that $\psi^{-1}(p) \cap \tilde S \neq \emptyset$. To this end, it is enough to show that $\overline{\psi^{-1}(Y'\setminus \{p\})}$ is not equal to $\psi^{-1}(Y'\setminus \{p\})$; indeed, both sets are contained in $\tilde S$, and the complement $\overline{\psi^{-1}(Y'\setminus \{p\})} \setminus \psi^{-1}(Y'\setminus \{p\})$ is contained in $\psi^{-1}(p)$ since $\psi$ is an isomorphism outside $\varphi^{-1}(0)$ and $\sigma^{-1}(0)$. Suppose that $\overline{\psi^{-1}(Y'\setminus \{p\})} = \psi^{-1}(Y'\setminus \{p\})$, or in other words, that $\psi^{-1}(Y'\setminus \{p\})$ is closed in $\tilde X$.  Then, the restriction $\psi\vert_{\psi^{-1}(Y'\setminus \{p\})}: \psi^{-1}(Y'\setminus \{p\})\rightarrow Y'$ of $\psi$ is proper so that $Y'\setminus \{p\} = \psi(\psi^{-1}(Y'\setminus \{p\}))$ is closed in $Y'$. This is a contradiction. Second, let $A$ be an irreducible component of $\psi^{-1}(p)$ such that $A \cap \tilde S \neq \emptyset$. We prove that $A$ is contained in $\tilde S$. Because $A \subset \psi^{-1}(p) \subset \cup_{j\in J_1} E_j$ is irreducible, there exists a component $E_j$ with $j \in J_1$ such that $A \subset E_j$. Then, the intersection $E_j \cap \tilde S$ is non-empty, and there exists an irreducible component $B$ of $E_j \cap \tilde S$ such that $A\cap B \neq \emptyset$. Note that both $A$ and $B$ are contained in $E_j \cap \psi^{-1}(p)= \psi_j^{-1}(p)$. We claim that they are also both irreducible components of $\psi_j^{-1}(p)$. Because $\psi_j^{-1}(p)$ is equidimensional of dimension one by (i), it is enough to show that $A$ and $B$ are one-dimensional. For $A$, this is trivial as it is an irreducible component of $\psi^{-1}(p)$. For $B$, this follows from the general intersection theory in the smooth $(g+1)$-dimensional variety $\tilde X$: the single component of $\tilde S$ that intersects $E_j$ (see Remark~\ref{rmk:tildeS-irreducible}) is two-dimensional and not contained in $E_j$. Hence, every irreducible component of the intersection of the surface $\tilde S$ and the hypersurface $E_j$ is one-dimensional. We thus found that $A$ and $B$ are irreducible components of $\psi_j^{-1}(p)$ that are intersecting. Because $\psi_j^{-1}(p)$ is smooth, this is only possible if $A = B$ is contained in $\tilde S$. Finally, as $\psi^{-1}(p)$ is connected, the whole of $\psi^{-1}(p)$ must be contained in $\tilde S$. 
\end{proof}

As a generic condition on the point $p \in \sigma^{-1}(0)$ translates into a generic condition on $(\lambda_2,\ldots, \lambda_g) \in (\C \setminus \{0\})^{g-1}$, we can rephrase Lemma~\ref{lemma:generic-p} in terms of generic $(\lambda_2,\ldots, \lambda_g)$, and consider $S$ and its strict transforms $S'$ and $\tilde S$ corresponding to such coefficients. In the next proposition, we show that $\rho: \tilde S \rightarrow S'$ is indeed an embedded resolution of $Y'$ on $S'$. We also determine the exceptional varieties and the part of their numerical data appearing in the formula of A'Campo.

\begin{prop} \label{prop:emb-res}
For generic $(\lambda_2,\ldots, \lambda_g) \in (\C \setminus \{0\})^{g-1}$, the restriction $\rho: \tilde S \rightarrow S'$ of $\psi$ to $\tilde S$ is an embedded resolution of $Y'$ on $S'$. The strict transform of $Y'$ is $\tilde Y$, and the exceptional varieties are the irreducible components of $E_j \cap \tilde S$ for $j \in J_1$. Furthermore, the pull-back of $Y'$ seen as a Cartier divisor on $S'$ is given by \[ \rho^{\ast}Y' = \tilde Y + \sum_{j\in J_1}N_j (E_j\cap \tilde S),\] which yields (the needed) part of the numerical data associated with this resolution. 
\end{prop}

\begin{proof}
The previous lemma already implies that $\rho: \tilde S \rightarrow S'$ is a well-defined surjective proper birational morphism from the smooth surface $\tilde S$ to $S'$. Additionally, $\rho$ induces an isomorphism $\tilde S \setminus \rho^{-1}(Y) \simeq S' \setminus Y'$: even more, because $\psi$ is an isomorphism above $X' \setminus \sigma^{-1}(0)$, its restriction $\rho$ gives an isomorphism $\tilde S \setminus \psi^{-1}(p) = \psi^{-1}(S'\setminus \{p\}) \simeq S'\setminus \{p\}$. The first equality follows from the third description of $\tilde S$ in the proof of Lemma~\ref{lemma:generic-p}. From the same lemma, we know that $E_j \cap \rho^{-1}(p) =  E_j \cap \tilde S$ for every $j \in J_1$, and that $E_j \cap \rho^{-1}(p)= \emptyset$ for $j \in J_2$. In other words, we have that $\rho^{-1}(p) = \cup_{j\in J_1}(E_j \cap \tilde S)$ or, thus, the irreducible components of $E_j \cap \tilde S$ for $j \in J_1$ are indeed the exceptional varieties of $\rho$. To show that $\tilde Y = \tilde E \cap \tilde S$ is the strict transform $\overline{\rho^{-1}(Y' \setminus \{p\}})$ of $Y'$ under $\rho$, we first remark that $Y' \setminus \{p\} \simeq \tilde Y \setminus \rho^{-1}(p) = (\tilde E \cap \tilde S) \setminus (\tilde E \cap E_k \cap \tilde S)$, where $E_k$ denotes the second last exceptional variety of $\varphi$, which is the only one intersecting $\tilde E$. Similarly as in Lemma~\ref{lemma:generic-p}, one can see that every irreducible component of $\tilde E \cap \tilde S$ is one-dimensional. Therefore, it suffices to show that $\tilde E \cap E_k \cap \tilde S$ only consists of a finite number of points. To this end, we recall the specific construction of the principalization $\varphi$ and let $\check{E}_k$ be the last exceptional variety of the first part $\varphi_1$, of which $E_k$ is the strict transform under the last blow-up $\varphi_2$. By the properties of the blow-up, we know that the restriction $\varphi_2\vert_{E_k}: E_k \rightarrow \check{E}_k$ is the blow-up of $\check{E}_k$ along its intersection with the strict transform of $Y$ under $\varphi_1$. As the latter intersection consists of a single point, the exceptional divisor of this blow-up is given by $\tilde E \cap E_k \simeq \P^{g-1}$. It follows that each fiber of the surjective morphism $\psi\vert_{\tilde E \cap E_k}: \tilde E \cap E_k \simeq \P^{g-1} \rightarrow \sigma^{-1}(0) \simeq \P^{g-1}$ is finite. In particular, we find that $\tilde E \cap E_k \cap \psi^{-1}(p) = \tilde E \cap E_k \cap \tilde S$ consists of a finite number of points. Finally, for the last claim, we consider the commutative diagram
\begin{figure}[h!]
	\centering
	\begin{tikzcd}
		\tilde{X} \arrow{r}{\psi} & X' \\
		\tilde{S} \arrow[hookrightarrow]{u} \arrow{r}{\rho}
		& S' \arrow[hookrightarrow]{u}.
	\end{tikzcd}
\end{figure} \\

From the properties of the pull-back, we know that $\rho^{\ast}Y' = \rho^{\ast}(E'\vert_{S'}) = (\psi^{\ast}E')\vert_{\tilde{S}}$. Because the inverse images $\psi^{-1}(E')$ and $\varphi^{-1}(Y)$ are equal, the pull-back of the Cartier divisor $E'$ is \[\psi^{\ast}E' = \tilde{E} + \sum_{j\in J}N_jE_j.\] Then, indeed, \[\rho^{\ast}Y' =  \tilde E\vert_{\tilde S} + \sum_{j\in J}N_jE_j\vert_{\tilde S} = \tilde Y + \sum_{j\in J_1} N_j(E_j \cap \tilde S),\] where we used that $E_j \cap \tilde S = E_j \cap \rho^{-1}(p) = \emptyset$ for $j \notin J_1$.
\end{proof}

We are now ready to apply A'Campo's formula for the monodromy zeta function $Z^{mon}_{Y',p}(t)$ of $Y' \subset S'$, and to show the main result of this section.  

\begin{theorem}\label{thm:red-to-on-curve-surface}
Consider a complete intersection curve $Y = V(\I) \subset \C^{g+1}$ whose ideal $\I = (f_1, \ldots, f_g)$ is generated by a regular sequence $f_1, \ldots, f_g \in \C[x_0, \ldots, x_g]$, and whose singular set is $\text{Sing}(Y) = \{0\}$. Let $S = S(\lambda_2,\ldots, \lambda_g)$ be a generic embedding surface of $Y$ defined by the equations~\eqref{eq:equations-S}, where the coefficients $(\lambda_2,\ldots, \lambda_g) \in (\C \setminus \{0\})^{g-1}$ are generic such that all previous results hold. Denote by $\sigma: X' \rightarrow \C^{g+1}$ the blow-up of $\C^{g+1}$ with center $Y$ and by $S'$ the strict transform of $S$ under $\sigma$. Then, the monodromy zeta function $Z^{mon}_{Y,p}(t)$ of $Y$ considered in $\C^{g+1}$ at the generic point $p = S' \cap \sigma^{-1}(0)$ is equal to the monodromy zeta function $Z^{mon}_{Y,0}(t)$ of $Y$ considered as a Cartier divisor on $S$ at the origin. Therefore, we refer to both zeta functions as the monodromy zeta function of $Y$. 
\end{theorem}

\begin{proof}
Let $E_k$ be the second last exceptional variety of the principalization $\varphi$ or, thus, the only one intersecting $\tilde E$. Then, the formula~\eqref{eq:ACampo-poly} of A'Campo with the embedded resolution $\rho: \tilde S \rightarrow S'$ of $Y'\subset S'$ from Proposition~\ref{prop:emb-res} gives \[Z^{mon}_{Y',p}(t) = \prod_{j\in J_1}(1-t^{N_j})^{\chi((E_j \cap \tilde S)^{\circ} \cap \rho^{-1}(p))} = \prod_{j\in J_1}(1-t^{N_j})^{\chi((E_j \cap \tilde S)^{\circ})},\] where 
\[(E_j \cap \tilde S)^{\circ} 
	= \left\{\begin{array}{ll} (E_j \cap \tilde S) \setminus \cup_{i\neq j} (E_i\cap E_j \cap \tilde S) 
	& \text{for } j \neq k \\(E_k \cap \tilde S) \setminus (\cup_{i\neq k} (E_i\cap E_k \cap \tilde S) \cup (\tilde E \cap E_k \cap \tilde S)) & \text{for } j = k.
\end{array}\right.\] 
By the choice of the generic point $p \in \sigma^{-1}(0) \setminus \cup_{j\in J_2}\psi(E_j)$ satisfying $E_j \cap \psi^{-1}(p) = E_j \cap \tilde S$ for $j \in J_1$, this is the same as the monodromy zeta function $Z^{mon}_{Y,p}(t)$ given in~\eqref{eq:zeta-function-mon-Y}. Because $0 \in Y \subset S$ is isomorphic to $p \in Y' \subset S'$ under $\sigma$, the theorem follows. 
\end{proof}


\section{Embedded \texorpdfstring{$\Q$}{Q}-resolution of a space monomial curve} \label{Resolution}

The purpose of this section is to construct an embedded $\Q$-resolution of a space monomial curve $Y$ considered as a Cartier divisor on a generic surface $S \subset \C^{g+1}$ with $g\geq 2$ satisfying all results in Section~\ref{RedCurveSurface}. We will also describe the combinatorics of the exceptional divisor that are needed to compute the monodromy zeta function of $Y$ in Section~\ref{ZetaFunction}.

\vspace{14pt}

Our method requires $g$ steps, denoted by Step $k$ for $k = 1,\ldots,g$, consisting of a weighted blow-up in higher dimension. Roughly speaking, in every step, we are able to eliminate one equation in $Y$ and $S$, and to lower the dimension of the ambient space by one. Therefore, the last step coincides with the resolution of a cusp in a Hirzebruch-Jung singularity of type $\frac{1}{d}(1,q)$. We will see that the resolution graph obtained in this process is a tree as in Figure~\ref{fig:dual-graph}, but that the exceptional varieties do not have zero genus in general. The latter implies that the link of the surface singularity $(S,0)$ is not always a rational nor an integral homology sphere. However, using this embedded $\Q$-resolution, one can obtain necessary and sufficient conditions for the link of $(S,0)$ to be a rational or integral homology sphere, see~\cite{MV}.

\subsection{Technical results}

We extract some results from the main construction that are interesting in their own right and discuss them in this section separately. 

\vspace{14pt}

A first challenge in the resolution will be to investigate the irreducible components of the exceptional divisor in each weighted blow-up. We will see that in Step $k$ for $k = 1,\ldots, g$, the exceptional divisor $\E_k$ can be described by a similar system of equations in the quotient of a weighted projective space $\P^r_{\omega}/\mu_{\d}$ that arises as the exceptional divisor of the ambient space. Except from the number of irreducible components, we are also interested in the singular points of $\E_k$, which lie on the coordinate hyperplanes $\{x_i = 0\}$ of $\P^r_{\omega}/\mu_{\d}$. Since our exceptional divisors will always have one common intersection point $A_k$ with the coordinate hyperplanes for $i = 2,\ldots, r$, we restrict in the following proposition to that case. In fact, the single intersection point $A_k = \E_k \cap \{x_i = 0\}$ for $i = 2,\ldots, r$ will be the center of the blow-up in the next step.

\begin{prop}\label{prop:number-of-comp}
Consider the quotient $\P^r_{(p_0, \ldots, p_r)}(d;a_0, \ldots, a_r) = \P^r_{(p_0, \ldots, p_r)}/\mu_{d}$ of some weighted projective space $\P^r_{(p_0, \ldots, p_r)}$ under an action of type $(d;a_0, \ldots, a_r)$ with $r\geq 2$. Let $\E$ be defined in this space by a system of equations 
	\[\left \{ \begin{array}{clccccl}
		x_0^{m_0} & + & x_1^{m_1}  &+& x_2^{m_2} &=& 0 \\
 		&  & x_2^{m_2}  &+& x_3^{m_3} &=& 0  \\
		& & & \vdots & & & \\
		& & x_{r-1}^{m_{r-1}}   &+& x_r^{m_r} &=& 0
	\end{array}\right. \]
for positive integers $m_i$ such that $d \mid a_im_i$ for $i = 0 ,\ldots, r$, and such that each equation is weighted homogeneous with respect to the weights $(p_0, \ldots, p_r)$. Assume that the intersection of $\E$ with $\{x_i = 0\}$ for $i = 2,\ldots, r$ only consists of one fixed point $A$, and that $a_ip_j - a_jp_i = 0$ for all $i,j \in \{2, \ldots, r\}$. Put $P := \prod_{i=2}^r p_i$, and $Q := a_i\prod_{j=2,j\neq i}^rp_j$ for $i = 2,\ldots, r$. Then, 
\begin{enumerate}[wide, labelindent=0pt]
    \item[(i)] \label{item:1-irred-comp} the number of irreducible components of $\E$ is equal to \[\frac{m_2\cdots m_r}{\lcm(m_2,\ldots, m_r)};\]  
    \item[(ii)] all irreducible components of $\E$ have the point $A$ in common and are pairwise disjoint outside $A$; and
    \item[(iii)] each irreducible component has\[\frac{m_1\cdot \gcd\big(dP\cdot(p_1,p_2,\ldots, p_r),(a_1P - p_1Q)\cdot(p_2,\ldots, p_r)\big)}{dP\cdot \gcd(p_2,\ldots, p_r)}\] intersections with $\{x_0 = 0\}$, and \[\frac{m_0\cdot \gcd\big(dP\cdot(p_0,p_2,\ldots, p_r),(a_0P-p_0Q)\cdot(p_2,\ldots, p_r)\big)}{dP\cdot \gcd(p_2,\ldots, p_r)}\] intersections with $\{x_1 = 0\}$. 	
\end{enumerate}
\end{prop}

Computing the numbers in (i) and (iii) relies on counting the number of solutions of a system of polynomial equations in a cyclic quotient space such as in the next result.

\begin{lemma}\label{lemma:number-of-solutions}
Let $X$ be a cyclic quotient space $X(d;a_0, \ldots, a_r)$ with $r \geq 0$ and let $k_i$ for $i = 0, \ldots, r$ be
positive integers such that $d \mid a_ik_i$ for every $i = 0, \ldots, r$. Consider the system of equations
	\[\left\{\begin{array}{ccl}
		x_0^{k_0} & =& c_0 \\
		x_1^{k_1} & =& c_1 \\
		& \vdots & \\
		x_r^{k_r} & =& c_r,
	\end{array}\right.\]
where $c_i \in \C \setminus \{0\}$. If $r \geq 1$, then the number of solutions in $X$ of the form $[(x_0,b_1,\ldots, b_r)]$ with $[(b_1,\ldots, b_r)] \in X(d;a_1,\ldots, a_r)$ fixed is equal to \[\frac{k_0\cdot \gcd(d,a_0, \ldots, a_r)}{\gcd(d,a_1,\ldots, a_r)}.\] The total number of solutions for $r \geq 0$ is equal to \[\frac{k_0\cdots k_r\cdot\gcd(d,a_0, \ldots, a_r)}{d}.\]
\end{lemma}

\begin{proof}
For $r \geq 1$, the solutions with $[(b_1, \ldots, b_r)] \in X(d;a_1,\ldots, a_r)$ fixed can be written as $[(\xi b_0,b_1,\ldots, b_r)]$ for some fixed $k_0$th root $b_0$ of $c_0$ and varying $\xi \in \mu_{k_0}$. Two elements $\xi$ and $\xi'$ in $\mu_{k_0}$ yield the same solution if and only if there exists a $d$th root $\eta \in \mu_d$ such that $\xi b_0 = \eta^{a_0}\xi'b_0$ and $b_i = \eta^{a_i} b_i$ for $i = 1,\ldots, r$ or, thus, if and only if there exists an element $\eta \in \mu_d \cap \mu_{a_1} \cap \cdots \cap \mu_{a_r} = \mu_{\gcd(d,a_1,\ldots, a_r)}$ such that $\xi \xi'^{-1}= \eta^{a_0}$. It follows that the solutions of the above form are in bijection with $\mu_{k_0}/\text{Im}\,h$ where $h$ is the well-defined group homomorphism $h: \mu_{\gcd(d,a_1,\ldots, a_r)} \rightarrow \mu_{k_0}$ given by $\eta \mapsto \eta^{a_0}.$ As $\text{Im}\,h$ is isomorphic to $\mu_{\gcd(d,a_1,\ldots, a_r)}/\text{Ker}\,h$ and $\text{Ker}\,h = \mu_{\gcd(d,a_0,\ldots, a_r)}$, we obtain the right number of solutions. The total number of solutions for $r \geq 0$ can be shown by an induction argument, using the first part of the lemma in the induction step.
\end{proof}

\begin{proof}[Proof of Proposition~\ref{prop:number-of-comp}]
We start with the case where $r \geq 3$ and we determine the irreducible components of $\E$ by first identifying the irreducible components of $\E\setminus \{A\}$. To find the components of $\E \setminus \{A\}$, we consider the chart of $\P^r_{(p_0, \ldots, p_r)}(d;a_0, \ldots, a_r)$ where $x_2 \neq 0$ which is given by \[X\left(\begin{array}{c|ccccc} 
		p_2 & p_0 & p_1 &  p_3 & \ldots & p_r \\
		dp_2& A_0 & A_1  & 0 & \ldots & 0
		\end{array} \right),\] 
with $A_0 = a_0p_2 - a_2p_0$ and $A_1 = a_1p_2 - a_2p_1$, see~\eqref{eq:chart-weighted-proj-space-with-action}. On this chart, the equations of $\E$ become 
	\[\left \{ \begin{array}{clccccl}
    	x_0^{m_0} & + & x_1^{m_1}  &+& 1 &=& 0 \\
 		&  & 1  &+& x_3^{m_3} &=& 0  \\
		& & & \vdots & & & \\
		& & x_{r-1}^{m_{r-1}}   &+& x_r^{m_r} &=& 0.
	\end{array}\right.\]
For a fixed solution $b = [(b_3,\ldots, b_r)]$ in $X(p_2;p_3,\ldots, p_r)$ of the last $r-2$ equations, we denote by $\E_b$ the set $\{[(x_0,x_1,b_3,\ldots, b_r) ] \mid x_0^{m_0} + x_1^{m_1} + 1 = 0\}$. By the second part of Lemma~\ref{lemma:number-of-solutions}, the number of such solutions $b \in X(p_2;p_3,\ldots, p_r)$ is given by 
	\begin{equation}\label{eq:number-of-comp}
    	\frac{m_3\cdots m_r\cdot \gcd(p_2, \ldots, p_r)}{p_2}.
	\end{equation}
It is not hard to see that every $\E_b$ is irreducible and that all these sets are pairwise disjoint. In other words, the irreducible components of $\E \setminus \{A\}$ are the sets $\E_b$ for each solution $b \in X(p_2;p_3,\ldots, p_r)$ of $1 + x_3^{m_3} = \cdots = x_{r-1}^{m_{r-1}} + x_r^{m_r} = 0$. One can also show that $A$ is contained in each closure $\overline{\E}_b$ in $\P^r_{(p_0, \ldots, p_r)}(d;a_0, \ldots, a_r)$ or, thus, that all $\overline{\E}_b = \E_b \cup \{A\}$ are the irreducible components of $\E$. Hence, the number of components of $\E$ is given by~\eqref{eq:number-of-comp}, which can be rewritten as the expression in the proposition by using the relation~\eqref{eq:rel-gcd-lcm}. Furthermore, all $\overline{\E}_b$ contain the point $A$ and are pairwise disjoint outside $A$, proving (ii). To show the last part of the proposition, we still work on the chart where $x_2 \neq 0$: the point $A$ is not contained in the intersection $\E \cap \{x_i = 0\}$ for $i = 0,1$. We thus need to compute the number of intersections of each component $\E_b = \{[(x_0,x_1,b_3,\ldots, b_r) ] \mid x_0^{m_0} + x_1^{m_1} + 1 = 0\}$ with $\{x_0 = 0\}$ and $\{x_1 = 0\}$. For the first intersection, this reduces to counting the number of points in $X \left(\begin{smallmatrix} p_2 \\ dp_2\end{smallmatrix} \middle| \begin{smallmatrix}  p_1 & p_3 & \cdots & p_r \\  A_1 & 0 & \cdots & 0 \end{smallmatrix} \right)$ of the form $[(x_1,b_3,\ldots, b_r)]$ with $x_1^{m_1} + 1= 0$ and $[(b_3,\ldots, b_r)]$ a fixed solution of $1 + x_3^{m_3} = \cdots = x_{r-1}^{m_{r-1}} + x_r^{m_r} = 0$ in $X(p_2;p_3,\ldots, p_r)$. This can be further simplified with the isomorphism, see Example~\ref{ex:quotient-space}, 
	\begin{equation}\label{eq:isom-to-one-line}
   		X\left(\begin{array}{c|cccc} 
		p_2 & p_1 &  p_3 & \ldots & p_r \\
		dp_2 & A_1  & 0 & \ldots & 0
		\end{array} \right)
		\simeq X \left(p_2; \frac{dp_1p_2}{\gcd(dp_2,A_1)}, p_3, \ldots, p_r\right) 
	\end{equation} 
defined by $[(x_1,x_3,\ldots, x_r)]  \mapsto [(x_1^{\frac{dp_2}{\gcd(dp_2,A_1)}},x_3,\ldots, x_r)]$ to counting the number of points in $X\big(p_2; \frac{dp_1p_2}{\gcd(dp_2,A_1)}, p_3, \ldots, p_r\big)$ of the form $[(x_1,b_3,\ldots, b_r)]$ with $x_1^{\frac{m_1\gcd(dp_2,A_1)}{dp_2}} + 1 = 0$ and $[(b_3,\ldots, b_r)]$ a fixed solution of $1 + x_3^{m_3} = \cdots = x_{r-1}^{m_{r-1}} + x_r^{m_r} = 0$ in $X(p_2;p_3,\ldots, p_r)$. By the first part of Lemma \ref{lemma:number-of-solutions}, this number is given by 
	\begin{equation}\label{eq:intersections-x0}
		\frac{m_1\cdot \gcd\big(dp_2\cdot(p_1,p_2,\ldots, p_r),(a_1p_2 - a_2p_1)\cdot(p_2,\ldots, p_r)\big)}{dp_2\cdot \gcd(p_2,\ldots, p_r)},\
	\end{equation}
which is equal to the expression in the proposition. Analogously, one can show that the number of intersections of each component with $\{x_1 = 0\}$ is given by
	\begin{equation}\label{eq:intersections-x1}
    	\frac{m_0\cdot \gcd\big(dp_2\cdot(p_0,p_2,\ldots, p_r),(a_0p_2 - a_2p_0)\cdot(p_2,\ldots, p_r)\big)} {dp_2\cdot \gcd(p_2,\ldots, p_r)}.
	\end{equation}
If $r = 2$, then $\E \subset \P^r_{(p_0,p_1,p_2)}(d;a_0,a_1,a_2)$ given by the single equation $x_0^{m_0} + x_1^{m_1} + x_2^{m_2} = 0$ is irreducible, showing items (i) and (ii). The number of intersections with $\{x_0 = 0\}$ and $\{x_1 = 0\}$ can be shown similarly as in the case where $r \geq 3$.  
\end{proof}

\begin{remark}\label{rmk:symm-formula}
The expressions in Proposition \ref{prop:number-of-comp} are computed by looking locally on the chart where $x_2 \neq 0$, but they could also be obtained by looking on one of the other charts $x_i \neq 0$ for $i = 3,\ldots, g$. This is the reason why we rewrote the formulas \eqref{eq:number-of-comp}, \eqref{eq:intersections-x0} and \eqref{eq:intersections-x1} of the proof into the formulas of the statement; this way, it is clear that they are independent of the choice of chart. In practice, however, we will often use the local expressions of the proof as they are slightly easier to work with.
\end{remark}

Another challenge will be to understand how the exceptional divisors intersect each other. When blowing up at the point $A_{k-1}$ in Step $k$, the components of $\E_{k-1}$ will be separated, and the intersections with the new exceptional divisor $\E_k$ will be \emph{equally distributed} as explained in the next proposition, in which $\mathcal D$ plays the role of the strict transform of $\E_{k-1}$. Furthermore, the new center of the blow-up will not be contained in any of the components of $\E_{k-1}$, which implies that every exceptional divisor only intersects the divisor of the previous and of the next blow-up, and that the combinatorics of these intersections stay unchanged throughout the rest of the resolution. This will be the key ingredient to show that the dual graph of the resolution is a tree as in Figure~\ref{fig:dual-graph}, see Theorem~\ref{thm:resolutionY} for the details. It is also worth mentioning that the first part of the next result is a generalization of the resolution of a cusp $x^p+y^q$ in $\C^2$ with $\gcd(p,q)$ not necessarily equal to $1$; such a cusp consists of $\gcd(p,q)$ irreducible components going through the origin and pairwise disjoint elsewhere, and after the $(q,p)$-weighted blow-up at the origin, all the components are separated, see for instance~\cite[Example 3.3]{Ma1}.

\begin{prop}\label{prop:intersection-with-previous-divisor}
We work in the same situation as Proposition \ref{prop:number-of-comp} with the stronger condition that $a_ip_j - a_jp_i = 0$ for all $i,j \in \{1,\ldots, r\}$. Consider $\P^r_{(p_0, \ldots, p_r)}(d;a_0, \ldots, a_r)$ as the exceptional divisor of the weighted blow-up $\pi: \hat{X}(d;a_0,\ldots, a_r)_{\omega} \rightarrow X(d;a_0,\ldots, a_r)$ of $X(d;a_0,\ldots, a_r)$ at the origin with weights $\omega = (p_0,\ldots, p_r)$, and let $\mathcal D$ be the strict transform under this blow-up of $D$ in $X(d;a_0,\ldots, a_r)$ defined by 
	\begin{equation}\label{eq:equations-D}
		\left \{ \begin{array}{ccccl}
	 	& & x_0^{m_0} &=& 0 \\
 		x_1^{m_1}  &+& x_2^{m_2} &=& 0  \\
	 	& \vdots & & & \\
		x_{r-1}^{m_{r-1}}   &+& x_r^{m_r} &=& 0.
		\end{array}\right.
	\end{equation}
Then, 
\begin{enumerate}[wide, labelindent=0pt]
	\item[(i)] the total number of irreducible components of $\mathcal D$ is \[\frac{m_1\cdots m_r}{\lcm(m_1,\ldots, m_r)},\] and they are all pairwise disjoint;
	\item[(ii)] each component of $\mathcal D$ is intersected by precisely one component of $\E$, and this intersection consists of a single point; and
	\item[(iii)] each component of $\E$ intersects the same number, \[\frac{m_1\lcm(m_2,\ldots, m_r)}{\lcm(m_1,\ldots, m_r)},\] of components of $\mathcal D$, which is precisely the number of components of $\mathcal D$ divided by the number of components of $\E$.
\end{enumerate}
If the above conditions (i) - (iii) are satisfied, we will say that the intersections of $\mathcal D$ and $\E$ are equally distributed.
\end{prop}

\begin{remark}\label{rmk:intersections-with-previous-divisor}
In item (iii), one can rewrite \[\frac{m_1\lcm(m_2,\ldots, m_r)}{\lcm(m_1,\ldots, m_r)} = \frac{m_1\gcd(p_1,\ldots, p_r)}{\gcd(p_2,\ldots, p_r)}.\] This is consistent with Proposition~\ref{prop:number-of-comp}, item (iii), with $a_1P - p_1Q = 0$ as $a_1p_i - a_ip_1 = 0$ for all $i \in \{1,\ldots, r\}$: the intersection of $\E$ with $\mathcal D$ corresponds to the intersection of $\E$ with $\{x_0 = 0\}$.
\end{remark} 

\begin{proof}
We start by considering for a moment the subspace of $\C^{r+1}$ defined by the equations~\eqref{eq:equations-D} and prove that the number of irreducible components of this subspace is \[\frac{m_1\cdots m_r}{\lcm(m_1,\ldots, m_r)}.\] This provides an upper bound on the number of irreducible components of $D$ and, hence, of $\mathcal D$. First of all, we can reduce to the subspace of $\C^r$ given by the last $r-1$ equations and we work by induction on $r \geq 2$. For $r = 2$, we have to consider $\{x_1^{m_1} + x_2^{m_2} = 0\}$ in $\C^2$. Let $q = \gcd(m_1,m_2)$ and denote by $\xi_i$ for $i = 1,\ldots,q$ the $q$th roots of $-1$. We can rewrite \[x_1^{m_1} + x_2^{m_2} = \prod_{i=1}^q\big(x_2^{\frac{m_2}{q}} - \xi_ix_1^{\frac{m_1}{q}}\big),\] where each factor $x_2^{\frac{m_2}{q}} - \xi_ix_1^{\frac{m_1}{q}}$ is an irreducible polynomial in $\C[x_1,x_2]$. In other words, the irreducible components are given by $\{x_2^{\frac{m_2}{q}} - \xi_ix_1^{\frac{m_1}{q}} = 0\}$, and there are $q = \gcd(m_1,m_2) = \frac{m_1m_2}{\lcm(m_1,m_2)}$ components in total. In the induction step, assuming that the statement holds for $r-1$, one can again decompose the first equation as above and reduce the problem to showing that each of the subspaces given by one factor of the first equation together with the last $r-2$ equations from~\eqref{eq:equations-D} has \[\frac{m_1\cdots m_r}{q\lcm(m_1,\ldots, m_r)}\] irreducible components. In each of these problems, the first equation can be parametrized with a parameter $t \in \C$ to further reduce the problem to investigating the components of 
	\[\left \{ \begin{array}{ccccl}
		t^{\frac{m_1m_2}{q}}&+&x_3^{m_3}&=&0\\ 
 		x_3^{m_3}  &+& x_4^{m_4} &=& 0  \\
		& \vdots & & & \\
		x_{r-1}^{m_{r-1}}   &+& x_r^{m_r} &=& 0
	\end{array}\right.\]
in $\C^{r-1}$. By the induction hypothesis, we can conclude. To show that the upper bound is attained for $\mathcal D$, we take a look at the third chart of $\hat{X}(d;a_0,\ldots, a_r)_{\omega} $ where the exceptional divisor is given by $\{x_2 = 0\}$; one could also obtain this by looking at one of the other charts, except for the first one, where the strict transform of $\mathcal D$ is not visible. The third chart is given by 
	\[X\left( \begin{array}{c|ccccccc} 
		p_2 & p_0 & p_1 & -1 & p_3 & \cdots & p_r \\
    	dp_2 & A_0 & 0 & a_2 & 0 & \cdots & 0    
	\end{array}\right),\] 
with $A_0 = a_0p_2 - a_2p_0$, via \[[(x_0,\ldots, x_r)] \longmapsto [((x_0x_2^{p_0},x_1x_2^{p_1},x_2^{p_2},x_2^{p_3}x_3,\ldots, x_2^{p_r}x_r),[x_0:x_1:1:x_3:\ldots:x_r]_{\omega})],\] see~\eqref{eq:chart-blow-up-quotient-space}. By pulling back the equations of $D$ along this map, we find the following equations of $\mathcal D$ in this chart: 
	\[\left \{ \begin{array}{ccccl}
		& & x_0^{m_0} &=& 0 \\
 		x_1^{m_1}  &+& 1 &=& 0  \\
		& \vdots & & & \\
		x_{r-1}^{m_{r-1}}   &+& x_r^{m_r} &=& 0.
	\end{array}\right.\]
From these equations, it is not hard to see that the irreducible components of $\mathcal D$ in this chart are all pairwise disjoint and given by $\mathcal D_{b'} = \{[(0,b'_1,x_2,b'_3,\ldots,b'_r)] \mid x_2 \in \C\}$ for $b' = [(b'_1,b'_3,\ldots,b'_r)] \in X(p_2;p_1,p_3,\ldots, p_r)$ a fixed solution of $x_1^{m_1}  + 1 =  \cdots = x_{r-1}^{m_{r-1}}  + x_r^{m_r} = 0$. By the second part of Lemma~\ref{lemma:number-of-solutions}, their total number is \[\frac{m_1m_3\cdots m_r\gcd(p_1,\ldots, p_r)}{p_2} = \frac{m_1\cdots m_r}{\lcm(m_1,\ldots, m_r)}.\] It follows that the total number of irreducible components of $\mathcal D$ is given by the same number and that all irreducible components of $\mathcal D$ are visible in this chart. Furthermore, by symmetry between the charts, we can conclude that all components are pairwise disjoint. This shows (i). To prove the other two statements, we first suppose that $r \geq 3$ and we keep on working in the third chart; the irreducible components of $\E$ are obtained from those of $\E \setminus \{A\}$ by adding the point $A$. As we saw in the proof of Proposition~\ref{prop:number-of-comp}, all irreducible components of $\E \setminus \{A\}$ are given by $\E_b = \{[(x_0,x_1,0,b_3,\ldots, b_r))] \mid x_0^{m_0} + x_1^{m_1} + 1 = 0\}$ for $b = [(b_3,\ldots, b_r)]$ a fixed solution in $X(p_2;p_3,\ldots, p_r)$ of $1+x_3^{m_3} = \cdots = x_{r-1}^{m_{r-1}} + x_r^{m_r} = 0$, they are pairwise disjoint, and their total number is \[\frac{m_2\cdots m_r}{\lcm(m_2,\ldots, m_r)}.\] Assume now that a component $\mathcal D_{b'}$ of $\mathcal D$ in this chart intersects a component $\E_b$ of $\E \setminus \{A\}$. Then, there exist $b_0,b_1,b_2' \in \C$ with $b_0^{m_0} + b_1^{m_1} + 1 = 0$ such that $[(0,b_1',\ldots, b_r')] = [(b_0,b_1,0,b_3,\ldots, b_r)]$ is a point in the intersection. This implies that ${b_0 = b_2' = 0}$ and that $[(b_1',b'_3,\ldots, b'_r)] = [(b_1,b_3,\ldots, b_r)]$ in $X(p_2;p_1,p_3,\ldots, p_r)$. Hence, the component $\mathcal D_{b'}$ only intersects the component of $\E \setminus \{A\}$ corresponding to $[(b'_3,\ldots, b'_r)]$, and the intersection consists of the single point $[(0,b'_1,0,b'_3,\ldots, b'_r)]$. It remains to show that each component of $\E$ has non-empty intersection with precisely \[\frac{m_1\lcm(m_2,\ldots, m_r)}{\lcm(m_1,\ldots, m_r)}\] components of $\mathcal D$. Along the same lines, we see that a component $\E_b$ of $\E \setminus \{A\}$ intersects every component $\mathcal D_{b'}$ of $\mathcal D$ in the third chart with $[(b_3',\ldots, b_r')] = [(b_3,\ldots, b_r)]$ in $X(p_2;p_3,\ldots, p_r)$. Hence, we need to count the solutions in $X(p_2;p_1,p_3,\ldots, p_r)$ of 
	\[\left \{\begin{array}{ccccl}
		x_1^{m_1}  &+& 1 &=& 0  \\
		& \vdots & & & \\
		x_{r-1}^{m_{r-1}}   &+& x_r^{m_r} &=& 0
	\end{array}\right.\]
with $[(b_3',\ldots, b_r')]$ fixed. The first part of Lemma~\ref{lemma:number-of-solutions} gives the right number, see also Remark~\ref{rmk:intersections-with-previous-divisor}. If $r = 2$, then $\E$ is irreducible and intersects every component of $\mathcal D$ in a single point; this can again be shown by considering the third chart.
\end{proof}

One last result that we discuss before going into the construction of the resolution is needed to control the power of some variables when pulling back the equations~\eqref{eq:equations-Y} of the curve $Y$. Recall that the numbers $b_{ij}$ and $n_i$ were introduced in~\eqref{eq:ni-betai}, see Section~\ref{SpaceMonomial}.

\begin{notation}
	Let $n := n_0 \cdots n_g$ and define the numbers $b_{i}^{(k)}$ for $i,k \in \{ 0,\ldots,g\}$ with $i>k$ recursively as follows:
	\begin{equation}\label{eq:def-b_i^{(k)}}
		\begin{cases}
			\displaystyle b_i^{(0)} := b_{i0} \frac{n}{n_0} & \quad \text{for } i>0, \\
			\displaystyle b_i^{(k)} := b_i^{(k-1)} + \Big( \frac{b_{ik}}{n_k} + \cdots + \frac{b_{i(i-1)}}{n_{i-1}} - 1 \Big) b_k^{(k-1)} &  \quad \text{for }i > k \geq 1.
		\end{cases}
	\end{equation}
\end{notation}

Note that $b_1^{(0)} = n$. For each $k \in \{1,\ldots, g\}$, the number $b_{i}^{(k)}$ for $i>k$ will be related to the $i$th variable $x_i$ in Step $k$ of the resolution. The following result expresses these numbers in terms of the generators $(\bar{\beta}_0,\ldots, \lbeta_g)$ of the semigroup introduced in Section~\ref{SpaceMonomial}. As a consequence, we show that they are all greater than $1$.

\begin{lemma}\label{lemma:pos-powers}
Let $i,k \in \{1, \ldots, g\}$ with $i > k$. Then, \[b_{i}^{(k)} = (n_i\lbeta_i - n_k\lbeta_k) - \frac{b_{i(i-1)}}{n_{i-1}}(n_{i-1}\lbeta_{i-1} - n_k\lbeta_k) - \cdots -  \frac{b_{i(k+1)}}{n_{k+1}}(n_{k+1}\lbeta_{k+1} - n_k\lbeta_k),\] and, in particular, $b_{k+1}^{(k)} = n_{k+1} \bar{\beta}_{k+1} - n_k \bar{\beta}_k$. Furthermore, $b_i^{(k)} > 1$ or, equivalently,
	\begin{equation}\label{eq:inequality-b}
		b_k^{(k-1)} + 1 < b_i^{(k-1)} + b_{ik} \frac{b_k^{(k-1)}}{n_k} + \cdots + b_{i(i-1)} \frac{b_k^{(k-1)}}{n_{i-1}}.
	\end{equation}
\end{lemma}

\begin{proof}
We use induction on $k$. Let us first consider $k=1$. Note that $\bar{\beta}_0 = \frac{n}{n_0}$ and $\bar{\beta}_1 = \frac{n}{n_1}$. Using equation~\eqref{eq:ni-betai}, the term $b_i^{(1)}$ for $i > 1$ can indeed be rewritten as
	\begin{align*}
 		b_{i}^{(1)} & = b_{i0} \lbeta_0 +  b_{i1} \lbeta_1 + \Big(\frac{b_{i2}}{n_2} + \cdots + \frac{b_{i(i-1)}}{n_{i-1}} - 1\Big)n \\
 		& = n_i\lbeta_i - b_{i2}\lbeta_2 - \cdots - b_{i(i-1)}\lbeta_{i-1} + \Big(\frac{b_{i2}}{n_2} + \cdots + \frac{b_{i(i-1)}}{n_{i-1}} - 1\Big)n_1\lbeta_1 \\
 		& = (n_i\lbeta_i - n_1\lbeta_1) - \frac{b_{i(i-1)}}{n_{i-1}}(n_{i-1}\lbeta_{i-1} - n_1\lbeta_1) - \cdots -  \frac{b_{i2}}{n_2}(n_2\lbeta_2 - n_1\lbeta_1). 
	\end{align*}
Let us now consider the general case. By induction, we know that $b_k^{(k-1)} = n_k \bar{\beta}_k - n_{k-1} \bar{\beta}_{k-1} $ and that $b_i^{(k-1)}$ for $i > k - 1$ can be written as \[b_i^{(k-1)} = n_i \bar{\beta}_i - b_{ik} \bar{\beta}_k - \cdots - b_{i(i-1)} \bar{\beta}_{i-1} + \Big( \frac{b_{ik}}{n_k} + \cdots + \frac{b_{i(i-1)}}{n_{i-1}} - 1 \Big) n_{k-1} \bar{\beta}_{k-1}.\] Hence, by definition, we have for $i > k$ that 
	\[\begin{aligned}
		b_i^{(k)} &= b_i^{(k-1)} + \Big( \frac{b_{ik}}{n_k} + \cdots + \frac{b_{i(i-1)}}{n_{i-1}} - 1 \Big) b_k^{(k-1)} \\
		&= n_i \bar{\beta}_i - b_{ik} \bar{\beta}_k - \cdots - b_{i(i-1)} \bar{\beta}_{i-1} + \Big( \frac{b_{ik}}{n_k} + \cdots + \frac{b_{i(i-1)}}{n_{i-1}} - 1 \Big) n_k \bar{\beta}_k. 
	\end{aligned}\]
After regrouping, we obtain the required formula. For the second part of the lemma, as $b_{ij} < n_j$ whenever $i>j\neq 0$, see the extra assumption on~\eqref{eq:ni-betai}, it is enough to show that \[(n_i\lbeta_i- n_k \lbeta_k) - (n_{i-1}\lbeta_{i-1}-n_k\lbeta_k) - \cdots - (n_{k+1}\lbeta_{k+1}-n_k\lbeta_k) > 1.\] We proceed by induction on $i > k$. For $i = k+1$, one indeed has $n_{k+1} \bar{\beta}_{k+1} - n_k \bar{\beta}_k > 1,$ since $\lbeta_{k+1} > n_{k}\lbeta_{k}$ and $n_{k+1} \geq 2$. Suppose now that it is true for $i-1$ with $i>k+1$. The conditions $\bar{\beta}_i > n_{i-1} \bar{\beta}_{i-1}$ and $n_i \geq 2$ imply that $n_i\lbeta_i - n_k\lbeta_k > n_i(n_{i-1}\lbeta_{i-1} - n_k\lbeta_k).$ Hence,
	\begin{align*}
		&(n_i\lbeta_i- n_k\lbeta_k) - (n_{i-1}\lbeta_{i-1}-n_k\lbeta_k) - \cdots - (n_{k+1}\lbeta_{k+1}-n_k\lbeta_k) \\
		&>(n_i-1)(n_{i-1}\lbeta_{i-1} - n_k\lbeta_k) - (n_{i-2}\lbeta_{i-2} - n_k\lbeta_k) - \cdots - (n_{k+1}\lbeta_{k+1}-n_k\lbeta_k) \\
		&\geq (n_{i-1}\lbeta_{i-1} - n_k\lbeta_k) - (n_{i-2}\lbeta_{i-2} - n_k\lbeta_k) - \cdots - (n_{k+1}\lbeta_{k+1}-n_k\lbeta_k)\\
		&> 1,
	\end{align*}
where the second inequality again follows from $n_i \geq 2$, and the last one from the induction hypothesis.
\end{proof}

\subsection{Construction of the embedded $\Q$-resolution of $Y \subset S$} We are now ready to start with Step~1 in the resolution of $Y \subset S$, focusing on the information needed to compute the zeta function of monodromy. The idea is to consider the blow-up $\pi_1$ at the origin of $\C^{g+1}$ with some weights and study its restriction to $S$ that we call $\varphi_1 := \pi_1|_{\hat{S}}: \hat{S} \to S$, with $\hat{S}$ the strict transform of $S$. After this blow-up, we will be able to eliminate one variable so that we attain the same situation as in the beginning, but in one dimension less and where the ambient space contains quotient singularities. In Step~2, we will again consider a weighted blow-up of the ambient space and its restriction $\varphi_2$ to $\hat{S}$. As mentioned in the beginning of this section, we will need $g$ such steps. Denote by $\E_k$ for $k=1,\ldots,g$ the exceptional divisor of $\varphi_k$ appearing at Step $k$; we will also denote their strict transforms throughout the process by $\E_k$. To keep track of the necessary combinatorics of these divisors, we introduce $H_i$ for $i = 0,\ldots, g$ as the divisor in $S$ defined by $\{ x_i = 0 \} \cap S \subset \C^{g+1}$. We will see in the process of resolving the singularity that (the strict transform of) $H_k$ is separated from the strict transform $\hat{Y}$ of $Y$ precisely at Step $k$ and that it intersects the $k$th exceptional divisor $\E_k$ transversely. Therefore, it is interesting to study how the $H_i$'s behave in the process of resolving $Y \subset S$, although they are not part of our curve. We again keep on denoting them by $H_i$.

\subsubsection{Step 1: weighted blow-up \texorpdfstring{$\pi_1$}{pi1} at \texorpdfstring{$0\in\C^{g+1}$}{0inCg+1}
with weights \texorpdfstring{$\omega_1$}{\omega1}}

Let $\pi_1: \hat{\C}^{g+1}_{\omega_1} \to \C^{g+1}$ be the weighted blow-up at the origin with respect to
$\omega_1 := \big(\frac{n}{n_0}, \ldots, \frac{n}{n_g}\big)$, where $n = n_0n_1\cdots n_g$. For a better exposition, we split the section into several parts.

\vspace{14pt}

\underline{Global situation}. Let us first discuss the global picture. Recall that the equations of $Y$ and $S$ are given by~\eqref{eq:equations-Y} and~\eqref{eq:equations-S}, respectively, and that the exceptional divisor $E_1$ of $\pi_1$ is identified with $\P^g_{\omega_1}$. The exceptional divisor $\E_1 := E_1 \cap \hat{S}$ of $\varphi_1 = \pi_1|_{\hat{S}}:\hat{S} \rightarrow S$ is in the coordinates of $\P^g_{\omega_1}$ given by the $\omega_1$-homogeneous part of $S$. By the inequality~\eqref{eq:inequality-b} in Lemma~\ref{lemma:pos-powers} for $k=1$ and $i = 2, \ldots, g$, we have \[n < n + 1 < b_{i0}\frac{n}{n_0} + b_{i1} \frac{n}{n_1} + \cdots + b_{i(i-1)}\frac{n}{n_{i-1}},\] so that $\E_1 \subset \P^g_{\omega_1}$ is defined by
	\begin{equation}\label{eq:E1-homog}
		\left \{ \begin{array}{clccccl}
		x_1^{n_1} & - & x_0^{n_0}  &+& \lambda_2x_2^{n_2} &=& 0 \\
		&  & x_2^{n_2}  &+& \lambda_3x_3^{n_3} &=& 0  \\
		& & & \vdots & & & \\
		& & x_{g-1}^{n_{g-1}}   &+& \lambda_g x_g^{n_g} &=& 0.
	\end{array}\right.
\end{equation}

After a change of variables, we can assume that all coefficients in these equations are equal to $1$ so that they  satisfy the conditions of Proposition~\ref{prop:number-of-comp} with $d = 1$ and $a_i = 0$ for $i = 0,\ldots, g$; for instance, the intersection $\E_1 \cap \{x_i = 0\} = \E_1 \cap H_i$ for $i = 2,\ldots, g$ is the point $P_1 := [1:1:0:\ldots:0]$. According to this proposition, the number of irreducible components of $\E_1$ is
	\begin{equation}\label{eq:irred-e1}
		\frac{n_2\cdots n_g}{\lcm(n_2, \ldots, n_g)} = \frac{e_1}{\lcm(n_2, \ldots, n_g)}.
	\end{equation}
If $g=2$, then this number is equal to $1$ or, thus, $\E_1$ is irreducible. All the irreducible components of $\E_1$ have the point $P_1$ in common and are pairwise disjoint outside $P_1$. Combining~\eqref{eq:number-of-comp} and~\eqref{eq:intersections-x0} from Proposition~\ref{prop:number-of-comp}, the intersection $\E_1 \cap H_0$, which is $\E_1 \cap \{ x_0=0 \}$ in these coordinates, contains
	\begin{equation}\label{eq:points-E1-H0}
		\frac{n_1n_3\cdots n_g\gcd(\frac{n}{n_1},\frac{n}{n_2},\ldots, \frac{n}{n_g})}{\frac{n}{n_2}}
		= \gcd\Big(\frac{\lbeta_0}{n_1},\frac{\lbeta_0}{n_2},\ldots, \frac{\lbeta_0}{n_g}\Big)
		= \frac{\lbeta_0}{\lcm(n_1,n_2, \ldots, n_g)}
	\end{equation}
points, where $n = n_0\lbeta_0$ and the relation~\eqref{eq:rel-gcd-lcm} was used in the first and second equality, respectively. Analogously, from~\eqref{eq:number-of-comp} and~\eqref{eq:intersections-x1}, the intersection $\E_1 \cap H_1$ is formed by
	\begin{equation}\label{eq:points-E1-H1}
		\frac{n_0n_3\cdots n_g\gcd(\frac{n}{n_0},\frac{n}{n_2},\ldots, \frac{n}{n_g})}{\frac{n}{n_2}}
		= \gcd\Big(\frac{\lbeta_1}{n_0},\frac{\lbeta_1}{n_2},\ldots, \frac{\lbeta_1}{n_g}\Big)
		= \frac{\lbeta_1}{\lcm(n_0,n_2, \ldots, n_g)}
	\end{equation}
points. The fact that each irreducible component of $\E_1$ has the same number of intersections with $H_0$ (resp.~$H_1$) is compatible with the fact that the integer in~\eqref{eq:irred-e1} divides the one in~\eqref{eq:points-E1-H0} (resp.~\eqref{eq:points-E1-H1}). The intersection $\E_1 \cap \hat{Y}$ of $\E_1$ with the strict transform of $Y$ is defined by the $\omega_1$-homogeneous part of $Y$: $x_1^{n_1} - x_0^{n_0} = x_2 = \cdots = x_g = 0$. This is simply the point $P_1$. The global situation in the strict transform $\hat{S}$ for $g \geq 3$ is illustrated in Figure~\ref{fig:first-step}. For simplicity, the components of $\E_1$ are represented by lines, but they are in general neither smooth nor rational curves. If $g = 2$, we can make the same picture with $\E_1$ irreducible.

\begin{figure}[ht]
	\includegraphics{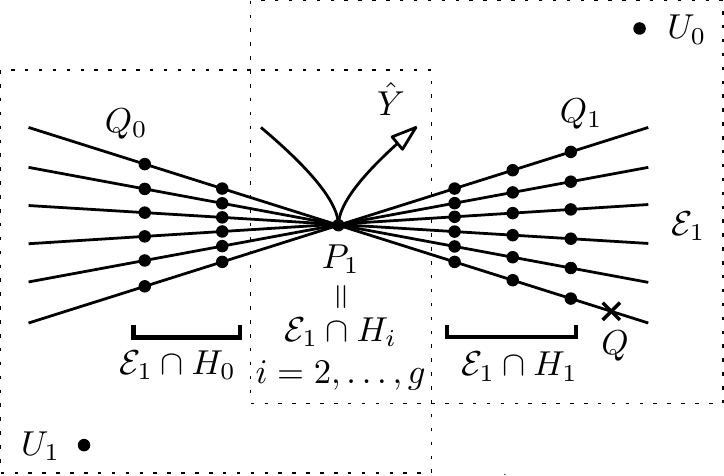}
	\caption{Step 1 in the resolution of $Y \subset S$ for $g \geq 3$.}
	\label{fig:first-step}
\end{figure}

\vspace{14pt}

In order to study the singular locus of $\hat{S}$, we use local coordinates. Note that the surface $\hat{S}$ is smooth outside $\E_1$: the complement $\hat{S} \setminus \E_1$ is isomorphic to $S \setminus \{0\}$, which is smooth as $(S,0)$ is an isolated singularity. To study the situation on $\E_1$, we just need to have a look at the first two charts $U_0$ and $U_1$ of $\hat{\C}^{g+1}_{\omega_1}$ because $\E_1 \cap H_0 \cap H_1 = \emptyset$. In fact, the local study of $\hat{S}$ around points of $\E_1$ can be understood using the first chart, except for the finite number of points in the intersection $\E_1 \cap H_0$. For the latter points, the second chart is employed.

\vspace{14pt}

\underline{Points in $\E_1 \setminus \bigcup_{i=0}^g H_i$}. Let us compute the equations of $\hat{S}$ and $\hat{Y}$ in the first chart $U_0$ of $\hat{\C}^{g+1}_{\omega_1}$. They are obtained via \[ (x_0, \ldots, x_g) \ \longmapsto \ (x_0^{\frac{n}{n_0}}, x_0^{\frac{n}{n_1}} x_1, \ldots, x_0^{\frac{n}{n_g}} x_g),\] and the new ambient space is $U_0 = X(\frac{n}{n_0};-1,\frac{n}{n_1}, \ldots, \frac{n}{n_g})$, see~\eqref{eq:chart-blow-up}. The total transform $\varphi_1^{-1}(Y)$ is defined by $x_0^n \hat{f}_1 = \dots = x_0^n \hat{f}_g = 0$, where
	\[\left\{ \begin{array}{l c l l}
		\hat{f}_1 & := &  x_1^{n_1} - 1 \\
		\hat{f}_2 & := &  x_2^{n_2} - x_0^{b_{2}^{(1)}} x_1^{b_{21}}\\
		& \vdots & & \\
		\hat{f}_g & := & x_g^{n_g} - x_0^{b_{g}^{(1)}} x_1^{b_{g1}} \cdots x_{g-1}^{b_{g(g-1)}}
	\end{array}\right.\]
define the strict transform $\hat{Y}$, and $x_0^n: \hat{S} \to \C$ is the exceptional part. Here, $b_i^{(1)} = b_{i0}\frac{n}{n_0} + (\frac{b_{i1}}{n_1} + \cdots +\frac{b_{i(i-1)}}{n_{i-1}} - 1)n  > 1$ for $i=2,\ldots,g$,  see Lemma~\ref{lemma:pos-powers}. The strict transform $\hat{S}$ is given by $\hat{f}_1 + \lambda_2 \hat{f}_2 = \cdots = \hat{f}_{g-1} +\lambda_g \hat{f}_g = 0$, and $H_i$ for $i = 1,\ldots, g$ by $\{x_i = 0\} \cap \hat{S}$. Note that $H_0$ is not visible in this chart. On $\E_1 \setminus \bigcup_{i=1}^g H_i$, the ambient space $U_0$ is smooth, and one can use the standard Jacobian criterion to show that $\hat{S}$ is also smooth on this set: the Jacobian matrix of $\hat{S}$ is a $(g-1)\times(g+1)$-matrix containing a lower triangular $(g-1)\times(g-1)$-matrix with diagonal $(\lambda_2 n_2 x_2^{n_2-1},\ldots,\lambda_g n_g x_g^{n_g-1})$. To compute the multiplicity of the exceptional divisor, we take a look at the equations around a generic point $Q = [(0,a_1,\ldots,a_g)] \in \E_1 \setminus \bigcup_{i=1}^g H_i$, where $a_i \in \C^{*}$. The order of the stabilizer subgroup of $Q$ is $\gcd(\frac{n}{n_0},\ldots, \frac{n}{n_g})$, and, hence, as germs, \[\bigg(X\Big(\frac{n}{n_0};-1,\frac{n}{n_1},\ldots, \frac{n}{n_g}\Big),Q\bigg) \simeq \bigg(X\Big(\gcd\Big(\frac{n}{n_0},\ldots, \frac{n}{n_g}\Big);-1,0,\ldots, 0\Big),Q\bigg) \simeq (\C^{g+1},0),\] see Section~\ref{MonodromyWeighted}. The function $x_0^n: U_0 \to \C$ is transformed under the previous isomorphism into $x_0^{N_1}: \C^{g+1} \to \C$, where \[N_1 :=  \frac{n}{\gcd(\frac{n}{n_0},\ldots, \frac{n}{n_g})} = \lcm(n_0, \ldots, n_g)\] is the multiplicity of $\E_1$ defined in~\eqref{eq:def-mult}. Here, we used once again the relation~\eqref{eq:rel-gcd-lcm}.

\vspace{14pt}

\underline{Points in the intersection $\E_1 \cap H_1$}. Let $Q_1 = [(0,0,a_2,\ldots,a_g)]$ be a point in $\E_1 \cap H_1$ considered on the first chart, where $a_i \in \C^{\ast}$ are chosen such that $-1 + \lambda_2 a_2^{n_2} = a_2^{n_2} + \lambda_3 a_3^{n_3} = \cdots = a_{g-1}^{n_{g-1}} + \lambda_g a_g^{n_g} = 0$. The order of the stabilizer subgroup of $Q_1$ is $\gcd(\frac{n}{n_0},\frac{n}{n_2},\ldots,\frac{n}{n_g})$, and $U_0$ around $Q_1$ becomes $X(\gcd(\frac{n}{n_0},\frac{n}{n_2},\ldots,\frac{n}{n_g}); -1,\frac{n}{n_1},0,\ldots,0)$. To have a chart centered at the origin, we can change the coordinates $x_i \mapsto x_i + a_i$ for $i=2,\ldots,g$. In these new coordinates, $\hat{S}$ is described in $X(\gcd(\frac{n}{n_0},\frac{n}{n_2},\ldots,\frac{n}{n_g}); -1,\frac{n}{n_1},0,\ldots,0)$ by equations of the form
	\[\left\{\begin{array}{l c l c}
		y_2 &:= & u_2(x_2) x_2 - h_2(x_0,x_1) &= 0 \\
		y_3 &:= & u_3(x_3) x_3 - h_3(x_0,x_1,x_2) &= 0 \\
		& \vdots & & \\
		y_g &:= & u_g(x_g) x_g - h_g(x_0,\ldots,x_{g-1}) &= 0,
	\end{array}\right.\]
where $u_i(x_i) \in \C\{ x_i \}$ are units, and $h_i \in \C[x_0,x_1,\ldots, x_{i-1}]$. By making the change of coordinates $y_0=x_0$, $y_1=x_1$, $y_i = u_i x_i - h_i$ for $i=2,\ldots,g$, we finally obtain the following situation at $[(x_0,x_1)]$ :
	\begin{equation}\label{eq:points-Q1}
		\left\{\begin{aligned}
			& \hat{S} = X \bigg( \gcd\Big(\frac{n}{n_0},\frac{n}{n_2},\ldots,\frac{n}{n_g}\Big);-1,\frac{n}{n_1} \bigg) \\
			& \E_1: \ x_0^n = 0, \qquad H_1: \ x_1 = 0.
		\end{aligned}\right.
	\end{equation}
In particular, the total transform $\varphi_1^{-1} (Y)$ has $\Q$-normal crossings on $\hat{S}$ at these points.

\vspace{14pt}

\underline{Points in the intersection $\E_1 \cap H_0$}. As mentioned before, to study these points, we need to consider the second chart $U_1 = X(\frac{n}{n_1};\frac{n}{n_0}, -1, \frac{n}{n_2}, \ldots, \frac{n}{n_g})$ via \[(x_0, \ldots, x_g) \longmapsto (x_0x_1^{\frac{n}{n_0}}, x_1^{\frac{n}{n_1}}, x_1^{\frac{n}{n_2}}x_2, \ldots, x_1^{\frac{n}{n_g}} x_g).\] Choose a point $Q_0 \in \E_1 \cap H_0$, which is of the form $[(0,0,a_2,\ldots,a_g)]$ for $a_i \in \C^{\ast}$ satisfying a set of equations similar as $Q_1 \in \E_1 \cap H_1$. Since its stabilizer subgroup has order $\gcd(\frac{n}{n_1},\ldots,\frac{n}{n_g})$, one obtains by repeating the same arguments as in~\eqref{eq:points-Q1} the following local situation around $Q_0$ at $[(x_0,x_1)]$:
	\begin{equation}\label{eq:points-Q0}
		\left\{\begin{aligned}
			& \hat{S} = X \bigg( \gcd\Big(\frac{n}{n_1},\frac{n}{n_2},\ldots,\frac{n}{n_g}\Big);\frac{n}{n_0},-1 \bigg) \\
			& \E_1: \ x_1^n = 0, \qquad H_0: \ x_0 = 0.
		\end{aligned}\right.
	\end{equation}
The total transform of $Y$ is again a $\Q$-normal crossings divisor around such points.

\vspace{14pt}

\underline{The point $P_1 = \E_1 \cap H_i$ for $i = 2,\ldots,g$}. In the first chart, $P_1 = [(0,1,0, \ldots, 0)]$, and the order of its stabilizer subgroup is $\gcd(\frac{n}{n_0},\frac{n}{n_1}) = e_1$. Hence, as germs, \[\bigg(X\Big(\frac{n}{n_0};-1,\frac{n}{n_1},\ldots, \frac{n}{n_g}\Big),P_1\bigg) \simeq \bigg(X\Big(e_1; -1,0,\frac{n}{n_2},\ldots, \frac{n}{n_g}\Big),P_1 \bigg).\] We use the change of variables $x_1 \mapsto x_1 + 1$ and $x_i \mapsto x_i$ for $i=0, 2, \ldots, g$ to get a chart centered at the origin in which $\hat{S}$ is given by
	\begin{equation}\label{eq:Shat-at-P1}
		\hat{f}_i(x_0,x_1+1, x_2, \ldots, x_i) + \lambda_{i+1} \hat{f}_{i+1}(x_0,x_1+1, x_2, \ldots, x_{i+1}) = 0,
		\quad i = 1, \ldots, g-1.
	\end{equation}
Consider the first equation as a function $F: \C^2 \times \C \rightarrow \C$. Since $\frac{\partial F}{\partial x_1}(0) = n_1 \neq 0$, 
there exists some $h \in \C\{x_0,x_2\}$ such that the set of zeros of $F$ in $\C^3$ can be described as $\{(x_0,x_1,x_2)\in \C^3 \mid x_1 = h(x_0,x_2)\}$. In particular, \[(h(x_0,x_2) + 1)^{n_1} - 1 + \lambda_2(x_2^{n_2} - x_0^{b_2^{(1)}}(h(x_0,x_2) + 1)^{b_{21}}) = 0.\] Because the action on $x_1$ is trivial, and $x_1=h(x_0,x_2)$ provides a set of zeros in the quotient space, we know that $h(x_0,x_2)$ is invariant under the group action of type $(e_1; -1,\frac{n}{n_2})$. The above equation can be rewritten as
	\begin{equation}\label{eq:hx0x2}
		h(x_0,x_2)= u(x_0,x_2)(x_2^{n_2} - x_0^{b_2^{(1)}})
	\end{equation}
with $u(x_0,x_2)\in \C\{x_0,x_2\}$ a unit. For a better understanding of the whole process, we distinguish two cases: $g=2$ and $g \geq 3$.

\vspace{14pt}

If $g = 2$, then $\hat{S}$ is locally around $P_1 = [(0,\ldots,0)]$ defined by $x_1 = h(x_0,x_2)$. The projection $pr:\big(X(n_2; -1,0,\frac{n}{n_2}),0\big) \rightarrow \big(X(n_2; -1,\frac{n}{n_2}),0\big)$ given by $[(x_0,x_1,x_2))] \mapsto [(x_0,x_2)]$ induces locally an isomorphism of $\hat{S}$ onto $X(n_2; -1,\frac{n}{n_2})$. The total transform $\varphi_1^{-1}(Y)$ is given by  
	\begin{equation}\label{eq:step1-g2}
		x_0^n(x_2^{n_2} - x_0^{b_2^{(1)}}) = 0,
	\end{equation}
where $x_2^{n_2} - x_0^{b_2^{(1)}} = 0$ defines the strict transform $\hat{Y}$, and $x_0^n = 0$ the exceptional divisor $\E_1$. This shows in particular that $\E_1$ is irreducible as was already stated in~\eqref{eq:irred-e1}. The divisor $H_2$ is still $\{x_2 = 0\}$ in $\hat{S}$.

\vspace{14pt}

If $g \geq 3$, then one can rewrite the equations~\eqref{eq:Shat-at-P1} using~\eqref{eq:hx0x2} so that $\hat{S}$ is defined by the equation $x_1 = h(x_0,x_2)$ locally around $P_1 = [(0,\ldots,0)]$, and \[\hat{f}_i(x_0,1, x_2, \ldots, x_i) + \lambda_{i+1} \hat{f}_{i+1}(x_0,1, x_2, \ldots, x_{i+1}) + (x_2^{n_2}  -  x_0^{b_2^{(1)}})R_i^{(1)}(x_0,x_2,\ldots, x_i) = 0,\] for $i = 2, \ldots, g-1$, where every $R_i^{(1)}(x_0,x_2,\ldots, x_i) \in \C \{x_0,x_2,\ldots, x_i\}$ is compatible with the action (i.e., it defines a zero set in the quotient) and satisfies $R_i^{(1)}(0,x_2,\ldots, x_i) = 0$. The projection \[pr: \bigg(X\Big(e_1; -1,0,\frac{n}{n_2},\ldots, \frac{n}{n_g}\Big),0\bigg) \longrightarrow \bigg(X\Big(e_1; -1,\frac{n}{n_2},\ldots,\frac{n}{n_g}\Big),0\bigg)\] given by $[(x_0,x_1,x_2,\ldots, x_g)] \mapsto \ [(x_0,x_2, \ldots, x_g)]$ induces an isomorphism of $\hat{S}$ onto the subvariety of $X(e_1; -1,\frac{n}{n_2},\ldots, \frac{n}{n_g})$ defined by
	\begin{equation}\label{eq:Shat-ggeq3}
		\left\{\begin{array}{lcl}
			x_2^{n_2} - x_0^{b_2^{(1)}} + \lambda_3 (x_3^{n_3} - x_0^{b_3^{(1)}} x_2^{b_{32}}) + (x_2^{n_2}  -  x_0^{b_2^{(1)}})R^{(1)}_2(x_0,x_2) &=& 0 \\
			x_3^{n_3} - x_0^{b_3^{(1)}} x_2^{b_{32}} + \lambda_4 ( x_4^{n_4} - x_0^{b_4^{(1)}} x_2^{b_{42}} x_3^{b_{43}} ) + (x_2^{n_2}  -  x_0^{b_2^{(1)}}) R^{(1)}_3(x_0,x_2,x_3) &=& 0 \\ 
			& \vdots &  \\
			x_{g-1}^{n_{g-1}} - x_0^{b_{g-1}^{(1)}} x_2^{b_{(g-1)2}} \cdots x_{g-2}^{b_{(g-1)(g-2)}}  +  \lambda_g(x_g^{n_g} - x_0^{b_g^{(1)}} x_2^{b_{g2}} \cdots x_{g-1}^{b_{g(g-1)}}) \\
			\hspace{170pt}  +  (x_2^{n_2} - x_0^{b_2^{(1)}}) R^{(1)}_{g-1}(x_0,x_2,\ldots, x_{g-1}) &=& 0.
		\end{array}\right.
	\end{equation}
The total transform of $Y$ is given by
	\begin{equation}\label{eq:Yhat-ggeq3}
		\left\{\begin{array}{lcl}
			x_0^n(x_2^{n_2} - x_0^{b_2^{(1)}}) & = & 0 \\
			x_0^n(x_3^{n_3} - x_0^{b_3^{(1)}}x_2^{b_{32}}) & = & 0 \\
			& \vdots &  \\
			x_0^n(x_g^{n_g} - x_0^{b_g^{(1)}}x_2^{b_{g2}} \cdots x_{g-1}^{b_{g(g-1)}}) & = & 0,
		\end{array}\right.
	\end{equation}
where $x_0^n = 0$ corresponds to $\E_1$, and $H_i = \{x_i = 0\} \cap \hat{S}$ for $i = 2,\ldots, g$.

\vspace{14pt}
In both cases, we can conclude that $\varphi_1$ is an embedded $\Q$-resolution of $Y \subset S$ except at the point $P_1$. In Step 2, we will blow up at this point. If $g=2$, the curve $\hat{Y}$ is a cusp inside a cyclic quotient singularity, and we will finish right after this blow-up. If $g \geq 3$, we see in~\eqref{eq:Shat-ggeq3} and \eqref{eq:Yhat-ggeq3} that we were able to eliminate $x_1$, and that we obtained a situation very similar to the one we have started with, but with one equation in $\hat{S}$ and $\hat{Y}$ less, see~\eqref{eq:equations-S} and ~\eqref{eq:equations-Y}. However, Step 2 is essentially different and more challenging than Step 1 because the ambient space of $\hat{S}$ contains singularities.

\subsubsection{Step 2: weighted blow-up \texorpdfstring{$\pi_2$}{pi2} at \texorpdfstring{$P_1$}{P1} with weights \texorpdfstring{$\omega_2$}{w2}} 

We keep the distinction between $g = 2$ and $g \geq 3$.

\vspace{14pt}

If $g = 2$, then we consider the weighted blow-up $\pi_2 = \varphi_2$ of $\hat{S} = X(n_2;-1,\frac{n}{n_2})$, on which $\varphi_1^{-1}(Y)$ is given by~\eqref{eq:step1-g2}, at $P_1 = [(0,0)]$ with respect to the weights $\omega_2 := (1,\frac{b_2^{(1)}}{n_2})$. Note that $b_2^{(1)} = n_2\lbeta_2-n_1\lbeta_1$ is divisible by $n_2 = e_1$. This produces an irreducible exceptional divisor $\E_2 = \P^{1}_{\omega_2}(n_2; -1, \frac{n}{n_2}) \simeq \P^1$ with multiplicity $N_2 = n + b_2^{(1)} = n_2 \bar\beta_2$. The new strict transform $\hat{Y}$ is smooth and intersects $\E_2$ transversely at a smooth point of $\hat{S}$. The intersection $\E_2 \cap H_2$ is just one point, and the equation of the total transform of $Y$ around this point is $x_0^{n_2 \bar\beta_2} : X(n_2;-1,\bar\beta_2) \to \C$. Finally, $\E_2$ intersects $\E_1$ at a single point, and around this point we have the function
	\[x_0^n x_2^{n_2 \bar\beta_2} : X \left(\begin{array}{c|cc}
		\frac{n_2\lbeta_2 - n_1\lbeta_1}{n_2} & 1 & -1 \\
		n_2\lbeta_2 - n_1\lbeta_1 & -\lbeta_2 & \frac{n}{n_2}
	\end{array} \right) \longrightarrow \C.\]
The composition $\varphi := \varphi_1 \circ \varphi_2: \hat{S} \to S$ is an embedded $\Q$-resolution of $Y$. The final situation is illustrated in Figure~\ref{fig:second-step-g2}; the numbers in brackets are the orders of the underlying small groups at the intersection points $\E_1 \cap H_i$ for $i = 0,1$ and $\E_2 \cap H_2$, see Remark~\ref{rmk:order-small-groups-g=2}.

\begin{figure}[ht]
	\includegraphics{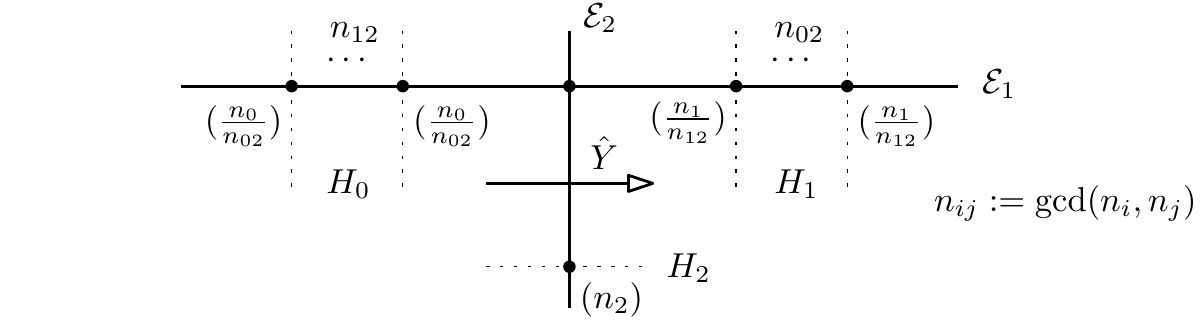}
	\caption{Embedded $\Q$-resolution of $Y \subset S$ for $g=2$.}
	\label{fig:second-step-g2}
\end{figure}

\vspace{14pt}

Assume $g \geq 3$ from now on. Consider the equations~\eqref{eq:Shat-ggeq3} and~\eqref{eq:Yhat-ggeq3} of $\hat{S}$ and $\hat{Y}$, respectively, around $P_1 = [(0,\ldots,0)]$ in  $X(e_1; -1,\frac{n}{n_2},\ldots, \frac{n}{n_g})$. Let $\pi_2$ be the blow-up of $X(e_1; -1,\frac{n}{n_2},\ldots, \frac{n}{n_g})$ at $P_1$ with respect to the weight vector $\omega_2 := ( 1, \frac{b_2^{(1)}}{n_2}, \ldots, \frac{b_2^{(1)}}{n_g}).$ Note that $b_2^{(1)} = n_2\lbeta_2 - n_1\lbeta_1$ is divisible by $e_1 = n_2e_2 = n_2 \cdots n_g$, see Section~\ref{SpaceMonomial}. Denote by $E_2 \simeq \P^{g-1}_{\omega_2}\big(e_1; -1, \frac{n}{n_2}, \ldots, \frac{n}{n_g}\big)$ the exceptional divisor of $\pi_2$, and let $\varphi_2 := \pi_2|_{\hat{S}}: \hat{S} \to \hat{S}$ be the restriction map with exceptional divisor $\E_2 := E_2 \cap \hat{S}$. Here, we denote the strict transform of $\hat{S}$ again by $\hat{S}$. As in Step 1, we start with the global situation.

\vspace{14pt}

\underline{Global situation}.
Because $R_i^{(1)}(x_0,x_2,\ldots, x_i)$ for $i = 2,\ldots, g-1$ is not a unit, and \[b_2^{(1)} < b_2^{(1)} + 1 < b_i^{(1)} + b_{i2} \frac{b_2^{(1)}}{n_2} + \cdots + b_{i(i-1)} \frac{b_2^{(1)}}{n_{i-1}}, \qquad i=3,\ldots,g,\] by~\eqref{eq:inequality-b} from Lemma~\ref{lemma:pos-powers}, the exceptional divisor $\E_2$ is in $\P^{g-1}_{\omega_2}\big(e_1; -1, \frac{n}{n_2}, \ldots, \frac{n}{n_g}\big)$ given by
	\[\left \{ \begin{array}{clccccl}
		x_2^{n_2} & - & x_0^{b_2^{(1)}}  &+& \lambda_3x_3^{n_3} &=& 0 \\
 		&  & x_3^{n_3}  &+& \lambda_4x_4^{n_4} &=& 0  \\
		& & & \vdots & & & \\
		& & x_{g-1}^{n_{g-1}} &+& \lambda_g x_g^{n_g} &=& 0. 
	\end{array}\right.\]
As these equations satisfy, modulo the coefficients, the conditions of Proposition~\ref{prop:number-of-comp}, we know that $\E_2$ has \[
\frac{n_3\cdots n_g}{\lcm(n_3, \ldots, n_g)} = \frac{e_2}{\lcm(n_3, \ldots, n_g)}\] irreducible components. Note that if $g=3$, then $\E_2$ is irreducible. The intersection $\E_2 \cap H_i$ for $i = 3,\ldots, g$ consists of the single point $P_2 := [1:1:0:\ldots:0]$, which is contained in all components of $\E_2$, while they are pairwise disjoint outside $P_2$. By equations~\eqref{eq:number-of-comp} and~\eqref{eq:intersections-x0} with $a_1p_2 - a_2p_1 = 0$, the intersection $\E_2 \cap \E_1$, which corresponds to $\E_2 \cap \{ x_0=0 \}$, consists of \[	\frac{n_2 n_4 \cdots n_g \gcd\big(\frac{b^{(1)}_2}{n_2}, \ldots, \frac{b^{(1)}_2}{n_g}\big)}{\frac{b^{(1)}_2}{n_3}} = \gcd\Big(\frac{e_1}{n_2},\ldots, \frac{e_1}{n_g} \Big) = \frac{e_1}{\lcm(n_2, \ldots, n_g)}.\] points. Note that this is precisely the number of irreducible components of $\E_1$, see~\eqref{eq:irred-e1}. Using ~\eqref{eq:number-of-comp} and~\eqref{eq:intersections-x1}, one can compute that there are 
	\begin{align*}
		& \frac{b^{(1)}_2 n_4 \cdots n_g \gcd \Big( e_1 \frac{b^{(1)}_2}{n_3},\frac{n_2 \bar{\beta}_2}{n_3} \gcd\big(\frac{b^{(1)}_2}{n_3},\ldots,\frac{b^{(1)}_2}{n_g}\big)\Big)}{e_1 \frac{b^{(1)}_2}{n_3} \frac{b^{(1)}_2}{n_3}} \\
		& = \gcd \Big( e_2,\frac{\lbeta_2}{n_3}, \ldots, \frac{\lbeta_2}{n_g} \Big) = \frac{\lbeta_2}{\lcm(\frac{\lbeta_2}{e_2},n_3,\ldots, n_g)}
	\end{align*}
points in the intersection $\E_2 \cap H_2$. The first equality is a consequence of the fact that $\frac{n_2 \bar{\beta}_2}{n_3} \gcd\big(\frac{b^{(1)}_2}{n_3},\ldots,\frac{b^{(1)}_2}{n_g}\big) = \frac{n_2b^{(1)}_2}{n_3} \gcd\big(\frac{\bar{\beta}_2}{n_3},\ldots,\frac{\bar{\beta}_2}{n_g}\big)$ as $n_3,\ldots, n_g$ divide $\lbeta_2$. To understand the combinatorics of $\E_2$ with $\E_1$, we can make use of Proposition~\ref{prop:intersection-with-previous-divisor}; the components of $\E_1$ are separated, each of them is intersected by precisely one component of $\E_2$, each intersection consists of only one point, and each component of $\E_2$ intersects \[\frac{n_2\lcm(n_3,\ldots, n_g)}{\lcm(n_2,\ldots, n_g)}\] components of $\E_1$, which is precisely the quotient of the number of components of $\E_1$ and $\E_2$. Finally, the strict transform $\hat{Y}$ of $Y$ intersects $\E_2$ only in the point $P_2$. Figure~\ref{fig:second-step} shows the global situation in $\hat{S}$ so far (for $g \geq 4$). The divisors are again visualized in a simplified way, and the intersection $\E_1 \cap \E_2$ is represented by white circles to emphasize the difference with the other points. 

\begin{figure}[ht]
	\includegraphics{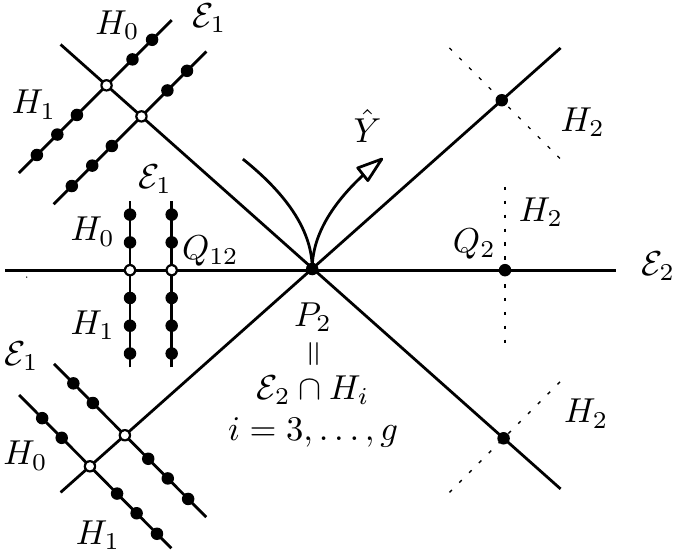}
	\caption{Step 2 in the resolution of $Y \subset S$ for $g \geq 4$.}
	\label{fig:second-step}
\end{figure}

\vspace{14pt}

As in Step 1, we make use of local coordinates to investigate the behavior around the singular points of $\hat{S}$. Note that $\hat{S}$ is smooth outside $\E_1 \cup \E_2$, and that it is again enough to consider the first two charts of the blow-up to understand the whole situation in $\E_2$.

\vspace{14pt}

\underline{Points in $\E_2 \setminus (\E_1 \cup \bigcup_{i=2}^g H_i)$}. The first chart is 
	\[U_0  = X\left(\!
		\begin{array}{c|cccc}
			1 & -1 & \frac{b_2^{(1)}}{n_2} & \ldots & \frac{b_2^{(1)}}{n_g} \\
			e_1 & -1 & \frac{n + b_2^{(1)}}{n_2} & \ldots & \frac{n + b_2^{(1)}}{n_g}
		\end{array}\!\right)
	= X \Big(e_1; -1, \frac{n_2\lbeta_2}{n_2}, \ldots, \frac{n_2\lbeta_2}{n_g} \Big),\] 
and we can compute the local equations of $\hat{S}$ and $\hat{Y}$ by pulling back~\eqref{eq:Shat-ggeq3} and~\eqref{eq:Yhat-ggeq3} via \[(x_0, x_2,\ldots, x_g) \longmapsto (x_0, x_0^{\frac{b_2^{(1)}}{n_2}} x_2, \ldots, x_0^{\frac{b_2^{(1)}}{n_g}} x_g).\] The total transform $\varphi_2^{-1} (\varphi_1^{-1} (Y))$ is given by $x_0^{n_2 \bar{\beta}_2} \hat{f}_2 = \cdots = x_0^{n_2 \bar{\beta}_2} \hat{f}_g = 0$, where 
	\[\left\{\begin{array}{l c l l}
		\hat{f}_2 & := &  x_2^{n_2} - 1 \\
		\hat{f}_3 & := &  x_3^{n_3} - x_0^{b_{3}^{(2)}} x_2^{b_{32}}\\
		& \vdots & & \\
		\hat{f}_g & := & x_g^{n_g} - x_0^{b_{g}^{(2)}} x_2^{b_{g2}} \cdots x_{g-1}^{b_{g(g-1)}}
	\end{array}\right.\]
correspond to the strict transform $\hat{Y}$, and $x_0^{n_2 \bar{\beta}_2}: \hat{S} \to \C$ to the exceptional divisor $\E_2$, see~\eqref{eq:def-b_i^{(k)}} and Lemma~\ref{lemma:pos-powers} for the definition and behavior of $b_i^{(2)} > 1$ for $i = 3, \ldots, g$. Here, we use again $\hat{f}_i$ to avoid complicating the notation. The strict transform $\hat{S}$ is defined by \[\hat{f}_i + \lambda_{i+1} \hat{f}_{i+1} + \hat{f}_2 R^{(1)}_i (x_0,x_0^{\frac{b^{(1)}_2}{n_2}}x_2,\ldots,x_0^{\frac{b^{(1)}_i}{n_i}}x_i) = 0,\qquad i = 2,\ldots, g-1,\] and $H_i$ for $i = 2,\ldots, g$ is still given by $\{x_i = 0\} \cap \hat{S}$. Observe that the divisor $\E_1$ is not visible in this chart. Similarly as in Step 1, the ambient space at points of $\E_2 \setminus \bigcup_{i=2}^g H_i$ is smooth, and the standard Jacobian criterion can be applied to see that $\hat{S}$ is also smooth at these points. To compute the multiplicity of $\E_2$, we consider a generic point $Q = [(0,a_2,\ldots,a_g)]$ in ${\E_2 \setminus \bigcup_{i=2}^g H_i}$ with $a_i \in \C^{\ast}$. The order of its stabilizer subgroup is $\gcd(e_1, \frac{n_2\lbeta_2}{n_2}, \ldots, \frac{n_2\lbeta_2}{n_g})$, and, as germs, $(U_0,Q) = \big(X (e_1; -1, \frac{n_2\lbeta_2}{n_2}, \ldots,\frac{n_2\lbeta_2}{n_g}),Q\big)$ equals
	\begin{equation}\label{eq:generic-E2}
		\bigg(X\Big(\gcd\Big(e_1, \frac{n_2\lbeta_2}{n_2}, \ldots, \frac{n_2\lbeta_2}{n_g}\Big);-1,0,\ldots, 0\Big),Q\bigg)\simeq (\C^g,0).
	\end{equation}
Under this isomorphism, the function $x_0^{n_2 \bar{\beta}_2}:U_0 \to \C$ becomes $x_0^{N_2}: \C^g \to \C$ with 
\[N_2 := \frac{n_2\lbeta_2}{\gcd \Big( e_1, \frac{n_2\lbeta_2}{n_2}, \ldots, \frac{n_2\lbeta_2}{n_g} \Big)} = \lcm \Big(\frac{\lbeta_2}{e_2}, n_2, \ldots, n_g \Big)\] the required multiplicity.

\vspace{14pt}

\underline{Points in the intersection $\E_2 \cap H_2$}. The order of the stabilizer subgroup of a point $Q_2 = [(0,0,a_3,\ldots,a_g)] \in \E_2 \cap H_2$ is $\gcd(e_1, \frac{n_2 \bar{\beta}_2}{n_3}, \ldots, \frac{n_2 \bar{\beta}_2}{n_g})$. Changing the variables as in~\eqref{eq:points-Q1}, one gets the following situation at $[(x_0,x_2)]$:
	\begin{equation}\label{eq:points-Q2}
		\left\{\begin{aligned}
			& \hat{S} = X \bigg( \gcd\Big(e_1, \frac{n_2 \bar{\beta}_2}{n_3}, \ldots, \frac{n_2 \bar{\beta}_2}{n_g}\Big);-1, \lbeta_2 \bigg) \\
			& \E_2: \ x_0^{n_2 \bar{\beta}_2} = 0, \qquad H_2: \ x_2 = 0,
		\end{aligned}\right.
	\end{equation}
and the total transform of $Y$ defines a $\Q$-normal crossings divisor around these points.

\vspace{14pt}

\underline{Points in the intersection $\E_2 \cap \E_1$}. These points cannot be seen in the first chart. Therefore, we consider the second chart $U_1$ where the exceptional divisor $\E_2$ corresponds to $x_2=0$; it is given by 
	\[\left(\!\begin{array}{c|ccccc}
		\frac{b_2^{(1)}}{n_2} & 1 & -1 & \frac{b_2^{(1)}}{n_3} & \cdots & \frac{b_2^{(1)}}{n_g} \\
		\frac{b_2^{(1)}}{n_2} e_1 & - \bar{\beta}_2 & \frac{n}{n_2} & 0 & \cdots & 0
	\end{array}\!\right)\]
via \[(x_0, x_2,\ldots, x_g) \longmapsto (x_0 x_2, x_2^{\frac{b_2^{(1)}}{n_2}}, x_2^{\frac{b_2^{(1)}}{n_3}} x_3, \ldots, x_2^{\frac{b_2^{(1)}}{n_g}} x_g).\] A point $Q_{12} \in \E_2 \cap \E_1$ is in this chart of the form $[(0,0,a_3,\ldots,a_g)]$ for some $a_i \in \C^{\ast}$. The stabilizer subgroup of $Q_{12}$ is the product of two cyclic groups of orders $\gcd(\frac{b_2^{(1)}}{n_2},\ldots,\frac{b_2^{(1)}}{n_g}) = \frac{n_2 \bar{\beta}_2 - n_1 \bar{\beta}_1}{\lcm(n_2,\ldots,n_g)}$ and $\frac{b_2^{(1)}}{n_2} e_1 = ( n_2 \bar{\beta}_2 - n_1 \bar{\beta}_1 ) e_2$, and one obtains the following local situation around $Q_{12}$ in the variables $x_0$ and $x_2$:
	\begin{equation}
		\left\{\begin{aligned}\label{eq:Q12}
			& \hat{S} = X \left(\!\!\! \begin{array}{c|cc}
			\frac{n_2 \bar{\beta}_2 - n_1 \bar{\beta}_1}{\lcm(n_2,\ldots,n_g)} & 1 & -1 \\[0.2cm]
			(n_2 \bar{\beta}_2 - n_1 \bar{\beta}_1) e_2 & - \bar{\beta}_2 & \frac{n}{n_2}
			\end{array} \!\!\right) \\[0.1cm]
			& \E_1: \ x_0^n = 0, \qquad \E_2: \ x_2^{n_2\bar{\beta}_2} = 0.
		\end{aligned}\right.
	\end{equation}
Hence, the total transform $\varphi_2^{-1}( \varphi_1^{-1} (Y))$ has $\Q$-normal crossings at each of the points in the intersection $\E_2 \cap \E_1$. Note that these data are compatible with the case $g=2$.

\vspace{14pt}

\underline{The point $P_2 = \E_2 \cap H_i$ for $i = 3,\ldots,g$}. This point considered in the first chart $U_0 = X(e_1; -1, \frac{n_2\lbeta_2}{n_2}, \ldots, \frac{n_2\lbeta_2}{n_g})$ is given by $P_2 = [(0,1,0,\ldots,0)]$, and its stabilizer subgroup has order $\gcd ( e_1, \frac{n_2\lbeta_2}{n_2}) = e_2$. Hence, as germs, \[\bigg( X \Big( e_1; -1, \frac{n_2\lbeta_2}{n_2}, \ldots, \frac{n_2\lbeta_2}{n_g} \Big), P_2 \bigg) = \bigg( X \Big( e_2; -1,0, \frac{n_2\lbeta_2}{n_3}, \ldots, \frac{n_2\lbeta_2}{n_g} \Big), P_2 \bigg).\] The idea is to follow the same procedure as the one we used for the point $P_1$ in Step 1. We use the change of variables $x_2 \mapsto x_2 + 1$ and $x_i \mapsto x_i$ for $i=0,3,\ldots,g$ to get a chart centered around the origin and we discuss two cases separately.

\vspace{14pt}

If $g = 3$, then $\E_2$ is irreducible, and using the Implicit Function Theorem, one easily sees that $\hat S \simeq X \big( n_3; -1,\frac{n_2\lbeta_2}{n_3} \big)$ with variables $[(x_0,x_3)]$ on which $H_3 = \{x_3 = 0\}$ and the total transform of $Y$ is given by $x_0^{n_2\lbeta_2} ( x_3^{n_3}  -  x_0^{b_3^{(2)}} ) = 0.$ The first factor represents the exceptional divisor $\E_2$, and the other the strict transform of $Y$. 

\vspace{14pt}

If $g \geq 4$, then the germ $(\hat{S},P_2 = [(0,\ldots,0)] )$ can be described inside the ambient space $X\left(e_2; -1,\frac{n_2\lbeta_2}{n_3}, \ldots, \frac{n_2\lbeta_2}{n_g}\right)$ in the variables $x_0,x_3,\ldots,x_g$ by equations of the form
	\begin{equation}\label{eq:Shat-ggeq4}
		\left\{\begin{array}{lcl}
			x_3^{n_3} - x_0^{b_3^{(2)}} + \lambda_4(x_4^{n_4} - x_0^{b_4^{(2)}}x_3^{b_{43}}) + (x_3^{n_3} - x_0^{b_3^{(2)}}) R^{(2)}_3(x_0,x_3) & = & 0 \\[5pt]
			x_4^{n_4} - x_0^{b_4^{(2)}}x_3^{b_{43}} + \lambda_5(x_5^{n_5} - x_0^{b_5^{(2)}}x_3^{b_{53}}x_4^{b_{54}})+ ( x_3^{n_3} - x_0^{b_3^{(2)}} ) R^{(2)}_4(x_0,x_3,x_4) & = & 0 \\
			& \vdots & \\
			x_{g-1}^{n_{g-1}} - x_0^{b_{g-1}^{(2)}} x_3^{b_{(g-1)3}} \cdots x_{g-2}^{b_{(g-1)(g-2)}} + \lambda_g(x_g^{n_g} - x_0^{b_g^{(2)}}x_3^{b_{g3}} \cdots x_{g-1}^{b_{g(g-1)}}) \\
			\hspace{170pt} + ( x_3^{n_3} - x_0^{b_3^{(2)}} ) R^{(2)}_{g-1}(x_0,x_3,\ldots, x_{g-1}) & = & 0,
		\end{array}\right.
	\end{equation}
where every $R^{(2)}_i(x_0,x_3,\ldots, x_i) \in \mathbb{C} \{ x_0, x_3,\ldots, x_i \}$ with $R_i^{(2)}(0,x_3,\ldots,x_i)=0$, and the total transform of $Y$ is given by
	\begin{equation}\label{eq:Yhat-ggeq4}
		\left\{ \begin{array}{lcl}
			x_0^{n_2\lbeta_2}(x_3^{n_3} - x_0^{b_3^{(2)}}) &=& 0 \\
			x_0^{n_2\lbeta_2}(x_4^{n_4} - x_0^{b_4^{(2)}}x_3^{b_{43}}) &=& 0 \\
			& \vdots &  \\
			x_0^{n_2\lbeta_2}(x_g^{n_g} - x_0^{b_g^{(2)}}x_3^{b_{g3}} \cdots x_{g-1}^{b_{g(g-1)}}) &=& 0.
		\end{array}\right.
	\end{equation}
Here, $x_0^{n_2\lbeta_2} = 0$ corresponds to the exceptional divisor $\E_2$, and $x_i = 0$ to $H_i$ for $i = 3,\ldots, g$.

\vspace{14pt}

The composition $\varphi_1 \circ \varphi_2$ is an embedded $\Q$-resolution of $Y \subset S$ except at the point $P_2$. Hence, in Step 3, we will blow up at this point. If $g=3$, this third step will finish the resolution. If $g \geq 3$, one sees in~\eqref{eq:Shat-ggeq4} and \eqref{eq:Yhat-ggeq4} that $x_2$ is eliminated and that the situation is the same as in the beginning of Step 2 but in one variable less, see~\eqref{eq:Shat-ggeq3} and~\eqref{eq:Yhat-ggeq3}. The idea is to repeat this procedure until we obtain a cusp in the $(g-1)$th step in a cyclic quotient singularity with variables $x_0$ and $x_g$. Then, one additional blow-up resolves the singularity. Because the next steps will be essentially the same as Step 2, we consider all of them simultaneously in Step $k$ for $k \geq 2$.

\subsubsection{Step $k$: weighted blow-up \texorpdfstring{$\pi_k$}{pik} at \texorpdfstring{$P_{k-1}$}{Pk-1} with weights \texorpdfstring{$\omega_k$}{wk}} Let $k \in \{ 2,\ldots,g\}$ and assume that the first $k-1$ blow-ups have already been performed. Recall that we denote by $\E_1, \ldots, \E_{k-1}$ the exceptional divisors of the corresponding weighted blow-ups $\varphi_1,\ldots,\varphi_{k-1}$ with respect to the weights $\omega_1,\ldots,\omega_{k-1}$, respectively.  We again consider two cases.

\vspace{14pt}

If $k = g$, then at the end of the $(g-1)$th step, the total transform $(\varphi_1 \circ \cdots \circ \varphi_{g-1})^{-1}(Y)$ is given by $x_0^{n_{g-1} \bar{\beta}_{g-1}} (x_g^{n_g} - x_0^{b_g^{(g-1)}}) = 0$ in $\hat{S} = X(n_g;-1, \frac{n_{g-1} \bar{\beta}_{g-1}}{n_g})$ around $P_{g-1} = [(0,0)]$. The blow-up $\pi_g = \varphi_g$ at $P_{g-1}$ with respect to $\omega_g = (1,\frac{b_g^{(g-1)}}{n_g})$ yields an irreducible exceptional divisor $\E_g = \P^{1}_{\omega_g}(n_g; -1, \frac{n_{g-1}\bar\beta_{g-1}}{n_g})\simeq \P^1$ with multiplicity $N_g = n_{g-1} \bar{\beta}_{g-1} + b_g^{(g-1)} = n_g \bar\beta_g$. The intersection $\E_g \cap H_g$ consists of a single point, and the equation of the total transform of $Y$ at this point is $x_0^{n_g \bar\beta_g}: X(n_g;-1,\bar\beta_g) \to \C$. The intersection $\E_g \cap \E_{g-1}$ consists also of one point around which we have the function
	\[x_0^{n_{g-1} \bar{\beta}_{g-1}} x_g^{n_g \bar\beta_{g}}:
		X\left(\begin{array}{c|cc}
		\frac{n_g\lbeta_g - n_{g-1}\lbeta_{g-1}}{n_g} & 1 & -1 \\
		n_g\lbeta_g - n_{g-1}\lbeta_{g-1} & -\lbeta_g & \frac{n_{g-1}\bar\beta_{g-1}}{n_g}
	\end{array} \right) \longrightarrow \C.\]
Finally, the strict transform $\hat{Y}$ is smooth and intersects $\E_g$ in a transversal way at a smooth point of $\hat{S}$. Hence, the morphism $\varphi := \varphi_1 \circ \cdots \circ \varphi_g: \hat{S} \to S$ defines an embedded $\mathbb{Q}$-resolution of $Y \subset S$, cf.~Figure~\ref{fig:second-step-g2}.

\vspace{14pt}

Assume now that $2 \leq k \leq g-1$. In the first chart of $\varphi_{k-1}$ centered at $P_{k-1}$, one has \[P_{k-1} = [(0,\ldots,0)] \in X \left( e_{k-1}; -1, \frac{n_{k-1} \bar{\beta}_{k-1}}{n_k},\ldots, \frac{n_{k-1} \bar{\beta}_{k-1}}{n_g} \right)\] in the variables $(x_0,x_k,\ldots,x_g)$, and the strict transforms $\hat{S}$ and $\hat{Y}$ are given by equations as in~\eqref{eq:Shat-ggeq3} and~\eqref{eq:Yhat-ggeq3}, respectively. The strict transform $\E_{k-1}$ is given by $x_0^{n_k\lbeta_k}=0$, and $H_i = \{x_i = 0\} \cap \hat{S}$ for $i = k,\ldots,g$. Let $\pi_k$ be the weighted blow-up at $P_{k-1}$ with respect to $\omega_k = ( 1, \frac{b_k^{(k-1)}}{n_k}, \ldots, \frac{b_k^{(k-1)}}{n_g}),$ where $b_k^{(k-1)}  = n_k\lbeta_k - n_{k-1} \lbeta_{k-1}$ is divisible by $e_{k-1} = n_ke_k = n_k\cdots n_g$. Let $E_k \simeq \P^{g-k+1}_{\omega_k} \big( e_{k-1}; -1, \frac{n_{k-1} \bar{\beta}_{k-1}}{n_k}, \ldots, \frac{n_{k-1} \bar{\beta}_{k-1}}{n_g} \big)$ be the exceptional divisor of $\pi_k$ and let $\varphi_k := \pi_k|_{\hat{S}}:\hat{S}\to \hat{S}$ be the restriction map with exceptional divisor $\E_k := E_k \cap \hat{S}$. Once more, we split the exposition in different parts.

\vspace{14pt}

\underline{Global situation}. The new exceptional divisor $\E_k$ is given in homogeneous coordinates $[x_0:x_k:\ldots:x_g] \in \P^{g-k+1}_{\omega_k} \big( e_{k-1}; -1, \frac{n_{k-1} \bar{\beta}_{k-1}}{n_k}, \ldots, \frac{n_{k-1} \bar{\beta}_{k-1}}{n_g} \big)$ by the equations
	\begin{equation}\label{eq:Ek-homog}
		\left \{ \begin{array}{clccccl}
			x_k^{n_k} & - & x_0^{b_k^{(k-1)}} & + & \lambda_{k+1}x_{k+1}^{n_{k+1}} & = & 0 \\[5pt]
			& & x_{k+1}^{n_{k+1}} & + & \lambda_{k+2}x_{k+2}^{n_{k+2}} & = & 0 \\
			& & & \vdots & & \\
			& & x_{g-1}^{n_{g-1}} & + & \lambda_g x_g^{n_g} & = & 0,
		\end{array}\right.
	\end{equation}
and has \[\frac{n_{k+1} \cdots \, n_g}{\lcm(n_{k+1},\ldots, n_g)} = \frac{e_k}{\lcm(n_{k+1},\ldots, n_g)}\] irreducible components that contain the point $P_k = [1:1:0:\ldots:0]$ and are pairwise disjoint outside $P_k$ by Proposition~\ref{prop:number-of-comp}. Note that $\E_k$ is irreducible if $k=g-1$, and that $P_k = \E_k \cap H_i$ for $i = k+1, \ldots, g$. With Proposition~\ref{prop:number-of-comp}, one can also compute that $\E_k$ has \[\frac{e_{k-1}}{\lcm(n_k,\ldots, n_g)}\] intersections with $\E_{k-1}$ and \[\frac{\lbeta_k}{\lcm(\frac{\lbeta_k}{e_k},n_{k+1},\ldots, n_g)}\] with $H_k$, where the cardinality of $\E_k \cap \E_{k-1}$ is precisely the number of components of $\E_{k-1}$. Furthermore, Proposition~\ref{prop:intersection-with-previous-divisor} tells us that the components of $\E_{k-1}$ are disjoint, and that the intersections of $\E_k$ and $\E_{k-1}$ are equally distributed. Lastly, the strict transform $\hat{Y}$ of $Y$ and $\E_k$ intersect in the single point $P_k$. In the next step, we will blow up this point.

\vspace{14pt}

\underline{Points in $\E_k \setminus (\E_{k-1} \cup \bigcup_{i=k}^g H_i)$}. Outside the coordinate axes of $\E_k$, the Jacobian criterion can be used to check that $\hat{S}$ is smooth. Studying the stabilizer subgroup of a generic point in $\E_k \setminus (\E_{k-1} \cup \bigcup_{i=k}^g H_i)$ using local equations in the first chart as in~\eqref{eq:generic-E2}, one can compute the multiplicity $N_k$ of $\E_k$, which is equal to $\lcm( \frac{\lbeta_k}{e_k}, n_k, \ldots, n_g )$.

\vspace{14pt}

\underline{Points in the intersection $\E_k \cap H_k$}. The local situation around these points can be studied from the local charts as in~\eqref{eq:points-Q2} and becomes at $[(x_0,x_k)]$ the following:
	\begin{equation}
		\left\{\begin{aligned}
			& \hat{S} = X \bigg( \gcd \Big( e_{k-1}, \frac{n_k \bar{\beta}_k}{n_{k+1}}, \ldots,\frac{n_k \bar{\beta}_k}{n_g} \Big);-1,\lbeta_k \bigg) \\
			& \E_k: \ x_0^{n_k \bar{\beta}_k} = 0, \qquad H_k: \ x_k = 0.
		\end{aligned}\right.
	\end{equation}
Clearly, the total transform of $Y$ under $\varphi_1 \circ \cdots \circ \varphi_k$ is a $\Q$-normal crossings divisor around these points.

\vspace{14pt}

\underline{Points in the intersection $\E_k \cap \E_{k-1}$}. Using the second chart on which $\E_k$ corresponds to $x_k = 0$, the local equations at these points are given by 
	\begin{equation}
		\left\{\begin{aligned}
			& \hat{S} = X \left(\!\!\! \begin{array}{c|cc}
			\frac{n_k \bar{\beta}_k - n_{k-1} \bar{\beta}_{k-1}}{\lcm(n_k,\ldots,n_g)} & 1 & -1 \\[0.2cm]
			(n_k \bar{\beta}_k - n_{k-1} \bar{\beta}_{k-1}) e_k & - \bar{\beta}_k & \frac{n_{k-1}\lbeta_{k-1}}{n_k}
			\end{array} \!\!\right) \\[5pt]
			& \E_{k-1}: \ x_0^{n_{k-1} \bar{\beta}_{k-1}} = 0, \qquad \E_k: \ x_k^{n_k \bar{\beta}_k} = 0,
		\end{aligned}\right.
	\end{equation}
cf.~\eqref{eq:Q12}, and the total transform of $Y$ has again $\Q$-normal crossings at each of these points.
\vspace{14pt}

\underline{The point $P_k = \E_k \cap H_i$ for $i = k+1,\ldots,g$}. After centering the first chart around $P_k$, we distinguish for the last time two different cases.

\vspace{14pt}

If $k = g-1$, then $\hat S \simeq X \big( n_g; -1,\frac{n_{g-1}\lbeta_{g-1}}{n_g} \big)$ in the variables $x_0$ and $x_g$. The total transform $(\varphi_1 \circ \cdots \circ \varphi_{g-1})^{-1}(Y)$ of $Y$ is defined by the equation $x_0^{n_{g-1}\lbeta_{g-1}} (x_g^{n_g}  -  x_0^{b_g^{(g-1)}})= 0,$ where the exceptional divisor $\E_g$ is given by $x_0^{n_{g-1}\bar\beta_{g-1}} = 0$, the strict transform $\hat{Y}$ by $x_g^{n_g}  -  x_0^{b_g^{(g-1)}}= 0$, and $H_g$ by $x_g = 0$. 

\vspace{14pt}

If $2 \leq k \leq g-2$, then $\hat{S}$ is locally around $P_k = [(0,\ldots,0)]$ in $X\left(e_k; -1,\frac{n_k\lbeta_k}{n_{k+1}}, \ldots, \frac{n_k\lbeta_k}{n_g}\right)$ with the variables $x_0,x_{k+1},\ldots,x_g$ given by equations of the form
	\[\left\{ \begin{array}{lcl}
		x_{k+1}^{n_{k+1}} - x_0^{b_{k+1}^{(k)}} + \lambda_{k+2} ( x_{k+2}^{n_{k+2}} - x_0^{b^{(k)}_{k+2}}x_{k+1}^{b_{(k+2)(k+1)}} ) + ( x_{k+1}^{n_{k+1}}  -  x_0^{b^{(k)}_{k+1}} ) R^{(k)}_{k+1}(x_0,x_{k+1}) & = & 0 \\[5pt]
		x_{k+2}^{n_{k+2}} - x_0^{b^{(k)}_{k+2}} x_{k+1}^{b_{(k+2)(k+1)}}  + \lambda_{k+3} ( x_{k+3}^{n_{k+3}} - x_0^{b^{(k)}_{k+3}} x_{k+1}^{b_{(k+3)(k+1)}} x_{k+2}^{b_{(k+3)(k+2)}} ) \\
		\hfill + ( x_{k+1}^{n_{k+1}} - x_0^{b^{(k)}_{k+1}} ) R^{(k)}_{k+2}(x_0,x_{k+1},x_{k+2}) & = & 0 \\
		& \vdots &\\
		x_{g-1}^{n_{g-1}} - x_0^{b^{(k)}_{g-1}} x_{k+1}^{b_{(g-1)(k+1)}} \cdots x_{g-2}^{b_{(g-1)(g-2)}} + \lambda_g ( x_g^{n_g} - x_0^{b^{(k)}_g} x_{k+1}^{b_{g(k+1)}} \cdots x_{g-1}^{b_{g(g-1)}} ) \\
		\hspace{185pt} +  ( x_{k+1}^{n_{k+1}} - x_0^{b^{(k)}_{k+1}} ) R^{(k)}_{g-1}(x_0,x_{k+1},\ldots,x_{g-1}) & = & 0,
	\end{array}\right.\]
for some $R_i^{(k)}(x_0,x_{k+1},\ldots, x_i) \in \mathbb{C}\{x_0,x_{k+1},\ldots,x_i\}$ satisfying $R_i^{(k)}(0,x_{k+1},\ldots, x_i) = 0$. The total transform of $Y$ is defined by
	\[\left\{ \begin{array}{lcl}
		x_0^{n_k\lbeta_k} ( x_{k+1}^{n_{k+1}} - x_0^{b^{(k)}_{k+1}} ) & = & 0 \\
		x_0^{n_k\lbeta_k} ( x_{k+2}^{n_{k+2}} - x_0^{b^{(k)}_{k+2}}x_{k+1}^{b_{(k+2)(k+1)}} ) & = & 0 \\
		& \vdots & \\
		x_0^{n_k\lbeta_k} ( x_g^{n_g} - x_0^{b^{(k)}_g}x_{k+1}^{b_{g(k+1)}} \cdots x_{g-1}^{b_{g(g-1)}} ) & = & 0,
\end{array}\right.\]
where $\E_k = \{ x_0^{n_k\lbeta_k} = 0 \}$ and $H_i = \{ x_i = 0 \}$ for $i=k+1,\ldots,g$. 

\vspace{14pt}

To conclude, we have exactly the same situation as the one we had at the beginning of Step $k$ but in one variable less. Further blowing up at the point $P_k$ and repeating this procedure will lead after $g$ steps to an embedded $\Q$-resolution of $Y \subset S$ as illustrated in Figure \ref{fig:final-resolution}.

\begin{figure}[ht]
\includegraphics{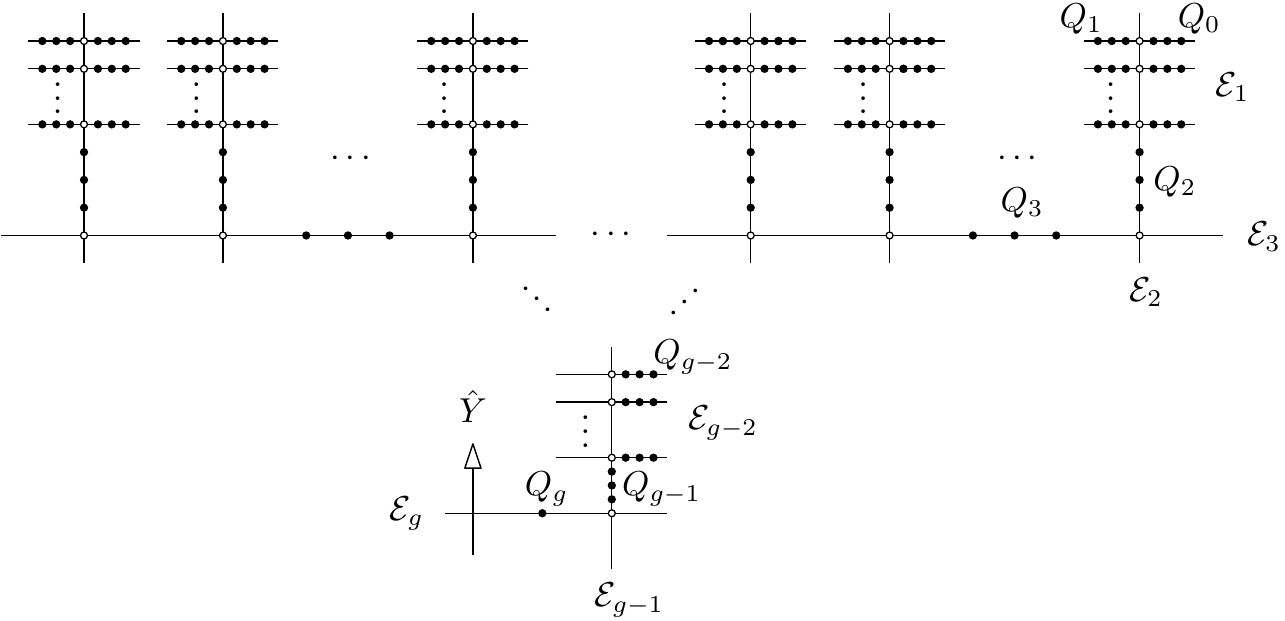}    
\caption{Resolution of $Y \subset S$.}
\label{fig:final-resolution}
\end{figure}

\subsection{Main result} We summarize the previous construction in the following result. 

\begin{theorem}\label{thm:resolutionY}
Let $Y \subset \C^{g+1}$ be a space monomial curve defined by the equations~\eqref{eq:equations-Y} with $g \geq 2$, and consider $Y$ as a Cartier divisor on a generic surface $S = S(\lambda_2,\ldots,\lambda_g) \subset \C^{g+1}$ given by~\eqref{eq:equations-S}, where $(\lambda_2,\ldots,\lambda_g)$ are chosen such that Section~\ref{RedCurveSurface} applies. There exists an embedded $\mathbb{Q}$-resolution $\varphi = \varphi_1 \circ \cdots \circ \varphi_g : \hat{S} \to S$ of $Y \subset S$ which is a composition of $g$ weighted blow-ups $\varphi_k$ with exceptional divisor $\E_k$ such that the pull-back of $Y$ is given by \[\varphi^{\ast}Y = \hat{Y} + \sum_{\smallmatrix 1 \leq k \leq g \\ 1 \leq j \leq r_k \endsmallmatrix} N_k \E_{kj},\] where $\E_k = \E_{k1} + \cdots + \E_{kr_k}$ is the decomposition of $\E_k$ into $r_k = \frac{e_k}{\lcm(n_{k+1},\ldots,n_g)}$ if $k=1,\ldots,g-2$ and $r_{g-1} = r_g=1$ irreducible components, and $N_k = \lcm ( \frac{\lbeta_k}{e_k}, n_k, \ldots, n_g )$ is the multiplicity of $\E_k$. Furthermore, each divisor $\E_k$ for $k = 2,\ldots, g-1$ only intersects $\E_{k-1}$ and $\E_{k+1}$, and $\E_g$ only intersects $\E_{g-1}$. Finally, for every $k = 2,\ldots, g$, the intersections of $\E_{k-1}$ and $\E_k$ are equally distributed; each of the components $\E_{kj}$ of $\E_k$ intersects precisely $\frac{r_{k-1}}{r_k}$ components of $\E_{k-1}$, each component $\E_{(k-1)j}$ of $\E_{k-1}$ is intersected by only one of the components of $\E_k$, and each non-empty intersection between two components $\E_{kj}$ and $\E_{(k-1)j'}$ consists of a single point. In particular, the dual graph of the resolution is a tree as in Figure~\ref{fig:dual-graph}.
\end{theorem}

\begin{figure}[ht]
\includegraphics{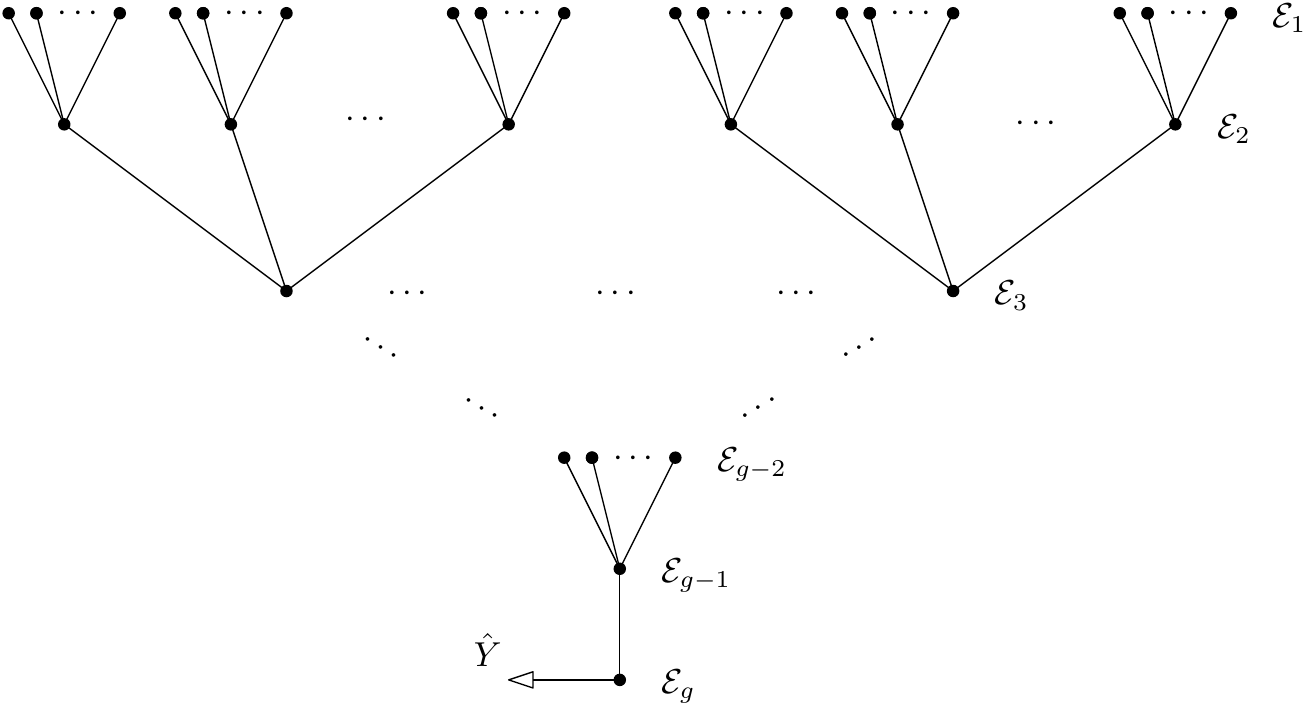}
\caption{Dual graph of the resolution of $Y \subset S$.}
\label{fig:dual-graph}
\end{figure}

\begin{remark}
Besides the monodromy zeta function, this resolution could also be used to compute other invariants associated with the curve singularity $Y \subset S$, such as the mixed Hodge structure on the cohomology of the Milnor fiber.	
\end{remark}


\section{The monodromy zeta function of a space monomial curve} \label{ZetaFunction}

Using the embedded $\Q$-resolution $\varphi: \hat{S} \rightarrow S$ of a space monomial curve $Y$ seen as a Cartier divisor on a generic surface $S$ constructed in the previous section, we will now compute the monodromy zeta function of $Y$. More precisely, we will compute the zeta function of monodromy $Z_{Y,0}^{mon}(t)$ of $Y \subset S$ at the origin with the A'Campo formula from Theorem~\ref{thm:ACampo-QEmb} in terms of $\varphi$. To this end, we still need to stratify the exceptional divisor such that the multiplicity defined in~\eqref{eq:def-mult} is constant along each stratum, and compute the Euler characteristic of these strata.

\vspace{14pt}

With Figure~\ref{fig:final-resolution}, we define a stratification of the exceptional divisor as follows. The first set of strata are the points of the intersection $\E_1 \cap H_0$, which we will all denote by $Q_0$; there are \[\frac{\lbeta_0}{\lcm(n_1,n_2, \ldots, n_g)}\] such points, see~\eqref{eq:points-E1-H0}. From~\eqref{eq:points-Q0}, we know that the local equation of $\E_0$ at each $Q_0$ is given by $x_1^n:X (\gcd(\frac{n}{n_1},\frac{n}{n_2},\ldots,\frac{n}{n_g});\frac{n}{n_0},-1) \rightarrow \C$. Hence, the multiplicity $m(\E_0,Q_0)$ is equal to \[m(\E_0,Q_0) = \frac{n}{\gcd(\frac{n}{n_1},\frac{n}{n_2},\ldots,\frac{n}{n_g})} = \lcm(n_1,\ldots, n_g).\] Analogously, each point in an intersection $\E_k \cap H_k$ for $k = 1,\ldots, g$ will be a stratum denoted by $Q_k$, the total number of such $Q_k$ is \[\frac{\lbeta_k}{\lcm(\frac{\lbeta_k}{e_k},n_{k+1}, \ldots, n_g)},\] and the multiplicity at each such point is $m(\E_k,Q_k)= \lcm(\frac{\lbeta_k}{e_k},n_{k+1},\ldots,n_g).$

\begin{remark} \label{rmk:order-small-groups-g=2}
	For $g = 2$, the resolution was already illustrated in Figure~\ref{fig:second-step-g2}, together with the order of the underlying small group at the points $Q_0$, $Q_1$ and $Q_2$. This provides another way of computing the multiplicity at these points. For example, at $Q_0$, we know that $\E_0$ is given by $x_1^n:X (\gcd(\frac{n}{n_1},\frac{n}{n_2});\frac{n}{n_0},-1) \rightarrow \C$. Using the morphism $[(x_0,x_1)] \mapsto [(x_0,x_1^{n_{02}n_{12}})]$, the space $X (\gcd(\frac{n}{n_1},\frac{n}{n_2});\frac{n}{n_0},-1)$ can be normalized into $X (\frac{n_0}{n_{02}};\frac{n_1n_2}{n_{02}n_{12}},-1)$ on which $\E_0$ is locally given by the function $x_1^{\frac{n}{n_{02}n_{12}}}$. This yields the same multiplicity. In general, one could also first normalize the space around the points to compute the multiplicity. 
\end{remark}

Another set of strata are the intersection points $\E_k\cap \E_{k+1}$ for $k = 1,\ldots, g-1$, denoted by $Q_{k(k+1)}$. For every $k = 1,\ldots,g-1$, the number of points $Q_{k(k+1)}$ is equal to the number of irreducible components of $\E_k$, see Theorem~\ref{thm:resolutionY}, and the multiplicity at these points can be computed from the results in the previous section: for example, if $g \geq 3$ and $k = 1$, it can be computed from~\eqref{eq:Q12} with the more general definition of multiplicity introduced in~\cite{Ma2}. As these strata will not contribute to the zeta function of monodromy, see Theorem~\ref{thm:ACampo-QEmb}, we will not go into more detail. Similarly, the intersection point $\E_g \cap \hat{Y}$ is a stratum that we do not have to consider. The last set of strata are the parts of the irreducible components $\E_{kj}$ for $j=1,\ldots, r_k$ of $\E_k$ for each $k = 1,\ldots, g$ that are not yet contained in the previous strata. Because all $\E_{kj}$ for fixed $k$ have the same behavior, we will consider them at once; we introduce \[\check{\E}_k := \left\{\begin{array}{ll}
		\E_1 \setminus ((\E_1 \cap H_0) \cup (\E_1\cap H_1) \cup (\E_1\cap \E_2)) & \text{for } k = 1 \\
		\E_k \setminus ((\E_k \cap H_k) \cup (\E_k\cap \E_{k-1}) \cup (\E_k \cap \E_{k+1})) & \text{for } k = 2,\ldots, g-1 \\
		\E_g \setminus ((\E_g \cap H_g) \cup (\E_g\cap \E_{g-1}) \cup (\E_g \cap \hat{Y})) & \text{for } k = g. \\
	\end{array}\right.\]
The multiplicity along each of these `strata' $\check{\E}_k$ is equal to the multiplicity of $\E_k$ given by $N_k = \lcm(\frac{\lbeta_k}{e_k},n_k,\ldots, n_g)$. It remains to compute their Euler characteristics.

\vspace{14pt}

The Euler characteristic of $\check{\E}_g$ is easy to compute: as $\E_g \simeq \P^1$, we find $\chi(\check{\E}_g) = -1$. The other Euler characteristics can be computed from the following proposition, in which we work in the same situation as Proposition~\ref{prop:intersection-with-previous-divisor}. Because of the symmetry in the variables $x_2,\ldots, x_g$, the result is written in such a way that it is independent of the choice of chart in the proof, cf. Proposition~\ref{prop:number-of-comp} and, in particular, Remark~\ref{rmk:symm-formula}.  

\begin{prop}\label{prop:euler-char}
Consider the quotient $\P^r_{(p_0, \ldots, p_r)}(d;a_0, \ldots, a_r)$ of some weighted projective space $\P^r_{(p_0, \ldots, p_r)}$ under an action of type $(d;a_0, \ldots, a_r)$ with $r\geq 2$. Let $\E$ be defined in this space by a system of equations
	\[\left \{ \begin{array}{clccccl}
 		x_0^{m_0} & + & x_1^{m_1}  &+& x_2^{m_2} &=& 0 \\
  		&  & x_2^{m_2}  &+& x_3^{m_3} &=& 0  \\
 		& & & \vdots & & & \\
 		& & x_{r-1}^{m_{r-1}}   &+& x_r^{m_r} &=& 0
	\end{array}\right.\]
for positive integers $m_i$ such that $d \mid a_im_i$ for $i = 0 ,\ldots, r$ and such that each equation is weighted homogeneous with respect to the weights $(p_0, \ldots, p_r)$. Assume that the intersection of $\E$ with $\{x_i = 0\}$ for $i = 2,\ldots, r$ only consists of one fixed point $A$, and that $a_ip_j - a_jp_i = 0$ for all $i,j \in \{1, \ldots, r\}$. Then, $\chi\Big(\E \setminus \bigcup_{i=0}^r\{x_i = 0\}\Big)$ is given by \[-\frac{m_1\cdots m_r \cdot \gcd\big(dP\cdot(p_0,\ldots, p_r),(p_0Q - a_0P)\cdot(p_1,\ldots, p_r)\big)}{dp_0P},\] where $P := \prod_{i=1}^rp_i$ and $Q := a_i \prod_{j=1,j\neq i}^rp_j$ for $i = 1,\ldots, r$.
\end{prop}
 
To prove this result, we will reduce the problem of computing this Euler characteristic to computing the less complicated Euler characteristic considered in the next lemma.

\begin{lemma}\label{lemma:euler-char-plane-curve} Let $C$ in $\P^2_{(p_0,p_1,p_2)}(d;a_0,a_1,a_2)$ be defined by a single equation of the form $x_0^{m_0} + x_1^{m_1} + x_2^{m_2} = 0$ which is weighted homogeneous with respect to the weights $(p_0,p_1,p_2)$. Put $K = p_0m_0 = p_1m_1 = p_2m_2$, and let $M_i$ for $i = 0,1,2$ be the $2\times 2$-minor of
	\[\left(\begin{matrix}
 		p_0 & p_1 & p_2 \\
 		a_0 & a_1 & a_2 \\
	\end{matrix}\right)\]
where the column of $p_i$ is removed. Then, we have \[\chi\Big(C \setminus \bigcup_{i=0}^2\{x_i = 0\}\Big) = - \frac{K^2\cdot\gcd\big(d\cdot(p_0,p_1,p_2),M_0,M_1,M_2\big)}{dp_0p_1p_2}.\] 
\end{lemma}

\begin{proof}
We will once more simplify the problem of computing this Euler characteristic by looking at an easier Euler characteristic. More precisely, we consider the curve $\tilde C$ in $\P^2$ defined by $x_0^K + x_1^K + x_2^K = 0$. As this is a smooth curve of degree $K$, we know its genus \[g(\tilde C) = \frac{(K-1)(K-2)}{2},\] and, hence, its Euler characteristic $\chi(\tilde C) = 2 - 2g(\tilde C) = -K^2 + 3K$. Since each intersection $\tilde C \cap \{x_i = 0\}$ for $i=0,1,2$ consists of $K$ points, we find that ${\chi(\tilde C \setminus \bigcup_{i=0}^2\{x_i = 0\})} = -K^2.$ From this result, we can deduce $\chi(C \setminus \bigcup_{i=0}^2\{x_i = 0\})$ by considering the well-defined surjective morphism \[h: \P^2 \setminus \bigcup_{i=0}^2\{x_i=0\} \longrightarrow \P^2_{(p_0,p_1,p_2)}(d;a_0,a_1,a_2) \setminus \bigcup_{i=0}^2\{x_i=0\} : [x_0:x_1:x_2] \mapsto [x_0^{p_0}:x_1^{p_1}:x_2^{p_2}],\] under which $h^{-1}\big(C \setminus \bigcup_{i=0}^2\{x_i = 0\}\big) = \tilde C \setminus \bigcup_{i=0}^2\{x_i = 0\}$. We claim that $h$ is a covering map of degree \[D = \frac{dp_0p_1p_2}{\gcd(d\cdot \gcd(p_0,p_1,p_2),M_0,M_1,M_2)}.\] Then, indeed, \[\chi\Big(C \setminus \bigcup_{i=0}^2\{x_i = 0\}\Big) = \frac{\chi\big(\tilde C \setminus \bigcup_{i=0}^2\{x_i = 0\}\big)}{D} = - \frac{K^2\cdot\gcd\big(d\cdot\gcd(p_0,p_1,p_2),M_0,M_1,M_2\big)}{dp_0p_1p_2}.\] First, to show that $h$ is a covering map, one can see that it is enough to show that $h$ is a local homeomorphism. To prove the latter, we can work locally around a point $x \in \P^2 \setminus \bigcup_{i=0}^2\{x_i=0\}$ by considering the chart where $x_0 \neq 0$: \[h_0: \C^2 \setminus \bigcup_{i=1}^2\{x_i = 0\} \longrightarrow X \left(\begin{array}{c|cc} p_0 & p_1 & p_2 \\ dp_0 & M_2 & M_1\end{array} \right) \setminus \bigcup_{i=1}^2\{x_i = 0\}: (x_1,x_2) \mapsto [(x_1^{p_1},x_2^{p_2})].\] Because $X \left(\begin{smallmatrix} p_0 \\ dp_0\end{smallmatrix} \middle| \begin{smallmatrix}  p_1 & p_2 \\  M_2 & M_1 \end{smallmatrix} \right)\setminus \bigcup_{i=1}^2\{x_i = 0\}$ is smooth at $h_0(x)$, we can further reduce to showing that \[\Big(\C^2\setminus \bigcup_{i=1}^2\{x_i = 0\}, x\Big) \longrightarrow \Big(\C^2 \setminus \bigcup_{i=1}^2\{x_i = 0\}, h_0(x)\Big): (x_1,x_2) \mapsto (x_1^{p_1},x_2^{p_2})\] is a local homeomorphism, which is clearly true. Second, to find the degree of $h$, we can still work with $h_0$ on the chart where $x_0 \neq 0$. Because the morphism $h_0$ can be decomposed into the morphism $\sigma: \C^2 \rightarrow \C^2$ defined by $(x_1,x_2) \mapsto (x_1^{p_1},x_2^{p_2})$ and the projection $\text{pr}: \C^2 \rightarrow X \left(\begin{smallmatrix} p_0 \\ dp_0\end{smallmatrix} \middle| \begin{smallmatrix}  p_1 & p_2 \\  M_2 & M_1 \end{smallmatrix} \right),$ its degree is equal the product of the degrees of $\sigma$ and $\text{pr}$. Clearly, the morphism $\sigma$ has degree $p_1p_2$. For the degree of $\text{pr}$, the result in~\cite[Lemma 5.1]{AMO2} tells us that this is equal to \[\frac{dp_0}{\gcd(d\cdot \gcd(p_0,p_1,p_2),M_0,M_1,M_2)}.\] Together, these degrees yield the correct expression for the degree $D$.
\end{proof}

In the proof of Proposition~\ref{prop:euler-char}, we will work similarly as in the proof of Lemma~\ref{lemma:euler-char-plane-curve}: we will construct a covering from which the Euler characteristic of $\E \setminus \bigcup_{i=0}^r \{x_i = 0\}$ can be easily computed. To find the degree of this covering, we will use the following lemma.

\begin{lemma}\label{lemma:degree}  
Consider a cyclic quotient space $X$ of the form $X(\frac{K}{k};\frac{K}{k_0},\ldots, \frac{K}{k_r})$ where $r \geq 2$ and $k, k_0,\ldots, k_r \mid K$. Let $\E$ in $X$ be defined by
	\[\left\{\begin{array}{rcc}
		x_0^{k_0} + x_1^{k_1} & = & c_1 \\
		x_2^{k_2} & = & c_2 \\
		& \vdots & \\
		x_r^{k_r} & = &c_r
	\end{array}\right.\]
for some constants $c_i \in \C \setminus \{0\}$, and denote by $N$ its number of irreducible components. Consider also the cyclic quotient space $\tilde X = X(\frac{K}{k};\frac{K}{k_0},\frac{K}{k_1})$ and $\tilde \E $ in $\tilde X$ defined by the single equation $x_0^{k_0} + x_1^{k_1} = c_1$. The degree of the projection $\text{pr}: \E \setminus \bigcup_{i=0}^r\{x_i = 0\} \rightarrow {\tilde \E \setminus \bigcup_{i=0}^1\{x_i = 0\}}$ given by $[(x_0,\ldots,x_r)]\mapsto [(x_0,x_1)]$ is \[\frac{KN\cdot\gcd(\frac{K}{k},\frac{K}{k_0},\ldots, \frac{K}{k_r})}{k\cdot \gcd(\frac{K}{k},\frac{K}{k_0},\frac{K}{k_1}) \cdot \gcd(\frac{K}{k},\frac{K}{k_2},\ldots,\frac{K}{k_r})}.\]
\end{lemma}

\begin{proof}
First of all, the projection $pr$ is a covering map: as in the proof of Lemma~\ref{lemma:euler-char-plane-curve}, it suffices to see that $pr$ is a local homeomorphism around every point $x \in \E \setminus \bigcup_{i=0}^r\{x_i = 0\}$. In this case, because $X$ and $\tilde X$ are smooth around $x$ and $pr(x)$, respectively, the problem is equivalent to showing that the projection
\begin{align*}
	\Big(\E \setminus \bigcup_{i=0}^r\{x_i = 0\}, x\Big) \subset & ~\Big(\C^{r+1} \setminus \bigcup_{i=0}^r\{x_i = 0\}, x\Big) \\
	 \longrightarrow & ~\Big(\tilde \E \setminus \bigcup_{i=0}^1\{x_i = 0\}, pr(x)\Big)  \subset \Big(\C^2 \setminus \bigcup_{i=0}^1\{x_i = 0\}, pr(x)\Big)
\end{align*}
is a local homeomorphism, which is again easy to see. To compute the degree of $pr$, we count the number of elements in $\text{pr}^{-1}([(a_0,a_1)])$ of a point $[(a_0,a_1)] \in \tilde \E \setminus \bigcup_{i=0}^1\{x_i = 0\}$. These elements are of the form $[(\xi^{\frac{K}{k_0}} a_0, \xi^{\frac{K}{k_1}} a_1, b_2,\ldots,b_r)]$ for some $\xi \in \mu_{\frac{K}{k}}$ and ${b_i \in \C}$ for $i = 2,\ldots, r$ satisfying $b_i^{k_i} = c_i$. Note that the irreducible components of $\E$ are pairwise disjoint and given by $\{[(x_0,x_1,b_2,\ldots, b_r)]\mid x_0^{k_0} + x_1^{k_1} = c_1\} $ for some fixed solution $[(b_2, \ldots, b_r)]$ of $x_2^{k_2} - c_2 = \cdots = x_r^{k_r} - c_r = 0$ in $X(\frac{K}{k};\frac{K}{k_2},\ldots, \frac{K}{k_r})$. It follows that the degree is equal to the product of the number $N$ of irreducible components and the number of points $[(\xi^{\frac{K}{k_0}} a_0, \xi^{\frac{K}{k_1}} a_1, b_2,\ldots,b_r)]$ for some $\xi \in \mu_{\frac{K}{k}}$ and fixed $[(b_2,\ldots, b_r)] \in X(\frac{K}{k};\frac{K}{k_2},\ldots, \frac{K}{k_r})$. Working analogously as in the proof of Lemma~\ref{lemma:number-of-solutions}, the latter number is equal to \[\frac{\Big\vert \Big\{\big(\xi^{\frac{K}{k_0}},\xi^{\frac{K}{k_1}}\big) ~\big\vert~ \xi \in \mu_{\frac{K}{k}}\Big\} \Big\vert}{\vert \text{Im}~ h \vert},\] where $h$ is the group homomorphism $h: \mu_{\gcd(\frac{K}{k},\frac{K}{k_2}, \ldots, \frac{K}{k_r})} \longrightarrow \big\{(\xi^{\frac{K}{k_0}},\xi^{\frac{K}{k_1}}) \mid \xi \in \mu_{\frac{K}{k}}\big\}$ given by ${\eta \mapsto (\eta^{\frac{K}{k_0}},\eta^{\frac{K}{k_1}})}$ with kernel $\mu_{\gcd(\frac{K}{k},\frac{K}{k_0}, \ldots, \frac{K}{k_r})}$. Finally, an easy computation gives that \[\big\vert \big\{(\xi^{\frac{K}{k_0}},\xi^{\frac{K}{k_1}}) \mid \xi \in \mu_{\frac{K}{k}}\big\}\big\vert = \frac{K}{k\cdot \gcd(\frac{K}{k},\frac{K}{k_0}, \frac{K}{k_1})},\] and we find the degree stated in the lemma.
\end{proof}

With these two preliminary results, we are now ready to prove Proposition~\ref{prop:euler-char}.

\begin{proof}[Proof of Proposition~\ref{prop:euler-char}]
For $r = 2$, the result follows from Lemma~\ref{lemma:euler-char-plane-curve} in which $M_0 = a_1p_2 - a_2p_1 = 0$. For $r \geq 3$, we work similarly as in the proof of Lemma~\ref{lemma:euler-char-plane-curve}: we will show that the well-defined surjective morphism \[ h: \E \setminus \bigcup_{i=0}^r\{x_i = 0\} \longrightarrow C \setminus \bigcup_{i=0}^2\{x_i = 0\}: [x_0:\ldots:x_r] \mapsto [x_0:x_1:x_2],\] where $C := \{x_0^{m_0} + x_1^{m_1} + x_2^{m_2} = 0\} \subset \P^2_{(p_0,p_1,p_2)}(d;a_0,a_1,a_2)$, is a $D$-sheeted covering with \[D = \frac{m_3\cdots m_r\cdot \gcd\big(dp_2\cdot(p_0,\ldots, p_r),(a_2p_0 - a_0p_2)\cdot(p_1,\ldots, p_r)\big)}{p_2\cdot\gcd\big(d\cdot(p_0,p_1,p_2),a_2p_0 - a_0p_2,a_1p_0 - a_0p_1\big)}.\] Together with Lemma~\ref{lemma:euler-char-plane-curve} applied to $C$ with $M_0 = 0$, we find that $\chi(\E \setminus \bigcup_{i=0}^r\{x_i = 0\})$ is given by
	\begin{equation}\label{eq:euler-char}
	 	- \frac{m_1\cdots m_r \cdot \gcd\big(dp_2\cdot(p_0,\ldots, p_r), (a_2p_0 - a_0p_2)\cdot(p_1,\ldots, p_r)\big)}{dp_0p_2}.
	\end{equation}
This can be rewritten as the formula in the statement. To show that $h$ is a covering map, it is once more enough to show that $h$ is a local homeomorphism. This time, we consider the chart where $x_2\neq 0$: this gives \[h_2: \E' \setminus \bigcup_{i=0,i\neq 2}^r\{x_i = 0\} \longrightarrow C' \setminus \bigcup_{i=0}^1\{x_i = 0\}:[(x_0,x_1,x_3,\ldots, x_r)] \mapsto [(x_0,x_1)],\] where $\E'$ is given by \[\left \{ \begin{array}{clccccl}
		x_0^{m_0} & + & x_1^{m_1}  &+& 1 &=& 0 \\
		&  & 1  &+& x_3^{m_3} &=& 0  \\
		& & & \vdots & & & \\
		& & x_{r-1}^{m_{r-1}}   &+& x_r^{m_r} &=& 0
		\end{array}\right.\]
in the embedding space  \[X\left(\begin{array}{c|ccccc} 
		p_2 & p_0 & p_1 & p_3 & \ldots & p_r \\
		dp_2 & -M_1 & 0 & 0 & \ldots & 0
	\end{array} \right),\] with $M_1 = a_2p_0 - a_0p_2$, and $C'$ by $\{x_0^{m_0} + x_1^{m_1}  + 1 = 0 \}$ in $X \left(\begin{smallmatrix} p_2 \\ dp_2\end{smallmatrix} \middle| \begin{smallmatrix}  p_0 & p_1 \\  -M_1 & 0 \end{smallmatrix} \right)$. Because the embedding spaces of $\E$ and $\E'$ are smooth outside their coordinate hyperplanes, one can conclude by working similarly as in Lemma~\ref{lemma:degree}. To prove the correct formula for the degree of $h$, we again consider the chart where $x_2\neq 0$. The morphism $h_2$ can be further simplified with an isomorphism 
	\[ X\left(\begin{array}{c|ccccc} 
  		p_2 & p_0 & p_1 & p_3 & \ldots & p_r \\
  		dp_2 & -M_1 & 0 & 0 & \ldots & 0
  	\end{array} \right)  \simeq X\left(p_2; \frac{dp_0p_2}{\gcd(dp_2,M_1)}, p_1, p_3, \ldots, p_r\right)\] 
as in~\eqref{eq:isom-to-one-line} under which $\E'$ is transformed into 
	\[\left \{\begin{array}{clccccl}
 		x_0^{\frac{m_0\gcd(dp_2,M_1)}{dp_2}} & + & x_1^{m_1}  &+& 1 &=& 0 \\
 		&  & 1  &+& x_3^{m_3} &=& 0  \\
 		& & & \vdots & & & \\
 		& & x_{r-1}^{m_{r-1}}   &+& x_r^{m_r} &=& 0.
 	\end{array}\right.\]
Using the corresponding isomorphism on the embedding space of $C'$ under which $C'$ is transformed in the same way as $\E'$, we arrive at the situation of Lemma~\ref{lemma:degree} with $K = m_ip_i$ for $i = 0,\ldots, r$ and $N = \frac{m_3\cdots m_r\gcd(p_2,\ldots, p_r)}{p_2}$ (see~\eqref{eq:number-of-comp}), which leads to the degree $D$.
\end{proof}

\begin{cor}
For $k = 1, \ldots, g$, the Euler characteristic of $\check{\E}_k$ is given by \[\chi(\check{\E}_k) = - \frac{n_k\lbeta_k}{\lcm(\frac{\lbeta_k}{e_k}, n_k,\ldots, n_g)}.\] 
\end{cor}

\begin{proof}
For $k = g$, we already know that $\chi(\check{\E}_g) = -1$. Because $\gcd(\lbeta_g,n_g) = e_g = 1$, this is the same as the expression in the statement. For $k = 1$, by construction of the resolution, $\check{\E}_1$ is isomorphic to $\E_1 \setminus \bigcup_{i=0}^g \{x_i = 0\}$ in $\P^g_{w_1}$ after the first blow-up. From~\eqref{eq:euler-char} in the proof of Proposition~\ref{prop:euler-char} applied to the equations~\eqref{eq:E1-homog}, we indeed find that \[\chi(\check{\E_1}) = -\frac{n_1\cdots n_g\gcd(\frac{n}{n_0},\ldots, \frac{n}{n_g})}{\frac{n}{n_0}} = - \frac{n_1\lbeta_1}{\lcm(\frac{\lbeta_1}{e_1}, n_1,\ldots, n_g)},\] where we used that $n = n_1\lbeta_1$ and the relation~\eqref{eq:rel-gcd-lcm}. If $g \geq 3$ and $k \in {\{2,\ldots, g-1\}}$, the Euler characteristic of $\check{\E}_k$ can be computed from~\eqref{eq:Ek-homog} in the same way. 
\end{proof}

We are finally ready to compute the zeta function of monodromy associated with a space monomial curve $Y \subset \C^{g+1}$.
 
\begin{theorem}\label{thm:zeta-function-mon-Y}
Let $Y \subset \C^{g+1}$ be a space monomial curve defined by the equations~\eqref{eq:equations-Y} with $g\geq 2$. Consider a generic embedding surface $S = S(\lambda_2,\ldots,\lambda_g) \subset \C^{g+1}$ given by~\eqref{eq:equations-S}, where $(\lambda_2,\ldots,\lambda_g)$ are chosen such that Section~\ref{RedCurveSurface} applies. Denote by $\sigma: X' \rightarrow \C^{g+1}$ the blow-up of $\C^{g+1}$ with center $Y$ and by $S'$ the strict transform of $S$ under $\sigma$. Then, the monodromy zeta function of $Y$ considered in $\C^{g+1}$ at the generic point $p = S' \cap \sigma^{-1}(0)$ is given by \[Z^{mon}_{Y,p}(t) = \frac{\prod\limits_{k = 0}^g(1-t^{M_k})^{\frac{\lbeta_k}{M_k}}}{\prod\limits_{k = 1}^g(1-t^{N_k})^{\frac{n_k\lbeta_k}{N_k}}},\] where $M_k := \lcm(\frac{\lbeta_k}{e_k},n_{k+1},\ldots, n_g)$ for $k = 0,\ldots, g$, and $N_k := \lcm(\frac{\lbeta_k}{e_k},n_k,\ldots, n_g)$ for $k = 1,\ldots, g$.
\end{theorem}

\begin{proof}
This immediately follows from all the results in this section: the strata $Q_k$ for $k = 0,\ldots, g$ yield the factors in the numerator, and the `strata' $\check{\E}_k$ for $k = 1,\ldots, g$ yield the factors in the denominator.
\end{proof}

We illustrate this theorem with two examples, in which we already see that every pole of the motivic Igusa zeta function induces an eigenvalue of monodromy. In the next section, we will prove this in general.

\begin{ex}\label{ex:zeta-function-mon} \
\begin{enumerate}[wide, labelindent=0pt] 
 	\item[(i)] The irreducible plane curve given by $(x_1^2-x_0^3)^2-x_0^5x_1 = 0$ has $(4,6,13)$ as minimal generating set of its semigroup, and leads to the space monomial curve $Y_1 \subseteq \C^3$ defined in three variables $(g=2$) by 
 		\[\left\{\begin{array}{r c l l}
 			x_1^2 & - & x_0^3  &  = 0 \\
 			x_2^2  &- & x_0^5x_1 &= 0.  \\
  		\end{array}\right.\]
  	The expression for the monodromy zeta function in Theorem~\ref{thm:zeta-function-mon-Y} gives \[Z^{mon}_{Y_1,p_1}(t) = \frac{(1-t^2)^2(1-t^6)(1-t^{13})}{(1-t^6)^2(1-t^{26})} = \frac{(1-t^2)^2(1-t^{13})}{(1-t^6)(1-t^{26})}.\] In~\cite[Example 4.1]{MVV}, it was shown that the motivic Igusa zeta function of $Y_1$ has three poles: $\L^2,\L^{\frac{8}{6}}$ and $\L^{\frac{37}{26}}$. Every pole $\L^{-s_0}$ of these three induces a monodromy eigenvalue $e^{2\pi i s_0}$: $e^{-4\pi i}$ is a zero of $Z^{mon}_{Y_1,p_1}(t)$, while $e^{\frac{-8 \pi i}{3}}$ and $e^{\frac{-37\pi i}{13}}$ are poles of $Z^{mon}_{Y_1,p_1}(t)$.
  	\item[(ii)] Consider the space monomial curve $Y_2\subseteq \C^4$ associated with the plane curve defined by $((x_1^2-x_0^3)^2 - x_0^5x_1)^2 - x_0^{10}(x_1^2-x_0^3) = 0$, whose semigroup is minimally generated by $(8,12,26,53)$. Its equations are given by
 		\[\left\{ \begin{array}{r c l l}
 			x_1^2 & - & x_0^3  &  = 0 \\
 			x_2^2  &- & x_0^5x_1 &= 0 \\
 			x_3^2 & - &x_0^{10}x_2 &=0.
  		\end{array}\right.\]  
 	Using Theorem~\ref{thm:zeta-function-mon-Y}, we find \[Z^{mon}_{Y_2,p_2}(T) = \frac{(1-t^2)^4(1-t^6)^2(1-t^{26})(1-t^{53})}{(1-t^6)^4(1-t^{26})^2(1-t^{106})} = \frac{(1-t^2)^4(1-t^{53})}{(1-t^6)^2(1-t^{26})(1-t^{106})}.\] The poles of the motivic zeta function of $Y_2$ were also computed in~\cite[Example 4.1]{MVV}: they are given by $\L^{3}, \L^{\frac{11}{6}}, \L^{\frac{50}{26}},$ and $\L^{\frac{235}{106}}$. Similarly as in the previous example, it is easy to see that they all induce eigenvalues of monodromy associated with $Y_2$.
\end{enumerate}
\end{ex}


\section{The monodromy conjecture for a space monomial curve} \label{MonConj}

This last section consists of a proof of the main result in this article, namely the monodromy conjecture for a space monomial curve $Y \subset \C^{g+1}$ with $g \geq 2$. In other words, we will show that every pole $\L^{-s_0}$ of the motivic Igusa zeta function associated with $Y$ yields a monodromy eigenvalue $e^{2\pi is_0}$ of $Y$.

\vspace{14pt}

In~\cite{MVV}, it was shown that a complete list of poles of both the local and global motivic Igusa zeta function of a space monomial curve $Y \subset \C^{g+1}$ is given by \[\L^g, \qquad \L^{\frac{\nu_k}{N_k}}, \qquad k = 1,\ldots, g,\] where
 	\begin{equation}\label{eq:poles-motivic}
 		\frac{\nu_k}{N_k} = \frac{1}{n_k\lbeta_k}\bigg(\sum_{l=0}^k \lbeta_l - \sum_{l=1}^{k-1}n_l\lbeta_l\bigg) + (k-1) + \sum_{l=k+1}^g\frac{1}{n_l},  
 	\end{equation}
and $N_k = \lcm(\frac{\lbeta_k}{e_k},n_k,\ldots, n_g)$. Clearly, the first pole, $\L^g$, and the poles $\L^{\frac{\nu_k}{N_k}}$ with $\frac{\nu_k}{N_k} \in \N$ induce the trivial monodromy eigenvalue $1$. We claim that for every $k= 1,\ldots,g$ with $\frac{\nu_k}{N_k} \notin \N$, the candidate monodromy eigenvalue $e^{-2\pi i\frac{\nu_k}{N_k}}$ is a pole of the monodromy zeta function of $Y$ computed in the previous section.

\begin{remark}
It is possible that $\frac{\nu_k}{N_k}$ is an integer for some $k \in \{1,\ldots, g\}$; for example, the space monomial curve $Y \subset \C^3$ defined by 
 	\[\left\{\begin{array}{r c l l}
 		x_1^2 & - & x_0^3  &  = 0 \\
 		x_2^6  &- & x_0^{17}x_1 &= 0  \\
  	\end{array}\right.\]
corresponds to the generators $(12,18,37)$ with $\frac{\nu_1}{N_1} = 1$.
\end{remark}

To prove this claim, we will not work directly with the monodromy zeta function of $Y$ at the point $p = S' \cap \sigma^{-1}(0)$, but we will again consider $Y$ as the Cartier divisor $\{f_1 = 0\}$ on a generic surface $S$. All the interesting information is contained in the characteristic polynomial $\Delta(t)$ of the monodromy transformation on $\mathcal H^1(\psi_{f_1}\C^{\boldsymbol{\cdot}})_0$. From Theorem~\ref{thm:zeta-function-mon-Y}, it follows that \[\Delta(t) = \frac{(t-1)\prod\limits_{k = 1}^g(t^{N_k}-1)^{\frac{n_k\lbeta_k}{N_k}}}{\prod\limits_{k = 0}^g(t^{M_k}-1)^{\frac{\lbeta_k}{M_k}}}\] is a polynomial of degree $\mu = 1 + \sum_{k=1}^g(n_k-1)\lbeta_k - \lbeta_0 > 0$. Hence, if we show that the candidate monodromy eigenvalue $e^{-2\pi i\frac{\nu_k}{N_k}} \neq 1$ is a zero of $\Delta(t)$, then it will be an eigenvalue of monodromy associated with $Y \subset \C^{g+1}$ at the generic point $p = S' \cap \sigma^{-1}(0)$.

\begin{theorem}\label{thm:mon-conj}
Let $Y \subset \C^{g+1}$ be a space monomial curve defined by the equations~\eqref{eq:equations-Y} with $g \geq 2$, and denote by $\sigma: X' \rightarrow \C^{g+1}$ the blow-up of $\C^{g+1}$ with center $Y$. Every pole $\L^{-s_0}$ of the local or global motivic Igusa zeta function associated with $Y$ induces a monodromy eigenvalue $e^{2\pi i s_0}$ of $Y$ at a point in $\sigma^{-1}(B \cap Y)$ for $B$ a small ball around $0$.
\end{theorem}

\begin{proof}
It remains to show that every $\lambda_k := e^{-2\pi i\frac{\nu_k}{N_k}}$ for $k = 1,\ldots, g$ with $\frac{\nu_k}{N_k} \notin \N$ is a zero of the characteristic polynomial. To this end, we will write $\Delta(t)$ as the product of $g$ polynomials of which each has one of the elements $\lambda_k$ as a zero. More precisely, we will write $\Delta(t)$ as a product of polynomials of the form \[\frac{(t^a-1)^p\cdot(t^{\gcd(b,c)} - 1)^{\gcd(q,r)}}{(t^b-1)^q \cdot (t^c-1)^r},\] where $a,b,c,p,q$ and $r$ are positive integers with $b,c \mid a$  and $q,r \mid p$. For this purpose, let $L_k := \lcm(n_k,\ldots, n_g)$ for $k = 1,\ldots, g$ and let $L_{g+1} := 1$. With the definitions $N_k = \lcm(\frac{\lbeta_k}{e_k},n_k,\ldots, n_g)$ and $M_k = \lcm(\frac{\lbeta_k}{e_k},n_{k+1},\ldots, n_g)$, it is easy to see that $M_k, L_k \mid N_k$ and that $\frac{\lbeta_k}{M_k}, \frac{e_{k-1}}{L_k} \mid \frac{n_k\lbeta_k}{N_k}$ for all $k = 1,\ldots, g$. Furthermore, we have for all $k = 1,\ldots, g$ that \[\gcd(M_k,L_k) =\lcm\Big(L_{k+1},\gcd\big(\frac{\lbeta_k}{e_k},n_k\big)\Big) = L_{k+1},\] where we used in the first equality the general property $\gcd(\lcm(\alpha,\gamma),\lcm(\alpha,\delta)) = \lcm(\alpha,\gcd(\gamma,\delta))$, and in the second equality the fact that $\gcd(\frac{\lbeta_k}{e_k},n_k)= 1$, see Section~\ref{SpaceMonomial}. Finally, using the relation~\eqref{eq:rel-gcd-lcm} and $\gcd(\lbeta_k,e_{k-1}) = e_k$, we see for $k = 1,\ldots, g$ that \[\gcd\Big(\frac{\lbeta_k}{M_k},\frac{e_{k-1}}{L_k}\Big) = \gcd\Big(e_k,\frac{\lbeta_k}{n_{k+1}},\ldots, \frac{\lbeta_k}{n_g},\frac{e_{k-1}}{n_k},\ldots, \frac{e_{k-1}}{n_g}\Big) = \gcd\Big(\frac{e_k}{n_{k+1}},\ldots, \frac{e_k}{n_g}\Big) = \frac{e_k}{L_{k+1}}.\] All this together implies for each $k = 1,\ldots, g$  that \[P_k(t) := \frac{(t^{N_k} - 1)^{\frac{n_k\lbeta_k}{N_k}} \cdot (t^{L_{k+1}} - 1)^{\frac{e_k}{L_{k+1}}}}{(t^{M_k} - 1)^{\frac{\lbeta_k}{M_k}} \cdot (t^{L_k} - 1)^{\frac{e_{k-1}}{L_k}}}\] is a polynomial of the above form. It is also easy to see that $\Delta(t) = \prod_{k=1}^g P_k(t)$.

\vspace{14pt}

Fix now some $k \in \{1,\ldots, g\}$. We prove that $\lambda_k = e^{-2\pi i\frac{\nu_k}{N_k}}$ is a zero of $P_k(t)$. Clearly, it is a zero of $t^{N_k} - 1$, but we still need to show that this candidate zero does not get canceled with the denominator. To show this, we distinguish the following four cases.  
\begin{enumerate}[wide, labelindent=0pt]
	\item[(i)] The candidate zero $\lambda_k$ is not a zero of $t^{M_k}-1 = 0$, nor of $t^{L_k} - 1 = 0$: trivially, the candidate zero $\lambda_k$ is not canceled in $P_k(t)$.
	\item[(ii)] The candidate zero $\lambda_k$ is a zero of $t^{M_k}-1 = 0$, but not of $t^{L_k} - 1 = 0$: in this case, it is sufficient to prove that $\frac{n_k\lbeta_k}{N_k} > \frac{\lbeta_k}{M_k}$ in order to conclude that $\lambda_k$ is a zero of $P_k(t)$. Because $\lambda_k = e^{-2\pi i\frac{\nu_k}{N_k}}$ is a zero of $t^{M_k}-1 = 0$, we know that $\frac{\nu_kM_k}{N_k}$ is an integer. Using the expression~\eqref{eq:poles-motivic} for $\frac{\nu_k}{N_k}$, one can see that this implies that $n_k\lbeta_k \mid (\sum_{l=0}^k \lbeta_l - \sum_{l=1}^{k-1}n_l\lbeta_l)M_k$, which in turn implies, using $n_k\lbeta_k = e_{k-1}\frac{\lbeta_k}{e_k} \mid \lbeta_lM_k$ for $l = 0,\ldots, k-1$, that $n_k \mid M_k$. We can conclude that $N_k = M_k$, and, hence, we indeed have that $\frac{n_k\lbeta_k}{N_k} > \frac{\lbeta_k}{M_k}$ as $n_k > 1$.
	\item[(iii)] The candidate zero $\lambda_k$ is a zero of $t^{L_k} - 1 = 0$, but not of $t^{M_k}-1 = 0$: as in the previous case, it is enough to show that $\frac{n_k\lbeta_k}{N_k} > \frac{e_{k-1}}{L_k}$. From $\lambda_k$ being a zero of $t^{L_k} - 1 = 0$, one can now deduce that $\frac{\lbeta_k}{e_k} \mid (\sum_{l=0}^{k-1} \lbeta_l - \sum_{l=1}^{k-1}n_l\lbeta_l)\frac{L_k}{e_{k-1}}$. Because $e_{k-1} \mid \lbeta_l$ for $l = 0,\ldots, k-1$, it follows that $N_k = \lcm(\frac{\lbeta_k}{e_k},L_k) \mid (\sum_{l=0}^{k-1} \lbeta_l - \sum_{l=1}^{k-1}n_l\lbeta_l)\frac{L_k}{e_{k-1}}$, and, thus, that \[\frac{1}{N_k} \geq \frac{1}{\left\vert \sum_{l=0}^{k-1} \lbeta_l - \sum_{l=1}^{k-1}n_l\lbeta_l \right\vert} \frac{e_{k-1}}{L_k} = \left \{\begin{array}{cl}
			\frac{1}{L_1} & \text{for } k=1  \\
			\frac{1}{-\lbeta_0 + \sum_{l=1}^{k-1}(n_l-1)\lbeta_l}\frac{e_{k-1}}{L_k}   & \text{for } k = 2,\ldots, g.
		\end{array}\right. \]
 	The equality comes from the fact that $-\lbeta_0 + (n_1 - 1)\lbeta_1 = n_1\lbeta_1(1 - \frac{1}{n_0} - \frac{1}{n_1}) > 0$ since $n_0, n_1 \geq 2$ are coprime. We can finish this case by using that $\lbeta_1 >  \lbeta_0=e_0$ and $\lbeta_k > -\lbeta_0 + \sum_{l=1}^{k-1}(n_l-1)\lbeta_l$ for $k = 2,\ldots, g$, which follows from $\lbeta_i > n_{i-1}\lbeta_{i-1}$ for $i = 2,\ldots, k$.
	\item[(iv)] The candidate zero $\lambda_k$ is a zero of both $t^{L_k} - 1 = 0$ and $t^{M_k}-1 = 0$: in this last case, the candidate zero $\lambda_k$ is also a zero of $t^{L_{k+1}} - 1 = 0$ and we need to show that $\frac{n_k\lbeta_k}{N_k} + \frac{e_k}{L_{k+1}} - \frac{\lbeta_k}{M_k} - \frac{e_{k-1}}{L_k} > 0$. Combining case (ii) and (iii), we know that \[\frac{n_k\lbeta_k}{N_k} + \frac{e_k}{L_{k+1}} - \frac{\lbeta_k}{M_k} - \frac{e_{k-1}}{L_k} \geq \frac{(n_k-1)\lbeta_k}{\left\vert \sum_{l=0}^{k-1} \lbeta_l - \sum_{l=1}^{k-1}n_l\lbeta_l \right\vert}\frac{e_{k-1}}{L_k}  + \frac{e_k}{L_{k+1}} - \frac{e_{k-1}}{L_k}, \] which is positive as one can, similarly as in case (iii), see that $(n_k-1)\lbeta_k \geq \lbeta_k > \left\vert \sum_{l=0}^{k-1} \lbeta_l - \sum_{l=1}^{k-1}n_l\lbeta_l \right\vert$ for $k = 1,\ldots,g$.
\end{enumerate}
Hence, every $\lambda_k$ is a zero of $P_k(t)$, and consequently, an eigenvalue of monodromy.
\end{proof}

\begin{remark}
In the proof of Theorem~\ref{thm:mon-conj}, the pole $\lambda_g =  e^{-2\pi i\frac{\nu_g}{N_g}}$ could have been treated way easier. More precisely, the candidate zero $\lambda_g$ is never a zero of the denominator of $P_g(t)$, and we are always in case (i). Indeed, in case (ii), we would have that $n_g \mid M_g = \lbeta_g$, which is impossible. Likewise, in case (iii), we would have the impossible property $\lbeta_g \mid \sum_{l=0}^{g-1} \lbeta_l - \sum_{l=1}^{g-1}n_l\lbeta_l$ because $\lbeta_g > \vert\sum_{l=0}^{g-1} \lbeta_l - \sum_{l=1}^{g-1}n_l\lbeta_l\vert = -\lbeta_0 + \sum_{l=1}^{g-1}(n_l-1)\lbeta_l$. For smaller $k$, however, it is possible that $\lambda_k$ is a zero of the denominator. For instance, we can consider the curve $Y_1$ from Example~\ref{ex:zeta-function-mon} whose characteristic polynomial $\Delta(t)$ is written as the product $P_1(t) \cdot P_2(t)$ where \[P_1(t) = \frac{(t^6-1)^2(t^2-1)}{(t^6-1)(t^2-1)^2}, \qquad P_2(t) = \frac{(t^{26}-1)(t-1)}{(t^{13}-1)(t^2-1)}.\] For $\lambda_1 = e^{\frac{-8\pi i}{3}}$, we are in case (ii): it is a zero of the first term of the denominator of $P_1(t)$, but not of the second. One can also find examples in which some candidate zero $\lambda_k$ for $k < g$ is in case (iii) or (iv).
\end{remark}

One can also investigate the monodromy conjecture for the related \emph{topological} and \emph{$p$-adic Igusa zeta function}, which are specializations of the motivic Igusa zeta function. See for instance~\cite{DL1} and~\cite{Den1}, respectively, for their expressions in terms of an embedded resolution for one polynomial, and~\cite{VZ} for their generalizations to ideals. Since the monodromy conjecture for the motivic zeta function implies the conjecture for the other two zeta functions, we have simultaneously shown all these monodromy conjectures for our monomial curves. 


\end{document}